\documentclass[reqno,10pt,a4paper]{amsart}

\usepackage{amssymb,eucal,mathrsfs,setspace}
\usepackage[noadjust]{cite}
\usepackage[dvipsnames]{xcolor}

\usepackage{array}
\newcolumntype{C}{>{$}c<{$}} 

\usepackage[largesc,theoremfont]{newtxtext}      
\usepackage[libertine,cmbraces,varbb]{newtxmath} 

\begingroup
  \makeatletter
  \@for\theoremstyle:=definition,remark,plain\do{%
    \expandafter\g@addto@macro\csname th@\theoremstyle\endcsname{%
      \addtolength\thm@preskip{.5\baselineskip plus .2\baselineskip minus .2\baselineskip}
      \addtolength\thm@postskip{.5\baselineskip plus .2\baselineskip minus .2\baselineskip}
    }%
  }
\endgroup

\usepackage[margin=24mm]{geometry}

\usepackage{enumitem}
\setitemize{leftmargin=*}   
\setenumerate{leftmargin=*, 
  label=\textup{(\roman*)}} 

\usepackage{mathtools} 

\numberwithin{equation}{section}
\allowdisplaybreaks 

\usepackage{graphicx}

\usepackage{tikz}
\usetikzlibrary{decorations.markings,positioning,arrows}
\tikzset{%
	owt/.style={circle, fill=black, inner sep=2pt, outer sep=0pt, minimum size=5pt}, 
	wt/.style={circle, draw=black, fill=white, inner sep=2pt, outer sep=0pt, minimum size=5pt}, 
	nom/.style={rectangle,draw=black!20,fill=red!10,inner sep=2pt}, 
	axiscolour/.style={blue}, 
	fillcolour/.style={fill=blue,fill opacity=0.3}, 
	drawcolour/.style={blue,opacity=0.3,line width=0.2cm,text opacity=1} 
}

\newcommand{\makeaxes}[2]{
	\draw[axiscolour] (-90:1) -- (90:1) #2;
	\draw[axiscolour] (-150:1) -- (30:1) #1;
	\draw[axiscolour] (-30:1) -- (150:1);
}

\newcommand{\thintriangle}[5]{
	\begin{scope}[rotate around={#2:(0,0)},shift={(0,-#1)}]
		\draw[thick] (-120:1) -- (0:0) node[owt,label=#3] {} -- (-60:1);
		\node at (0,-#4) {#5};
		\begin{scope}[shift={(0,-0.1)}]
			\fill[fillcolour] (-120:1) -- (0:0) -- (-60:1) -- cycle;
		\end{scope}
	\end{scope}
}

\newcommand{\thicktriangle}[5]{
	\begin{scope}[rotate around={#2:(0,0)},shift={(0,-#1)}]
		\draw[thick] (-150:1) -- (0:0) node[owt,label=#3] {} -- (-30:1);
		\node at (0,-#4) {#5};
		\begin{scope}[shift={(0,-0.06)}]
			\fill[fillcolour] (-150:1) -- (0:0) -- (-30:1) -- cycle;
		\end{scope}
	\end{scope}
}

\newcommand{\brick}[4]{
	\begin{scope}[rotate around={#2:(0,0)},shift={(0,-#1)}]
		\draw[thick] (180:1) -- node[owt,label=#3,pos=#4] {} (0:1);
		\begin{scope}[shift={(0,-0.06)}]
			\fill[fillcolour] (-1,0) rectangle (1,-0.75);
		\end{scope}
	\end{scope}
}

\usepackage[colorlinks=true,citecolor=red,linkcolor=blue]{hyperref} 

\usepackage[capitalise,noabbrev]{cleveref} 

\newcommand{\pd}{\partial}     

\renewcommand{\ge}{\geqslant} 
\renewcommand{\le}{\leqslant} 

\newcommand{\blank}{{-}}

\DeclareMathOperator{\vspn}{span}

\DeclareMathOperator{\Ind}{Ind}

\newcommand{\cc}{\mathsf{c}}   
\newcommand{\dd}{\mathrm{d}}   
\newcommand{\ee}{\mathsf{e}}   
\newcommand{\ii}{\mathfrak{i}} 
\newcommand{\kk}{\mathsf{k}}   
\newcommand{\qq}{\mathsf{q}}   
\newcommand{\uu}{\mathsf{u}}   
\providecommand{\vv}{\mathsf{v}}\renewcommand{\vv}{\mathsf{v}}   
\newcommand{\yy}{\mathsf{y}}   
\newcommand{\zz}{\mathsf{z}}   

\newcommand{\fracuv}{\frac{\uu}{\vv}} 

\newcommand{\wun}{\vvmathbb{1}}  

\DeclarePairedDelimiter{\brac}{\lparen}{\rparen}   
\DeclarePairedDelimiter{\sqbrac}{\lbrack}{\rbrack} 
\DeclarePairedDelimiter{\set}{\lbrace}{\rbrace}
\newcommand{\st}{\mspace{5mu} {:} \mspace{5mu}}    
\DeclarePairedDelimiter{\abs}{\lvert}{\rvert}
\DeclarePairedDelimiter{\norm}{\lVert}{\rVert}
\DeclarePairedDelimiter{\ang}{\langle}{\rangle}
\DeclarePairedDelimiter{\no}{{\mspace{2mu}} {:} {\mspace{2mu}}}{{\mspace{2mu}} {:} {\mspace{2mu}}} 

\DeclarePairedDelimiterX{\comm}[2]{\lbrack}{\rbrack}{#1 , #2}  
\DeclarePairedDelimiterX{\acomm}[2]{\lbrace}{\rbrace}{#1 , #2} 
\DeclarePairedDelimiterX{\inner}[2]{\langle}{\rangle}{#1 , #2} 
\DeclarePairedDelimiterX{\super}[2]{\lparen}{\rparen}{#1 \delimsize\vert \mathopen{} #2} 

\newcommand{\ira}{\hookrightarrow}    

\newcommand{\fld}[1]{\mathbb{#1}}    
\newcommand{\alg}[1]{\mathfrak{#1}}  
\newcommand{\grp}[1]{\mathsf{#1}}    
\newcommand{\Mod}[1]{\mathcal{#1}}   
\newcommand{\VOA}[1]{\mathsf{#1}}    
\newcommand{\categ}[1]{\mathscr{#1}} 

\newcommand{\ZZ}{\fld{Z}}
\newcommand{\NN}{\ZZ_{\ge 0}} 

\newcommand{\RR}{\fld{R}}
\newcommand{\CC}{\fld{C}}
\newcommand{\HH}{\fld{H}} 

\newcommand{\SLG}[2]{\grp{#1}_{#2}}       
\newcommand{\mgrp}{\grp{SL}\brac*{2;\ZZ}}

\newcommand{\finite}[1]{#1}
\newcommand{\affine}[1]{\widehat{#1}}
\newcommand{\SLA}[2]{\finite{\alg{#1}}_{#2}}             
\newcommand{\AKMA}[2]{\affine{\alg{#1}}_{#2}}            
\newcommand{\envalg}[1]{\mathsf{U}\brac{#1}}             
\newcommand{\ideal}[1]{\ang[\big]{#1}}

\newcommand{\glone}{\SLA{gl}{1}}
\newcommand{\sltwo}{\SLA{sl}{2}}
\newcommand{\gltwo}{\SLA{gl}{2}}
\newcommand{\slthree}{\SLA{sl}{3}}
\newcommand{\aslthree}{\AKMA{sl}{3}}

\newcommand{\aslpos}{\affine{\alg{sl}}{}_3^>}            
\newcommand{\aslzero}{\affine{\alg{sl}}{}_3^0}
\newcommand{\aslneg}{\affine{\alg{sl}}{}_3^<}

\newcommand{\killingsymb}{\kappa}
\newcommand{\killing}[2]{\killingsymb \brac{#1 , #2}}                
\newcommand{\bilin}[2]{\brac{#1 , #2}}                               
\newcommand{\pair}[2]{\ang{#1 , #2}}                                 
\newcommand{\tcartan}{\begin{psmallmatrix} 2 & -1 \\ -1 & 2 \end{psmallmatrix}}

\newcommand{\csub}{\alg{h}}                                          
\newcommand{\dcsub}{\alg{h}^*}                                       
\newcommand{\rdcsub}{\alg{h}^*_{\RR}}
\newcommand{\borel}{\alg{b}}                                         
\newcommand{\para}{\alg{p}}                                          
\newcommand{\nilrad}{\alg{u}}                                        
\newcommand{\levi}{\alg{l}}                                          

\newcommand{\roots}{\Delta}                                          
\newcommand{\sroot}[1]{\alpha_{#1}}                                  
\newcommand{\fwt}[1]{\Lambda_{#1}}                                   
\newcommand{\fcwt}[1]{\xi_{#1}}                                      
\newcommand{\wvec}{\rho}                                             
\newcommand{\weyl}{\grp{W}}                                          
\newcommand{\ccent}[1]{\envalg{#1}^{\csub}}                          

\newcommand{\co}[1]{#1^{\vee}}                                       
\newcommand{\wlat}{\grp{P}}                                          
\newcommand{\rlat}{\grp{Q}}                                          
\newcommand{\cwlat}{\co{\grp{P}}}                                    
\newcommand{\crlat}{\co{\grp{Q}}}                                    
\newcommand{\pwlat}{\wlat_{\ge}}

\newcommand{\nilorb}[1]{\mathbb{O}_{\text{#1}}}                      

\newcommand{\sing}[1]{\operatorname{sing}(#1)}                       

\newcommand{\apwlat}[1]{\affine{\wlat}_{\ge}^{#1}}                   
\newcommand{\afwt}[1]{\affine{\Lambda}_{#1}}                         

\newcommand{\admwts}[2]{A_{#1,#2}}                                   
\newcommand{\admuv}{\admwts{\uu}{\vv}}
\newcommand{\theawts}{\admwts{3}{2}}
\newcommand{\sadmwts}[2]{B_{#1,#2}}                                  
\newcommand{\sadmuv}{\sadmwts{\uu}{\vv}}
\newcommand{\theswts}{\sadmwts{3}{2}}
\newcommand{\radmwts}[2]{C_{#1,#2}}                                  
\newcommand{\radmuv}{\radmwts{\uu}{\vv}}
\newcommand{\therwts}{\radmwts{3}{2}}

\newcommand{\func}[2]{#1(#2)}                        
\newcommand{\autfont}[1]{\mathrm{#1}}

\newcommand{\weylaut}[1]{\autfont{w}_{#1}}           
\newcommand{\dynkaut}{\autfont{d}}                   
\newcommand{\conjaut}{\autfont{c}}                   
\newcommand{\sfaut}[1]{\sigma^{#1}}                  

\newcommand{\weylmod}[2]{\func{\weylaut{#1}}{#2}}    
\newcommand{\dynkmod}[1]{\func{\dynkaut}{#1}}        
\newcommand{\conjmod}[1]{\func{\conjaut}{#1}}        
\newcommand{\sfmod}[2]{\func{\sfaut{#1}}{#2}}        

\newcommand{\wcat}{\categ{W}}                        
\newcommand{\awcat}[1]{\affine{\categ{W}}_{#1}}      

\newcommand{\awcatk}{\affine{\categ{W}}_{\kk}}       
\newcommand{\awcatuv}{\affine{\categ{W}}_{\uu,\vv}}  

\newcommand{\bpminmod}[2]{\VOA{BP}\brac{#1, #2}}                
\newcommand{\slnminmod}[3]{\VOA{A}_{#1} \brac{#2, #3}}          
\newcommand{\minmod}[2]{\slnminmod{2}{#1}{#2}}                  
\newcommand{\minmoduv}{\minmod{\uu}{\vv}}
\newcommand{\themm}{\minmod{3}{2}}

\newcommand{\slnuniv}[2]{\VOA{V}_{#2}(\SLA{sl}{#1})}            
\newcommand{\slnsimp}[2]{\VOA{L}_{#2}(\SLA{sl}{#1})}            
\newcommand{\univ}[1]{\slnuniv{3}{#1}}                          
\newcommand{\univk}{\univ{\kk}}
\newcommand{\simp}[1]{\slnsimp{3}{#1}}                          
\newcommand{\simpk}{\simp{\kk}}

\newcommand{\IrrSymb}{\Mod{L}}
\newcommand{\SRelSymb}{\Mod{S}}
\newcommand{\RelSymb}{\Mod{R}}

\newcommand{\firr}[1]{\IrrSymb_{#1}}              
\newcommand{\fsemi}[2]{\SRelSymb^{#1}_{#2}}       
\newcommand{\frel}[2]{\RelSymb^{#1}_{#2}}         
\newcommand{\coh}[1]{\Mod{C}^{#1}}                  

\newcommand{\airr}[1]{\affine{\IrrSymb}_{#1}}         
\newcommand{\asemi}[2]{\affine{\SRelSymb}^{#1}_{#2}}  
\newcommand{\arel}[2]{\affine{\RelSymb}^{#1}_{#2}}    

\DeclareMathOperator{\tr}{tr}

\newcommand{\Gr}[1]{\sqbrac[\big]{#1}}                 

\newcommand{\traceover}[1]{\tr_{\raisebox{-2pt}{$\scriptstyle #1$}}} 

\DeclareMathOperator{\chmap}{ch}
\newcommand{\modcoords}[3]{\left( #1 \middle\bracevert #2 \middle\bracevert #3 \right)}

\newcommand{\ch}[1]{\chmap \Gr{#1}}                    
\newcommand{\fch}[4]{\ch{#1} \brac[\big]{#2;#3;#4}}    

\newcommand{\mch}[4]{\ch{#1} \modcoords{#2}{#3}{#4}}   

\newcommand{\modfont}[1]{\mathsf{#1}}
\newcommand{\mods}{\modfont{S}}                        
\newcommand{\modt}{\modfont{T}}                        

\newcommand{\smat}[2]{\mods^{#1}_{#2}}                         
\newcommand{\asmat}[2]{\overline{\mods}{}^{#1}_{#2}}           
\newcommand{\aasmat}[2]{\dbloverline{\mods}{}^{#1}_{#2}}       

\makeatletter
\newcommand{\dbloverline}[1]{\overline{\dbl@overline{#1}}}
\newcommand{\dbl@overline}[1]{\mathpalette\dbl@@overline{#1}}
\newcommand{\dbl@@overline}[2]{%
  \begingroup
  \sbox\z@{$\m@th#1\overline{#2}$}%
  \ht\z@=\dimexpr\ht\z@-2\dbl@adjust{#1}\relax
  \box\z@
  \ifx#1\scriptstyle\kern-\scriptspace\else
  \ifx#1\scriptscriptstyle\kern-\scriptspace\fi\fi
  \endgroup
}
\newcommand{\dbl@adjust}[1]{%
  \fontdimen8
  \ifx#1\displaystyle\textfont\else
  \ifx#1\textstyle\textfont\else
  \ifx#1\scriptstyle\scriptfont\else
  \scriptscriptfont\fi\fi\fi 3
}
\makeatother

\newcommand{\fuse}{\mathbin{\times}}                              
\newcommand{\Grfuse}{\mathbin{\boxtimes}}                         
\newcommand{\fuscoeff}[3]{\genfrac{\langle}{\rangle}{0pt}{0}{#3}{#1 \quad #2}}  

\newcommand{\cft}{conformal field theory}
\newcommand{\cfts}{conformal field theories}

\newcommand{\uea}{universal enveloping algebra}

\newcommand{\lw}{lowest-weight}
\newcommand{\lwv}{\lw\ vector}

\newcommand{\hw}{highest-weight}
\newcommand{\hwv}{\hw\ vector}
\newcommand{\hwvs}{\hwv s}
\newcommand{\hwm}{\hw\ module}
\newcommand{\hwms}{\hwm s}
\newcommand{\rhw}{relaxed highest-weight}

\newcommand{\rhwm}{\rhw\ module}
\newcommand{\rhwms}{\rhwm s}
\newcommand{\srhw}{semi\rhw}

\newcommand{\srhwm}{\srhw\ module}
\newcommand{\srhwms}{\srhwm s}

\newcommand{\vo}{vertex operator}
\newcommand{\voa}{\vo\ algebra}
\newcommand{\voas}{\voa s}

\newcommand{\ope}{operator product expansion}

\newcommand{\qhr}{quantum hamiltonian reduction} 

\newcommand{\lhs}{left-hand side}

\newcommand{\rhs}{right-hand side}

\newcommand{\fdim}{finite-dimensional}
\newcommand{\infdim}{in\fdim}

\newcommand{\bp}{Bershadsky--Polyakov}
\newcommand{\ep}{Euler--Poincar\'{e}}
\newcommand{\kl}{Kazhdan--Lusztig}
\newcommand{\km}{Kac--Moody}

\newcommand{\wzw}{Wess--Zumino--Witten}

\theoremstyle{plain}
\newtheorem{theorem}{Theorem}[section]
\newtheorem{corollary}[theorem]{Corollary}
\newtheorem{lemma}[theorem]{Lemma}
\newtheorem{proposition}[theorem]{Proposition}

\makeatletter
\renewcommand\author@andify{%
  \nxandlist {\unskip ,\penalty-1 \space\ignorespaces}%
    {\unskip {} \@@and~}%
    {\unskip \penalty-2 \space \@@and~}%
}
\makeatother

\begin{document}

\title{Admissible-level $\mathfrak{sl}_3$ minimal models}

\author[]{Kazuya Kawasetsu}
\address[Kazuya Kawasetsu]{
Priority Organization for Innovation and Excellence \\
Kumamoto University \\
Kumamoto 860-8555, Japan.
}
\email{kawasetsu@kumamoto-u.ac.jp}

\author[D~Ridout]{David Ridout}
\address[David Ridout]{
School of Mathematics and Statistics \\
University of Melbourne \\
Parkville, Australia, 3010.
}
\email{david.ridout@unimelb.edu.au}

\author[S~Wood]{Simon Wood}
\address[Simon Wood]{
School of Mathematics\\
Cardiff University \\
Cardiff, United Kingdom, CF24 4AG.
}
\email{woodsi@cardiff.ac.uk}


\begin{abstract}
	The first part of this work uses the algorithm recently detailed in \cite{KawRel19} to classify the irreducible weight modules of the minimal model \voa\ $\simpk$, when the level $\kk$ is admissible.  These are naturally described in terms of families parametrised by up to two complex numbers.  We also determine the action of the relevant group of automorphisms of $\aslthree$ on their isomorphism classes and compute explicitly the decomposition into irreducibles when a given family's parameters are permitted to take certain limiting values.

	Along with certain character formulae, previously established in \cite{KawRel20}, these results form the input data required by the standard module formalism to consistently compute modular transformations and, assuming the validity of a natural conjecture, the Grothendieck fusion coefficients of the admissible-level \(\slthree\) minimal models.  The second part of this work applies the standard module formalism to compute these explicitly when \(\kk=-\frac32\).  This gives the first nontrivial test of this formalism for a nonrational \voa\ of rank greater than $1$ and confirms the expectation that the methodology developed here will apply in much greater generality.
\end{abstract}

\maketitle

\markleft{K~KAWASETSU, D~RIDOUT AND S~WOOD} 

\onehalfspacing

\section{Introduction} \label{sec:intro}

\subsection{Background}

Vertex operator algebras are versatile algebraic structures that have made numerous contributions to the fields of mathematical physics, representation theory, number theory and geometry, to name but a few. They are arguably best known as a rigorous algebraic axiomatisation of the chiral part of a two-dimensional \cft. Some of the earliest families of \voas\ to be studied were those constructed from untwisted affine \km\ algebras \cite{FreVer92}. When restricted to nonnegative-integral levels of these algebras, the associated \cfts\ describe strings propagating on (compact, connected, simply connected) Lie groups \cite{WitNon84,GepStr86}. In this case, the \cft\ is rational and the category of modules of the \voa\ is a modular tensor category.  In particular, the celebrated Verlinde formula holds for the fusion multiplicities \cite{VerFus88,MooCla89,HuaVer04,HuaVer04a}.

There are, however, many other interesting levels that one can consider such as the nonintegral admissible levels of Kac and Wakimoto. The corresponding \voas\ are no longer rational, but the characters of the \hwms\ span a representation of the modular group \cite{KacMod88}.  Unfortunately, the Verlinde formula gives negative fusion multiplicities in these cases \cite{KohFus88}, an inconsistency that led some in the research community to declare that nonintegral admissible levels might be ``intrinsically sick''.  Nevertheless, \voas\ with these levels are useful in studying nonunitary coset models \cite{KenInf86} and are essential for studying W-algebras via \qhr\ \cite{FeiQua90,KacQua03}.

The explanation for these negative fusion multiplicities was eventually isolated in \cite{RidSL208} for the \voa\ $\slnsimp{2}{-1/2}$ (and later extended to all $\slnsimp{2}{\kk}$ with $\kk$ admissible \cite{CreMod12,CreMod13} and many other \voas\ \cite{CreRel11,BabTak12,RidBos14,MorBou15,CanFus15,CanFus15b,RidAdm17,BabRep20,FehMod21}).  The problem lies with the fact that the character formulae of \cite{KacMod88} are analytic continuations of the formal power series that encode the graded dimensions of modules. Because these continuations do not completely distinguish irreducible modules, the Verlinde formula returns (Grothendieck) fusion multiplicities for a quotient category of (virtual) modules. To obtain the correct (Grothendieck) fusion multiplicities, the fix \cite{CreMod12} is to consider a much larger category of modules and treat their characters as distributions rather than meromorphic functions.  This fix was dubbed the ``standard module formalism'' in \cite{CreLog13,RidVer14}.

Whilst there is much evidence for the validity of the standard module formalism, including direct comparison of the results with independently obtained fusion multiplicities \cite{GabFus01,RidFus10,AdaFus19,AllBos20}, it has thus far been mainly tested on \voas\ that one could describe as being ``rank $1$'', like $\slnsimp{2}{\kk}$.  More precisely, the so-called standard modules for which the formalism is named form a $1$-parameter family in most examples.  In the few cases where a second parameter is needed, see for example \cite{CreRel11,BabTak12,BabRep20}, the additional parameter is the eigenvalue of a central element of the mode algebra, a relatively trivial generalisation.

In this paper, we report what we believe is the first application of the standard module formalism in a genuine rank-$2$ example, this being the admissible-level affine \voa\ $\simp{-3/2}$.  We moreover set up the general results needed to extend this application to $\simp{\kk}$, for any admissible $\kk$, the intent being to analyse this generalisation in the near future.  In fact, we emphasise the language of parabolic subalgebras throughout so that the reader will appreciate the means for, and difficulties inherent in, generalising to all admissible-level affine \voas.

\subsection{Results and outlook}

Let us describe our results in more detail.  The first is the classification of all irreducible weight $\simp{\kk}$-modules when $\kk$ is admissible (\cref{thm:class}).  While this classification was previously obtained by Futorny and collaborators using Gelfand--Tsetlin combinatorics \cite{AraWei16}, see also \cite{FutSim20,FutPos20,FutAdm21} for further recent progress, our proof is very different.  It follows from Arakawa's celebrated \hw\ classification \cite{AraRat16} by applying a vertex-algebraic generalisation of Mathieu's theory \cite{MatCla00} of coherent families, see \cite{KawRel19} for the details underlying this methodology.

This classification is not our main result however.  Experience with nonrational \voas\ has shown that a much finer understanding of the irreducible modules in the weight category is necessary.  In particular, one needs to know the explicit action of the invertible functors obtained by twisting with automorphisms of the underlying vertex algebra.  For $\simp{\kk}$, these are the automorphisms of the root system of $\slthree$ and the spectral flow automorphisms of $\aslthree$.  The results are quite intricate and are reported in \cref{prop:Wtwists,prop:dtwists,prop:sftwists}.

A second point is that the non\hw\ irreducible $\simp{\kk}$-modules form continuously parametrised (generalised coherent) families.  However, these parametrisations may be naturally extended to include certain reducible $\simp{\kk}$-modules and it is these reducible modules that are crucial for the correct understanding of the modularity of the weight category.  In \cref{prop:sdegen,prop:rdegen}, we determine their composition factors explicitly.  These two refinements to the irreducible classification are new and, as already mentioned, are necessary for a serious investigation of the further properties of the weight category.

These refinements serve as our first main result.  Our second is then a serious investigation of the weight category, specifically of its modularity.  This is significantly complicated by the fact that the relevant characters, themselves the result of highly nontrivial calculations \cite{KawRel20}, are almost never linearly independent.  However, this is not the case for precisely one (nonintegral) admissible level: $\kk=-\frac{3}{2}$.  Specialising to this level, we compute the modular S-transforms of these standard characters (\cref{thm:stcharmod}) and then combine this with our first main result to calculate the Grothendieck fusion multiplicities for all irreducible weight $\simp{-3/2}$-modules (\cref{sec:32fusion}).  This relies on the conjectural ``standard Verlinde formula'' of \cite{CreLog13,RidVer14}.  These modularity investigations constitute our second main result.  The fact that the multiplicities are indeed found to be nonnegative integers is a strong endorsement of the applicability of the standard module formalism to higher-rank nonrational \voas.

A natural question to ask is how one can bypass the linear dependence problem for other (nonintegral) admissible levels.  An identical (and in fact closely related) issue arises with the characters of Zamolodchikov's $W_3$ minimal models \cite{ZamInf85}.  Whilst the modularity of these rational \voas\ has been known for at least thirty years, see \cite{FreCha92} for example, a correct derivation of the modular S-matrix only appeared recently in work of Arakawa and van~Ekeren \cite{AraMod19} because of the linear dependence of the irreducible $W_3$ characters.  This derivation required the development of technology to compute with generalisations of characters called one-point functions.  These functions only differ from characters by the insertion of an additional zero-mode and, with the right zero-mode, the one-point functions of the irreducible weight modules are indeed linearly independent \cite{ZhuMod96}.  Unfortunately, explicit formulae for these functions are unknown.

This work nevertheless opens the door to exploring the modularity of weight $\simp{\kk}$-modules for all admissible levels $\kk$.  They key here is the inverse quantum hamiltonian reduction functors introduced by Adamovi\'{c} in \cite{AdaRea17}, see also \cite{SemInv94}.  These can be used to deduce remarkable character formulae for certain irreducible $\simp{\kk}$-modules.  Indeed, since this paper was written, it has been shown that the string functions of some of these modules are irreducible characters of the \bp\ minimal models \cite{AdaRel21}.  As these \bp\ characters have string functions that were shown to be irreducible $W_3$-characters in \cite{AdaRea20}, it only remains to lift the Arakawa--van~Ekeren one-point function modularity from $W_3$ to $\simp{\kk}$.  Since this paper was written, the lift to the \bp\ minimal models has in fact appeared \cite{FehMod21}.  The remaining lift is however not quite as straightforward as the modules considered in \cite{AdaRel21} may not exhaust the standard modules of $\simp{\kk}$ in general.  We hope to report on this in the future.

A second natural question to ask concerns the obvious conjecture that $\simp{\kk}$ admits logarithmic weight modules, these being modules on which the Cartan zero-modes act semisimply but the Virasoro zero-mode acts nonsemisimply with Jordan blocks of finite rank.  Experience with $\slnsimp{2}{\kk}$ \cite{GabFus01,AdaCon04,RidFus10,AdaRea17} indicates that such modules surely do exist and, again, since this paper was written, their existence has been confirmed \cite{AdaRel21}.  Unfortunately, the structure of the known logarithmic $\simp{\kk}$-modules remains mysterious --- even their composition factors are unknown.  However, a companion paper to this one \cite{CreKaz21} combines the results presented here with a conjectural logarithmic \kl\ correspondence to posit not only composition factors but also complete Loewy diagrams and fusion rules (but only in the special case $\kk = -\frac{3}{2}$).

In this context, the work reported here represents an important part of a rapidly advancing program to understand the representation theory and modularity of higher-rank affine \voas\ and W-algebras.  With the methodology being carefully selected to facilitate generalisations, we expect to uncover a beautiful general theory for weight $\VOA{L}_{\kk}(\alg{g})$-modules at (nonintegral) admissible levels $\kk$ that significantly extends that of the well known integrable \hwms\ that describe the spectrum when $\kk$ a non-negative integer.

\subsection{Outline}

The paper is organised as follows. In \cref{sec:Lie}, we fix our conventions for the (\fdim) Lie algebra \(\slthree\), summarise pertinent properties of its automorphisms and outline Mathieu's celebrated classification of its irreducible weight modules (with \fdim\ weight spaces) \cite{MatCla00}. In \cref{sec:Aff}, we similarly fix our conventions for the affine \km\ algebra \(\aslthree\), introduce the affine automorphisms called spectral flow and discuss the Zhu-inductions \cite{ZhuMod96} of the irreducible weight $\slthree$-modules.  We also explicitly identify the result of applying spectral flow to such a Zhu-induced module (when the result is again positive-energy).

The simple admissible-level \voas\ $\simp{\kk}$ are then introduced in \cref{sec:MinMod}.  After reviewing the admissible weights of Kac--Wakimoto, we quote a specialisation of Arakawa's famous classification of \hw\ $\simp{\kk}$-modules \cite{AraRat16}.  This is the input required by the algorithm presented in \cite{KawRel19} to classify the irreducible \rhw\ $\simp{\kk}$-modules.  We present the results of this algorithm and thereby obtain explicit conditions for an irreducible weight $\aslthree$-module to be an $\simp{\kk}$-module.  This result is naturally presented in terms of vertex-algebraic generalisations of coherent families.

After comparing this classification with the nilpotent orbits in $\slthree$, we carefully analyse two features of the resulting families of irreducibles.  First, we identify the isomorphism class of the result of twisting any of the irreducible weight $\simp{\kk}$-modules by an automorphism.  We also determine the composition factors of the reducible $\simp{\kk}$-modules that naturally arise when the parameters appearing in our classification take certain limiting values.  Both analyses are rather intricate, but are essential for the standard module formalism computations that follow.

Finally, \cref{sec:minmod32} specialises to the level $-\frac{3}{2}$.  First, we identify a set of standard modules for $\simp{-3/2}$ and note that their characters are linearly independent.  We compute the modular S-transforms of these standard characters and then use the result to calculate the Grothendieck fusion multiplicities for all irreducible weight $\simp{-3/2}$-modules from the conjectural ``standard Verlinde formula'' of \cite{CreLog13,RidVer14}.

\section*{Acknowledgements}

KK's research is partially supported by the Australian Research Council Discovery Project DP160101520, MEXT Japan ``Leading Initiative for Excellent Young Researchers (LEADER)'', JSPS Kakenhi Grant numbers 19KK0065 and 21K3775.
DR's research is supported by the Australian Research Council Discovery Projects DP160101520 and DP210101502, as well as an Australian Research Council Future Fellowship FT200100431.
SW's research is supported by the Australian Research Council
Discovery Project DP160101520 and the Humboldt Fellowship for Experience
Researchers GBR-1212053-HFST-E.

\section{The \fdim\ Lie algebra $\slthree$ and its weight modules} \label{sec:Lie}

In this section, we introduce our notation and conventions for the simple Lie algebra $\slthree$, its automorphisms and its weight modules.  For convenience, we shall assume here and throughout that the definition of a weight module requires that every weight space is \fdim.

\subsection{The simple Lie algebra $\slthree$} \label{sec:LieAlg}

Let $E_{ij}$ denote the $3 \times 3$ matrix whose entries are all zero save for that corresponding to row $i$ and column $j$ which is instead $1$.  The Lie algebra $\slthree$ is the complex vector space spanned by the traceless matrices
\begin{equation} \label{eq:sl3basis}
	e^3 = E_{13}, \quad
	\begin{aligned}
			e^1 &= E_{12}, & h^1 &= E_{11}-E_{22}, & f^1 &= E_{21}, \\
			e^2 &= E_{23}, & h^2 &= E_{22}-E_{33}, & f^2 &= E_{32},
	\end{aligned}
	\quad f^3 = E_{31},
\end{equation}
equipped with the matrix commutator as Lie bracket.  We fix the Cartan subalgebra to be $\csub = \vspn_{\CC} \set{h^1, h^2}$ and normalise the Killing form so that its only nonvanishing entries, with respect to the basis \eqref{eq:sl3basis}, are
\begin{equation} \label{eq:Killing}
	\killing{e^i}{f^j} = \killing{f^i}{e^j} = \delta^{ij}, \quad \killing{h^k}{h^{\ell}} = A^{k\ell}, \quad i,j=1,2,3,\ k,\ell=1,2.
\end{equation}
Here, $A = \tcartan$ denotes the Cartan matrix of $\slthree$.  In what follows, the pairing of the dual space $\dcsub$ with $\csub$ will be denoted by $\pair{\blank}{\blank}$ while the bilinear form on $\dcsub$ induced from $\killingsymb$ will be denoted by $\bilin{\blank}{\blank}$.

With this data, we describe the root system $\roots$ of $\slthree$ using the following conventions:
\begin{itemize}
	\item The simple roots, denoted by $\sroot{1}$ and $\sroot{2}$, correspond to the simple root vectors $e^1$ and $e^2$, respectively.
	\item The simple coroots are $h^1$ and $h^2$ while the (dual) fundamental weights are denoted by $\fwt{1}$ and $\fwt{2}$, respectively.
	\item The highest root is denoted by $\sroot{3}$.  Its root vector is $e^3$ and its coroot is $h^3 = h^1 + h^2$.
\end{itemize}
When convenient, the element of \eqref{eq:sl3basis} that defines a root vector for the root $\alpha$ will also be denoted by $e^{\alpha}$.  If $\alpha$ is a positive root, then we will also use the notation $f^{\alpha} = e^{-\alpha}$.  We illustrate these roots and weights in \cref{fig:sl3roots}.
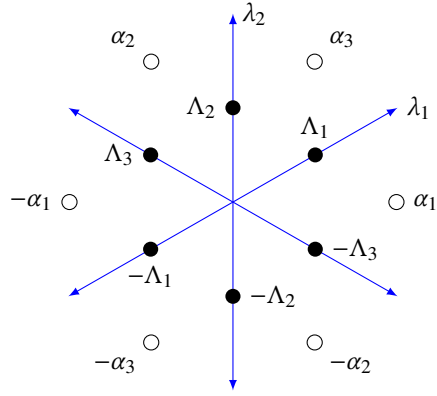
\begin{figure}
	\begin{tikzpicture}[scale=2.5,<->,>=latex]
		\makeaxes{node[black,right] {$\lambda_1$}}{node[black,right] {$\lambda_2$}}
		\node[wt,label=right:$\sroot{1}$] at (0:0.86) {};
		\node[wt,label=above right:$\sroot{3}$] at (60:0.86) {};
		\node[wt,label=above left:$\sroot{2}$] at (120:0.86) {};
		\node[wt,label=left:$-\sroot{1}$] at (180:0.86) {};
		\node[wt,label=below left:$-\sroot{3}$] at (-120:0.86) {};
		\node[wt,label=below right:$-\sroot{2}$] at (-60:0.86) {};
		\node[owt,label=above:$\fwt{1}$] at (30:0.5) {};
		\node[owt,label=left:$\fwt{2}$] at (90:0.5) {};
		\node[owt,label=left:$\fwt{3}$] at (150:0.5) {};
		\node[owt,label=below:$-\fwt{1}$] at (-150:0.5) {};
		\node[owt,label=right:$-\fwt{2}$] at (-90:0.5) {};
		\node[owt,label=right:$-\fwt{3}$] at (-30:0.5) {};
	\end{tikzpicture}
\caption{The roots (white) of $\slthree$ in (the real slice of) $\dcsub$ along with the weights of the fundamental $\slthree$-modules (black).  Here, we set $\fwt{3} = \fwt{2} - \fwt{1}$ for convenience.  The coordinates $\lambda_1$ and $\lambda_2$ are identified as Dynkin labels.} \label{fig:sl3roots}
\end{figure}

The weight and root lattices of $\slthree$ will be denoted by
\begin{equation}
	\wlat = \vspn_{\ZZ} \set{\fwt{1},\fwt{2}} \subset \dcsub \quad \text{and} \quad \rlat = \vspn_{\ZZ} \set{\sroot{1},\sroot{2}} \subset \dcsub,
\end{equation}
respectively.  The set of dominant integral weights will be denoted by $\pwlat = \vspn_{\NN} \set{\fwt{1},\fwt{2}}$.  The integer duals, with respect to $\pair{\blank}{\blank}$, of $\wlat$ and $\rlat$ are the coroot and coweight lattices
\begin{equation}
	\crlat = \vspn_{\ZZ} \set{h^1,h^2} \subset \csub \quad \text{and} \quad \cwlat = \vspn_{\ZZ} \set{\fcwt{1},\fcwt{2}} \subset \csub,
\end{equation}
respectively.  The fundamental coweights $\fcwt{1}$ and $\fcwt{2}$ thus
satisfy $\pair{\sroot{i}}{\fcwt{j}} = \delta_{ij}$.

As $\slthree$ is a rank $2$ simple Lie algebra with exponents $1$ and $2$, the centre of its \uea{} $\envalg{\slthree}$ is a polynomial ring in two generators $Q$ and $C$, where $Q$ is quadratic in the basis \eqref{eq:sl3basis} and $C$ is cubic.  Explicit expressions for these generators are easily found:
\begin{subequations} \label{eq:defCasimirs}
	\begin{align}
		Q &= \frac{1}{3} \brac*{h^1 h^1 + h^2 h^2 + h^3 h^3} + h^1 + h^2 + h^3 + 2 \brac*{f^1 e^1 + f^2 e^2 + f^3 e^3}, \label{eq:defQ} \\
		C &= \brac*{h^1+2h^2+3} \brac*{2h^1+h^2+3} \brac*{h^1-h^2} \notag \\
		&\qquad + 9 f^1 e^1 \brac*{h^1+2h^2+3} - 9 f^2 e^2 \brac*{2h^1+h^2+3} + 9 f^3 e^3 \brac*{h^1-h^2} + 27 f^1 f^2 e^3 + 27 f^3 e^1 e^2. \label{eq:defC}
	\end{align}
\end{subequations}
We remark that the factors $h^1+2h^2+3$, $2h^1+h^2+3$, and $h^1-h^2$ appearing in $C$ commute with the generators $e^1$ and $f^1$, $e^2$ and $f^2$, and $e^3$ and $f^3$, respectively.

\subsection{Automorphisms of $\slthree$} \label{sec:LieAut}

The automorphisms of $\roots$ come in two flavours: the $6$ inner automorphisms of the Weyl group $\weyl \cong \grp{S}_3$ and the $6$ outer automorphisms obtained by composing a Weyl symmetry with the $\ZZ_2$-symmetry of the Dynkin diagram.  Together, these automorphisms form the dihedral group $\grp{D}_6 \cong \grp{S}_3 \times \ZZ_2$.  We denote the simple Weyl reflections by $\weylaut{1}$ and $\weylaut{2}$.  The reflection corresponding to the highest root $\sroot{3}$ is denoted by $\weylaut{3} = \weylaut{1} \weylaut{2} \weylaut{1} = \weylaut{2} \weylaut{1} \weylaut{2}$.  The actions on $\csub$ and $\dcsub$ are thus
\begin{equation} \label{eq:Waction}
	\weylmod{i}{h} = h - \killing{h}{h^i} h^i \quad \text{and} \quad
	\weylmod{i}{\lambda} = \lambda - \pair{\lambda}{h^i} \sroot{i}, \qquad i=1,2,3,\ h \in \csub,\ \lambda \in \dcsub.
\end{equation}
Note that $\lambda_i = \pair{\lambda}{h^i}$ is, for $i=1,2$, the $i$-th Dynkin label of $\lambda$.  We also recall the shifted action of $\weyl$ on $\dcsub$:
\begin{equation} \label{eq:shWaction}
	\weylaut{i} \cdot \lambda = \weylmod{i}{\lambda + \wvec} - \wvec = \lambda - (\lambda_i + 1) \sroot{i}, \qquad i=1,2,\ \lambda \in \dcsub.
\end{equation}
Here, $\wvec = \fwt{1} + \fwt{2}$ denotes the Weyl vector of $\slthree$.

The automorphism corresponding to the Dynkin symmetry is denoted by $\dynkaut$ and acts by
\begin{equation} \label{eq:daction}
	\dynkmod{h^i} = h^{\dynkmod{i}} \quad \text{and} \quad
	\dynkmod{\fwt{i}} = \fwt{\dynkmod{i}},
\end{equation}
extended to $\csub$ and $\dcsub$ by linearity.  Here, $\dynkaut$ acts on $i \in \set{1,2,3}$ as the transposition $(1,2)$.  We also have $\dynkaut w_i = w_{\dynkmod{i}} \dynkaut$.  The automorphism $\conjaut = \dynkaut \weylaut{3} = \weylaut{3} \dynkaut$ is called the \emph{conjugation} automorphism.  It acts on $\csub$ and $\dcsub$ as $-1$ times the identity and so is central in $\grp{D}_6$: $\conjaut \weylaut{i} = \weylaut{i} \conjaut$.  Extending the shifted action to $\grp{D}_6$ gives
\begin{equation} \label{eq:shdaction}
	\dynkaut \cdot \lambda = \dynkmod{\lambda} \quad \text{and} \quad
	\conjaut \cdot \lambda = -\lambda - 2 \wvec, \qquad \lambda \in \dcsub.
\end{equation}

Each of these root system automorphisms may be extended to an automorphism of $\slthree$ and we shall use the same notation for these extensions as for their restrictions to $\alg{h}$.  These extensions are not unique, but we shall fix one choice arbitrarily for each automorphism of the root system, noting that different choices will not affect what follows.  The resulting set of $\slthree$-automorphisms does not satisfy the defining relations of $\grp{D}_6$, in particular those corresponding to reflections need not square to the identity.  However, we still have the property that each extended automorphism $\omega$ maps the root space labelled by the root $\alpha$ into the root space labelled by $\func{\omega}{\alpha}$.  It follows that the extended automorphisms define a projective action of $\grp{D}_6$ on $\slthree$.

We shall make frequent use of twisting representations of $\slthree$ by automorphisms.  This amounts to applying the automorphism before acting with the representation morphism.  As we prefer to keep representations implicit, we implement this twisting notationally through the language of modules as follows:  Given an $\slthree$-automorphism $\omega$ and an $\slthree$-module $\Mod{M}$, define $\func{\omega^*}{\Mod{M}}$ to be the image of $\Mod{M}$ under an (arbitrarily chosen) isomorphism $\omega^*$ of vector spaces.  The action of $\slthree$ on $\func{\omega^*}{\Mod{M}}$ is then defined by
\begin{equation}
	x \, \func{\omega^*}{v} = \func{\omega^*}{\func{\omega^{-1}}{x} \, v}, \quad x \in \slthree,\ v \in \Mod{M}.
\end{equation}
In other words, $\func{\omega}{x} \, \func{\omega^*}{v} = \func{\omega^*}{xv}$.  In view of this natural property, we shall, from here on, drop the star that distinguishes the automorphism $\omega$ from the corresponding vector space isomorphism $\omega^*$.

The extensions of the root system automorphisms therefore define autoequivalences on the category $\wcat$ of weight $\slthree$-modules.  However, it is not clear if these functors may be chosen to give an action of $\grp{D}_6$ on this category because the extensions do not satisfy all the defining relations of $\grp{D}_6$.  Nevertheless, we have a well defined $\grp{D}_6$-action on the set of isomorphism classes of weight $\slthree$-modules.  As we are mostly concerned with identifying modules up to isomorphism, this will suffice for what follows.  We remark that this action is obviously structure-preserving.  In particular, it preserves irreducibility.

\subsection{Irreducible weight $\slthree$-modules} \label{sec:LieIrr}

Recall that a parabolic subalgebra is one that contains a Borel subalgebra.  Every irreducible weight module over a simple (finite-dimensional complex) Lie algebra $\alg{g}$ is isomorphic to a module obtained by the following procedure \cite{FerLie90}:
\begin{itemize}
	\item Choose a parabolic subalgebra $\para$ of $\alg{g}$ and an irreducible dense module $\Mod{D}$ over the Levi factor $\levi$ of $\alg{p}$.
	\item Extend $\Mod{D}$ to a $\para$-module by letting the nilradical $\nilrad$ of $\para$ act trivially.
	\item Take the irreducible quotient of the induced $\alg{g}$-module $\Ind_{\para}^{\alg{g}} \Mod{D}$.
\end{itemize}
Here, a \emph{dense} module is a weight module whose weight support (its set of weights) coincides with a coset of $\dcsub / \rlat$, where $\rlat$ is the root lattice of $\alg{g}$.  Such weight supports are clearly maximal among those of all indecomposable weight modules.  We mention that irreducible dense modules are also called \emph{cuspidal} and \emph{torsion-free} in the literature.  Irreducible dense modules are known to only exist in types $\mathrm{A}$ and $\mathrm{C}$ \cite{FerLie90}.

For $\slthree$, there are four distinct parabolic subalgebras, up to twisting by $\weyl \cong \grp{S}_3$:
\begin{itemize}
	\item The Borel subalgebra $\borel = \vspn_{\CC} \set{e^1,e^2,e^3,h^1,h^2}$ with Levi \(\csub\) and nilradical \(\vspn_{\CC}\set{e^1,e^2,e^3}\).
	\item The subalgebra $\para = \vspn_{\CC} \set{e^1,e^2,e^3,h^1,h^2,f^1}$ with Levi $\levi = \vspn_{\CC}\set{e^1,h^1,h^2,f^1} \cong \sltwo \oplus \glone           \cong \gltwo$ and nilradical \(\nilrad=\vspn_{\CC}\set{e^2,e^3}\).
	\item The subalgebra $\dynkmod{\para} = \vspn_{\CC} \set{e^1,e^2,e^3,h^1,h^2,f^2}$ with Levi $\dynkmod{\levi} = \vspn_{\CC} \set{e^2,h^1,h^2,f^2} \cong \sltwo \oplus \glone \cong \gltwo$ and nilradical \(\dynkmod{\nilrad} = \vspn_{\CC} \set{e^1,e^3}\).
	\item The entire algebra $\slthree$ whose Levi is also $\slthree$ and whose nilradical is \(0\).
\end{itemize}
When the parabolic is $\borel$, the procedure above starts from an irreducible dense $\csub$-module.  As this is one-dimensional, the result is an irreducible \hw\ $\slthree$-module, classified up to isomorphism by its highest weight $\mu \in \dcsub$.  We shall therefore denote it by $\firr{\mu}$.  The weight support of a generic $\firr{\mu}$ is illustrated in \cref{fig:sl3mods} (left).  This irreducible may be twisted by an automorphism (autoequivalence) $\omega$.  In general, this results in an irreducible \hwm\ with respect to a different choice of Borel, namely $\func{\omega}{\borel}$.  We remark however that $\dynkaut$ preserves the standard Borel $\borel$, so $\dynkmod{\firr{\mu}} \cong \firr{\dynkmod{\mu}}$.  Moreover, we have $\weylmod{i}{\firr{\mu}} \cong \firr{\mu}$, for $i=1,2$, if and only if $\mu_i \in \NN$.
\begin{figure}
	\begin{tikzpicture}[scale=1.8,<->,>=latex,baseline=(current bounding box.center)]
		\makeaxes{}{}
		\begin{scope}[shift={(0.2,0.3)}]
			\node[owt,label=above right:$\scriptstyle\mu$] at (0,0) {};
			\foreach \n in {1,2,3,4} \node[owt] at (180:0.3*\n) {};
			\begin{scope}[shift={(-60:0.3)}]
				\foreach \n in {0,1,2,3,4} \node[owt] at (180:0.3*\n) {};
			\end{scope}
			\begin{scope}[shift={(-60:0.6)}]
				\foreach \n in {0,1,2,3,4} \node[owt] at (180:0.3*\n) {};
			\end{scope}
			\begin{scope}[shift={(-60:0.9)}]
				\foreach \n in {0,1,2,3,4} \node[owt] at (180:0.3*\n) {};
			\end{scope}
			\begin{scope}[shift={(-60:1.2)}]
				\foreach \n in {0,1,2,3,4} \node[owt] at (180:0.3*\n) {};
			\end{scope}
			\node[] at (180:0.3*5) {$\cdots$};
			\node[rotate=-60] at (-60:0.3*5) {$\cdots$};
			\node[rotate=60] at (-120:0.3*5) {$\cdots$};
		\end{scope}
	\end{tikzpicture}
	\hfill
	\begin{tikzpicture}[scale=1.8,<->,>=latex,baseline=(current bounding box.center)]
		\makeaxes{}{}
		\begin{scope}[shift={(0.15,0.4)}]
			\node[owt,label=above:$\scriptstyle\mu$] at (0,0) {};
			\foreach \n in {-3,-2,-1,1,2,3,4} \node[owt] at (180:0.3*\n) {};
			\begin{scope}[shift={(-60:0.3)}]
				\foreach \n in {-3,-2,-1,0,1,2,3,4,5} \node[owt] at (180:0.3*\n) {};
			\end{scope}
			\begin{scope}[shift={(-60:0.6)}]
				\foreach \n in {-2,-1,0,1,2,3,4,5} \node[owt] at (180:0.3*\n) {};
			\end{scope}
			\begin{scope}[shift={(-60:0.9)}]
				\foreach \n in {-2,-1,0,1,2,3,4,5,6} \node[owt] at (180:0.3*\n) {};
			\end{scope}
			\begin{scope}[shift={(-60:1.2)}]
				\foreach \n in {-1,0,1,2,3,4,5,6} \node[owt] at (180:0.3*\n) {};
			\end{scope}
			\node[] at (180:0.3*5) {$\cdots$};
			\node[] at (0:0.3*4) {$\cdots$};
			\node[rotate=-60] at (-60:0.3*5) {$\cdots$};
			\node[rotate=60] at (-120:0.3*5) {$\cdots$};
		\end{scope}
	\end{tikzpicture}
	\hfill
	\begin{tikzpicture}[scale=1.8,<->,>=latex,baseline=(current bounding box.center)]
		\makeaxes{}{}
		\begin{scope}[shift={(0.15,0.25)}]
			\node[owt,label=above:$\scriptstyle\mu$] at (0,0) {};
			\foreach \n in {-3,-2,-1,1,2,3,4} \node[owt] at (180:0.3*\n) {};
			\begin{scope}[shift={(-60:0.3)}]
				\foreach \n in {-2,-1,0,1,2,3,4} \node[owt] at (180:0.3*\n) {};
				\node[] at (180:0.3*5) {$\cdots$};
				\node[] at (0:0.3*3) {$\cdots$};
			\end{scope}
			\begin{scope}[shift={(-60:0.6)}]
				\foreach \n in {-2,-1,0,1,2,3,4,5} \node[owt] at (180:0.3*\n) {};
			\end{scope}
			\begin{scope}[shift={(-60:0.9)}]
				\foreach \n in {-1,0,1,2,3,4,5} \node[owt] at (180:0.3*\n) {};
			\end{scope}
			\begin{scope}[shift={(-60:1.2)}]
				\foreach \n in {-1,0,1,2,3,4,5,6} \node[owt] at (180:0.3*\n) {};
			\end{scope}
			\begin{scope}[shift={(120:0.3)}]
				\foreach \n in {-3,-2,-1,0,1,2,3} \node[owt] at (180:0.3*\n) {};
			\end{scope}
			\begin{scope}[shift={(120:0.6)}]
				\foreach \n in {-4,-3,-2,-1,0,1,2,3} \node[owt] at (180:0.3*\n) {};
			\end{scope}
			\node[rotate=-60] at (-60:0.3*5) {$\cdots$};
			\node[rotate=60] at (-120:0.3*5) {$\cdots$};
			\begin{scope}[shift={(-90:0.3*1.732)}]
				\node[rotate=60] at (60:0.3*5) {$\cdots$};
				\node[rotate=-60] at (120:0.3*5) {$\cdots$};
			\end{scope}
		\end{scope}
	\end{tikzpicture}
\caption{The weight support of a generic irreducible \hwm\ $\firr{\mu}$ (left), a generic irreducible semidense module $\fsemi{\lambda}{[\mu]}$ (middle) and a generic irreducible dense module $\frel{\lambda}{[\mu]}$ (right) for $\slthree$.  In each case, we indicate a weight $\mu$, noting that the coset in the semidense and dense cases is $[\mu] = \mu + \ZZ \sroot{1}$ and $[\mu] = \mu + \rlat$, respectively.} \label{fig:sl3mods}
\end{figure}

When the parabolic is $\para$, the procedure instead starts from an irreducible dense $\levi$-module, $\levi \cong \gltwo$, that is the tensor product of an irreducible dense $\sltwo$-module and a one-dimensional $\glone$-module.  In \cite{MatCla00}, Mathieu classified the irreducible dense modules over a general reductive Lie algebra $\alg{g}$, with Cartan subalgebra $\csub$, by showing that each may be uniquely realised as a direct summand of an irreducible semisimple \emph{coherent family}
\begin{equation}
	\coh{} = \bigoplus_{[\mu]} \coh{}_{[\mu]}.
\end{equation}
Here, the sum is over cosets of the weight support of $\coh{}$ modulo the root lattice of $\alg{g}$.  A coherent family of $\alg{g}$-modules is defined to satisfy the following conditions:
\begin{itemize}
	\item The dimension of the weight space $\coh{}(\mu)$ is independent of the weight $\mu \in \dcsub$.
	\item The trace of the action on $\coh{}(\mu)$ of every $U$ in the centraliser of the Cartan, $\ccent{\alg{g}}$, is a polynomial in $\mu$.
\end{itemize}
We recall that $\ccent{\alg{g}}$ is the subalgebra of $\envalg{\alg{g}}$ consisting of all weight-$0$ elements.  The direct summands $\coh{}_{[\mu]}$ are always dense and, in fact, are irreducible for all but finitely many $[\mu]$.  A coherent family $\coh{}$ is said to be irreducible, if some $\coh{}_{[\mu]}$ is irreducible, and semisimple, if every $\coh{}_{[\mu]}$ is completely reducible.

For $\alg{g} = \levi \cong \sltwo \oplus \glone$, it is easy to show that $\ccent{\levi}$ is a polynomial algebra in the Cartan basis elements and the quadratic Casimir.  It follows that an irreducible dense $\levi$-module has $1$-dimensional weight spaces, hence so do the corresponding coherent families.  Let $\sroot{}$ denote the simple root of $\sltwo$ and let $e$ and $f$ be the corresponding positive and negative root vectors, respectively.  Then, $fe \in \ccent{\levi}$ acts as a (nonconstant) polynomial in $\mu$ on the $\coh{}(\mu)$, hence it acts on some $\coh{}(\mu)$ as zero.  It therefore follows that this $\coh{}(\mu)$ is either spanned by a \hwv\ or $\coh{}(\mu+\sroot{})$ is spanned by a \lwv.  We may choose, among the $\mu$ satisfying this property, one whose $\sltwo$ Dynkin index has minimal real part.  Denote this weight by $\lambda$.  It then follows that every coherent family of $\levi$-modules possesses a reducible direct summand $\coh{}_{[\lambda]}$ with a composition factor isomorphic to an \infdim\ \hwm\ with highest weight $\lambda$.

Conversely, every irreducible \infdim\ \hw\ $\levi$-module yields, via the \emph{twisted localisation} functors introduced in \cite{MatCla00}, an irreducible coherent family with the \hwm\ as a composition factor.  This may then be \emph{semisimplified} if an irreducible semisimple coherent family is desired.  Moreover, two irreducible \infdim\ \hwms\ yield isomorphic (irreducible semisimple) coherent families of $\levi$-modules if and only if their highest weights, $\lambda$ and $\lambda'$ say, satisfy $\lambda = \lambda'$ or $\lambda = \weylaut{1} \cdot \lambda'$.\footnote{If we were considering coherent families of $\dynkmod{\levi}$-modules instead of $\levi$-modules, then the Weyl reflection appearing here would be $\weylaut{2}$.}

It follows that irreducible semisimple coherent families of $\levi$-modules (and hence dense $\levi$-modules) are classified by irreducible \infdim\ \hw\ $\levi$-modules (modulo the shifted action of the Weyl group of $\levi$).  We shall therefore denote an irreducible semisimple coherent family of $\levi$-modules by $\coh{\lambda}$, where $\lambda$ is the highest weight of one of its \infdim\ \hw\ composition factors, noting again that $\lambda$ is only determined uniquely up to the shifted action of the Weyl group.  We also note that the direct summand $\coh{\lambda}_{[\mu]}$ is reducible if and only if $[\mu] = [\lambda]$ or $[\weylaut{1} \cdot \lambda]$.

Returning now to the construction of irreducible $\slthree$-modules corresponding to the parabolic $\para$, note that the weight support of the irreducible semisimple coherent family $\coh{\lambda}$ of $\levi$-modules is the coset $\lambda + \CC \sroot{1} \in \dcsub / \CC \sroot{1}$.   That of the irreducible dense $\levi$-module $\coh{\lambda}_{[\mu]}$ then coincides with some coset $[\mu] \in \dcsub / \ZZ \sroot{1}$, where $\mu \in \lambda + \CC \sroot{1}$.  In other words, we have $[\mu] \in (\lambda + \CC \sroot{1}) / \ZZ \sroot{1}$.  The weight support of the irreducible quotient of the $\slthree$-module induced from $\coh{\lambda}_{[\mu]}$ then coincides with
\begin{equation}
	\mu + \ZZ \sroot{1} - \NN \sroot{2} = \set*{\mu+m\sroot{1}-n\sroot{2}\st  m\in\ZZ,\ n\in\NN}.
\end{equation}
We shall refer to this irreducible quotient as a \emph{semidense} $\slthree$-module, denoting it by $\fsemi{\lambda}{[\mu]}$.  The weight support of a generic $\fsemi{\lambda}{[\mu]}$ is illustrated in \cref{fig:sl3mods} (middle).

Twisting $\fsemi{\lambda}{[\mu]}$ by an automorphism $\omega$ again results in a new irreducible semidense module corresponding to $\omega(\para)$.  In particular, we obtain all the irreducible semidense modules for the parabolic $\dynkmod{\para}$ in this fashion.  Note also that twisting $\fsemi{\lambda}{[\mu]}$ by $\weylaut{1}$ results in the semidense module $\fsemi{\lambda}{[\weylmod{1}{\mu}]}$, because $\weylmod{1}{\para} = \para$.

It therefore only remains to consider the case in which the parabolic is all of $\slthree$, hence the corresponding irreducible modules are dense.  In this case, the irreducible semisimple coherent families of $\slthree$-modules again have irreducible \infdim\ \hw\ submodules.  But as the weight multiplicities of a coherent family are uniformly bounded, the same must be true for these submodules.  A weight $\slthree$-module is said to be \emph{bounded} if it is \infdim\ and its weight multiplicities are uniformly bounded.  Mathieu's work \cite{MatCla00} then classifies the irreducible semisimple coherent families $\coh{\lambda}$ of $\slthree$-modules in terms of the highest weight $\lambda \in \dcsub$ of any one of their (necessarily irreducible) bounded \hw\ $\slthree$-submodules.  Again, $\lambda$ is only unique up to the shifted action of $\weyl$.

We shall therefore denote an irreducible dense $\slthree$-module by $\frel{\lambda}{[\mu]}$, where $[\mu] \in \dcsub / \rlat$ is the weight support and $\lambda \in \dcsub$ specifies the unique (up to isomorphism) irreducible semisimple coherent family $\coh{\lambda}$ of $\slthree$-modules with $\coh{\lambda}_{[\mu]} \cong \frel{\lambda}{[\mu]}$.  The weight support of a generic $\frel{\lambda}{[\mu]}$ is illustrated in \cref{fig:sl3mods} (right).  It turns out that every $\weyl$-twist of a semisimple coherent family results in an isomorphic semisimple coherent family \cite{MatCla00}.  However, twists by $\dynkaut$ (and hence $\conjaut$) need not preserve the isomorphism class.  An easy way to see this is to note from \eqref{eq:defC} that $\conjaut$, which acts on the defining representation \eqref{eq:sl3basis} of $\slthree$ as $A \mapsto -A^T$, sends the cubic Casimir $C$ to $-C$.

\section{The affine \km\ algebra \texorpdfstring{$\aslthree$}{sl3} and its weight modules} \label{sec:Aff}

The focus of this \lcnamecref{sec:Aff} is the affine \km\ algebra $\aslthree$, its associated affine vertex algebras and their (smooth) weight modules.  We assume that a weight $\aslthree$-module has finite-dimensional weight spaces, where a weight space is defined to be the intersection of a simultaneous eigenspace of $\csub \oplus \CC K \ira \aslthree$, where $K$ denotes the central element, and a generalised eigenspace of the Virasoro zero mode (derivation) $L_0$.

\subsection{The affine \km\ algebra \texorpdfstring{$\aslthree$}{sl3} and its associated vertex algebras} \label{sec:AffAlg}

The affinisation of $\slthree$, with respect to the normalised Killing form \eqref{eq:Killing}, is
\begin{equation} \label{eq:asl3}
	\aslthree = \slthree \otimes \CC[t,t^{-1}] \oplus \CC K,
\end{equation}
where the Lie bracket is defined by
\begin{equation} \label{eq:asl3comm}
	\comm{x \otimes t^m}{y \otimes t^n} = \comm{x}{y} \otimes t^{m+n} + m \killing{x}{y} \delta_{m+n,0} K, \quad
	\comm{x \otimes t^m}{K} = 0, \qquad x,y \in \slthree,\ m,n \in \ZZ.
\end{equation}
We shall usually follow standard practice in abbreviating $x \otimes t^m$ as $x_m$.

There is an obvious decomposition of this affinisation into subalgebras:
\begin{equation}
	\begin{gathered}
		\aslthree = \aslpos \oplus \aslzero \oplus \aslneg; \\
		\aslpos = \vspn_{\CC} \set{x_n \st x \in \slthree,\ n \in \ZZ_{>0}}, \quad
		\aslzero = \vspn_{\CC} \set{x_0, K \st x \in \slthree}, \quad
		\aslneg = \vspn_{\CC} \set{x_n \st x \in \slthree,\ n \in \ZZ_{<0}}.
	\end{gathered}
\end{equation}
With this generalised triangular decomposition, we may extend an arbitrary $\slthree$-module to an $(\aslpos \oplus \aslzero)$-module by letting $K$ act as some multiple $\kk \in \CC$ of the identity and $\aslpos$ act trivially.  This may then be induced to an $\aslthree$-module.  The resulting $\aslthree$-module is called the level-$\kk$ affinisation of the original $\slthree$-module.

If we affinise the trivial $\slthree$-module in this fashion, then we obtain a
parabolic Verma module of highest weight $\kk \afwt{0}$, where $\afwt{i}$
denotes the $i$-th fundamental weight of $\aslthree$.  It is well known
\cite{FreVer92} that this module admits the structure of a vertex algebra with
strong generators and (\ope) relations  given by
\begin{equation} \label{eq:UnivVOA}
	x(z) = \sum_{n \in \ZZ} x_n z^{-n-1}, \quad
	x(z) y(w) \sim \frac{\killing{x}{y} \kk \wun}{(z-w)^2} + \frac{\comm{x}{y}(w)}{z-w}, \qquad x,y \in \slthree,
\end{equation}
where $\wun$ denotes the identity field.  For $\kk \neq -3$, the Sugawara construction gives a conformal structure for which the $x(z)$ with $x \in \slthree$ are Virasoro primary fields of conformal weight $1$.  Recalling that $h^3 = h^1 + h^2$, we have
\begin{multline} \label{eq:T}
  T(z) = \frac{1}{2(\kk+3)} \biggl[\frac{1}{3} \brac[\Big]{\no{h^1(z) h^1(z)} + \no{h^2(z) h^2(z)} + \no{h^3(z) h^3(z)}} \biggr. \\
  \biggl. - \pd h^1(z) - \pd h^2(z) - \pd h^3(z) + 2 \brac[\Big]{\no{e^1(z) f^1(z)} + \no{e^2(z) f^2(z)} + \no{e^3(z) f^3(z)}} \biggr]
\end{multline}
and the Virasoro modes defined by $T(z) = \sum_{n \in \ZZ} L_n z^{-n-2}$ have central charge
\begin{equation} \label{eq:c}
	\cc = \frac{8\kk}{\kk+3}.
\end{equation}

The \voa\ constructed on this level-$\kk$ parabolic Verma module is said to be \emph{universal}.  We shall denote this universal affine \voa\ by $\univk$.  It is simple unless the level $\kk \neq -3$ has the following form \cite{GorSim07}:
\begin{equation} \label{eq:ktuv}
	\kk+3 = \fracuv, \quad \uu \in \ZZ_{\ge 2},\ \vv \in \ZZ_{\ge 1},\ \gcd \set{\uu,\vv} = 1.
\end{equation}
If $\uu \in \ZZ_{\ge 3}$ in \eqref{eq:ktuv}, then $\kk$ is said to be an \emph{admissible level} \cite{KacMod88}.  When $\kk$ satisfies \eqref{eq:ktuv}, we shall denote the (unique) simple quotient of $\univk$ by $\minmoduv$ and refer to it as an $\slthree$ \emph{minimal model}.

\subsection{Automorphisms of $\aslthree$} \label{sec:AffAut}

The automorphisms $\omega \in \grp{D}_6$ of the root system $\roots$ of $\slthree$ all lift to automorphisms of $\aslthree$ that act trivially on $\CC[t,t^{-1}]$, $K$ and $L_0$.  We shall use the same symbols to denote these lifts as in the previous section.  Apart from these lifts, there is another important family of automorphisms of $\aslthree$ known as the \emph{spectral flow} automorphisms $\sfaut{\fcwt{}}$, parametrised by elements $\fcwt{}$ of the coweight lattice $\cwlat \subset \csub$.  More precisely, these automorphisms represent the translations in the extended affine Weyl group $\widetilde{\weyl} = \weyl \ltimes \cwlat$ and act on $\aslthree$ and the Virasoro modes explicitly as follows:
\begin{equation} \label{eq:sf}
	\begin{aligned}
		\sfmod{\fcwt{}}{e^\alpha_n} &= e^\alpha_{n-\pair{\alpha}{\fcwt{}}}, &
		\sfmod{\fcwt{}}{h_n} &= h_n - \killing{\fcwt{}}{h} \delta_{n,0} K, \\
		\sfmod{\fcwt{}}{K} &= K, &
		\sfmod{\fcwt{}}{L_n} &= L_n - \fcwt{n} + \tfrac{1}{2} \killing{\fcwt{}}{\fcwt{}} \delta_{n,0} K,
	\end{aligned}
	\qquad \fcwt{} \in \cwlat,\ \alpha \in \roots,\ h \in \csub,\ n \in \ZZ.
\end{equation}
Here, $\fcwt{n} = \fcwt{} \otimes t^n$ is the $n$-th mode of the affine field $\fcwt{}(z)$ corresponding to $\fcwt{} \in \cwlat \subset \csub$, as per \eqref{eq:UnivVOA}.  We note that the root system automorphisms $\omega \in \grp{D}_6$ and the spectral flow automorphisms $\sfaut{\fcwt{}}$, $\fcwt{} \in \cwlat$, satisfy
\begin{equation} \label{eq:sfD6rels}
  \sfaut{\fcwt{}} \sfaut{\fcwt{}'} = \sfaut{\fcwt{} + \fcwt{}'} \quad \text{and} \quad
  \omega \sfaut{\fcwt{}} \omega^{-1} = \sfaut{\func{\omega}{\fcwt{}}},
\end{equation}
where the action of $\grp{D}_6$ on $\cwlat \subset \csub$ is the usual one given in \eqref{eq:Waction} and \eqref{eq:daction}.

All of these $\aslthree$-automorphisms define automorphisms of the level-$\kk$ vertex algebras corresponding to $\univk$ and their simple quotients $\minmoduv$, though spectral flow does not preserve the conformal structure, hence does not give automorphisms of the \voas{} themselves.  Consequently, these automorphisms induce autoequivalences of the category $\awcatk$ of weight $\univk$-modules, which we identify with the category of smooth weight $\aslthree$-modules, and the full subcategory $\awcatuv$ of weight $\minmoduv$-modules (when $\kk$ satisfies \eqref{eq:ktuv}).  We shall again use the same symbols to denote an automorphism and the associated autoequivalence.  As in \cref{sec:LieAut}, the induced action of these automorphisms on the set of isomorphism classes of modules in $\awcatk$ or $\awcatuv$ satisfies \eqref{eq:sfD6rels}.

We record, for future convenience, how weights and conformal weights change under spectral flow.  Let $v$ be a weight vector in some level-$\kk$ $\aslthree$-module.  If its weight is $\nu$ and its conformal weight is $\Delta$, then the weight and conformal weight of $\sfmod{\fcwt{}}{v}$ are
\begin{equation} \label{eq:sfwts}
	\nu + \kk \fcwt{}^* \quad \text{and} \quad \Delta + \pair{\nu}{\fcwt{}} + \frac{1}{2} \killing{\fcwt{}}{\fcwt{}} \kk,
\end{equation}
respectively, where $\fcwt{}^* = \killing{\fcwt}{\blank} \in \dcsub$.  The weight follows from
\begin{equation}
	h_0 \sfmod{\fcwt{}}{v} = \sfmod{\fcwt{}}{\sfmod{-\fcwt{}}{h_0} v} = \sfmod{\fcwt{}}{(h_0 + \killing{\fcwt{}}{h} K) v} = \pair{\nu + \kk \fcwt{}^*}{h} \, \sfmod{\fcwt{}}{v}
\end{equation}
and a similar computation gives the conformal weight.

\subsection{Irreducible weight $\aslthree$-modules} \label{sec:AffIrr}

We can construct many families of smooth level-$\kk$ weight $\aslthree$-modules, hence $\univk$-modules, by affinising the $\slthree$-modules introduced in \cref{sec:Lie}.  Moreover, the affinisation of an irreducible $\slthree$-module will have a unique irreducible quotient.  We shall indicate this irreducible quotient with a hat: $\affine{\Mod{M}}$ is the irreducible quotient of the level-$\kk$ affinisation of $\Mod{M}$.  In this way, we arrive at irreducible smooth level-$\kk$ $\aslthree$-modules denoted by $\airr{\mu}$, $\asemi{\lambda}{[\mu]}$ and $\arel{\lambda}{[\mu]}$.  We shall refer to these as \hwms, \emph{semirelaxed} \hwms\ and \emph{relaxed} \hwms, respectively.\footnote{We remark that under the general definition in \cite{RidRel15}, all of these modules would be examples of \rhwms.  However, in this case the nomenclature introduced here is convenient and so we shall adopt it, hoping that no confusion will arise.}  We emphasise that the weights $\lambda$ and $\mu$ appearing in this notation are $\slthree$-weights: $\lambda, \mu \in \dcsub$.  They, of course, completely determine their affine counterparts $\affine{\lambda}$ and $\affine{\mu}$ because the latter are fixed to have level $\kk$.

Given any $\aslthree$-module $\affine{\Mod{N}}$, the vectors of minimal
conformal weight (should any exist) are \emph{relaxed} \hwvs{} in the sense
that they would be genuine \hwvs{} except that the requirement to be
annihilated by the $e^i_0$, $i=1,2,3$, has been relaxed, in other words removed.
Such vectors are also known as \emph{ground states}.  If $\affine{\Mod{N}}$ is a quotient of the affinisation of an irreducible $\slthree$-module $\Mod{M}$, then its space of ground states is called the \emph{top space} of $\affine{\Mod{N}}$.  It is, moreover, isomorphic to $\Mod{M}$ as an $\slthree$-module.  The conformal weight of the ground states of $\airr{\lambda}$, $\asemi{\lambda}{[\mu]}$ and $\arel{\lambda}{[\mu]}$ is
\begin{equation}
	\Delta_{\lambda} = \frac{\bilin{\lambda}{\lambda + 2 \wvec}}{2(\kk+3)}.
\end{equation}

The action of an $\slthree$-automorphism $\omega \in \grp{D}_6$ on the modules $\airr{\mu}$, $\asemi{\lambda}{[\mu]}$ and $\arel{\lambda}{[\mu]}$ of $\awcatk$ is easily deduced from the action on their respective top spaces $\firr{\mu}$, $\fsemi{\lambda}{[\mu]}$ and $\frel{\lambda}{[\mu]}$, discussed in \cref{sec:LieIrr}.  The action of the spectral flow functors is more interesting because, in general, $\sfaut{\fcwt{}}$ maps an irreducible weight module whose conformal weights are bounded below to an irreducible weight module whose conformal weights are not bounded below.

A weight module (with \fdim\ weight spaces) whose conformal weights are bounded below is said to be \emph{positive-energy}.  Thus, the full subcategories of positive-energy modules in $\awcatk$ and $\awcatuv$, respectively, are preserved by the action of $\omega$, but not by spectral flow.\footnote{There is a well known exception to this statement when $\uu \ge 3$ and $\vv=1$, so $\kk \in \NN$.  Then, the positive energy category coincides with $\awcat{\uu,1}$ and the category of integrable \hw{} $\aslthree$-modules; the latter is of course preserved by spectral flow.}  Spectral flow does, however, map smooth weight modules to smooth weight modules, so $\awcatk$ and $\awcatuv$ are closed under this action.  We can therefore obtain many new irreducible weight $\univk$- or $\minmoduv$-modules by applying spectral flow to the positive-energy ones already identified.

It is almost always the case that the result of applying spectral flow to a positive-energy module is a module that is not positive-energy.  However, it will be useful for what follows to know the (isomorphism class of the) result when it is positive-energy.  Let $\set{\fcwt{1},\fcwt{2}}$ denote the basis of $\cwlat$ dual to the simple root basis $\set{\sroot{1},\sroot{2}}$ of $\rlat$.
\begin{proposition} \label{prop:+ESF}
	\leavevmode
	\begin{enumerate}
		\item The spectral flows $\sfmod{\fcwt{1}}{\airr{\mu}}$ and $\sfmod{\fcwt{2}}{\airr{\mu}}$ are positive-energy if and only if $\mu \in \pwlat$.  In fact, \label{it:sffdim}
		\begin{equation}
			\begin{aligned}
				\sfmod{\fcwt{1}}{\airr{\mu}}
				&\cong \airr{\kk \fwt{1} + \weylaut{1} \weylmod{2}{\mu}}
				\cong \dynkmod{\airr{\kk \fwt{2} - \weylmod{1}{\mu}}} \\
				\text{and} \quad \sfmod{\fcwt{2}}{\airr{\mu}}
				&\cong \airr{\kk \fwt{2} + \weylaut{2} \weylmod{1}{\mu}}
				\cong \dynkmod{\airr{\kk \fwt{1} - \weylmod{2}{\mu}}}
			\end{aligned}
			\qquad \text{($\mu \in \pwlat$)}.
		\end{equation}
		\item The spectral flows $\sfmod{\fcwt{1}-\fcwt{2}}{\airr{\mu}}$ and $\sfmod{\fcwt{2}-\fcwt{1}}{\airr{\mu}}$ are positive-energy if and only if $\mu_1 \in \NN$ and $\mu_2 \in \NN$, respectively.  Moreover, \label{it:sfdeghw}
		\begin{equation}
			\begin{aligned}
				\sfmod{\fcwt{1}-\fcwt{2}}{\airr{\mu}}
				&\cong \weylmod{2}{\airr{\kk \fwt{2} + \weylaut{2} \weylmod{1}{\mu}}}
				\cong \dynkaut \weylmod{1}{\airr{\kk \fwt{1} - \weylmod{2}{\mu}}} &&& &\text{($\mu_1 \in \NN$)} \\
				\text{and} \quad \sfmod{\fcwt{2}-\fcwt{1}}{\airr{\mu}}
				&\cong \weylmod{1}{\airr{\kk \fwt{1} + \weylaut{1} \weylmod{2}{\mu}}}
				\cong \dynkaut \weylmod{2}{\airr{\kk \fwt{2} - \weylmod{1}{\mu}}} &&& &\text{($\mu_2 \in \NN$)}.
			\end{aligned}
		\end{equation}
		\item The spectral flows $\sfmod{-\fcwt{1}}{\airr{\mu}}$ and $\sfmod{-\fcwt{2}}{\airr{\mu}}$ are always positive-energy: \label{it:sfhw}
		\begin{equation}
			\begin{aligned}
				\sfmod{-\fcwt{1}}{\airr{\mu}}
				&\cong \weylaut{1} \weylmod{2}{\airr{\kk \fwt{2} + \weylaut{2} \weylmod{1}{\mu}}}
				\cong \conjaut \weylmod{2}{\airr{\kk \fwt{1} - \weylmod{2}{\mu}}} \\
				\text{and} \quad
				\sfmod{-\fcwt{2}}{\airr{\mu}}
				&\cong \weylaut{2} \weylmod{1}{\airr{\kk \fwt{1} + \weylaut{1} \weylmod{2}{\mu}}}
				\cong \conjaut \weylmod{1}{\airr{\kk \fwt{2} - \weylmod{1}{\mu}}}.
			\end{aligned}
		\end{equation}
		\item The spectral flow $\sfmod{\fcwt{}}{\asemi{\lambda}{[\mu]}}$ is positive-energy if and only if $\fcwt{} = 0$ or $\fcwt{} = -\fcwt{2}$.  Indeed, \label{it:sfsrel}
		\begin{equation}
			\sfmod{-\fcwt{2}}{\asemi{\lambda}{[\mu]}} \cong \conjaut \weylmod{1}{\asemi{\kk \fwt{2} - \weylmod{1}{\lambda}}{[\kk \fwt{2} - \weylmod{1}{\mu}]}}.
		\end{equation}
		\item The spectral flow $\sfmod{\fcwt{}}{\arel{\lambda}{[\mu]}}$ is positive-energy if and only if $\fcwt{} = 0$. \label{it:sfrel}
	\end{enumerate}
\end{proposition}
\begin{proof}
	We start with \ref{it:sfrel}.  The top space of $\arel{\lambda}{[\mu]}$ is an irreducible dense $\slthree$-module with weight support $\mu + \rlat$.  A ground state $v \in \arel{\lambda}{[\mu]}$ then has weight $\mu + m_1 \sroot{1} + m_2 \sroot{2}$, for some $m_1, m_2 \in \ZZ$, and its conformal weight is $\Delta$, independent of $m_1$ and $m_2$.  By \eqref{eq:sfwts}, the conformal weight of $\sfmod{\fcwt{}}{v}$ is then $\Delta' + m_1 \pair{\sroot{1}}{\fcwt{}} + m_2 \pair{\sroot{2}}{\fcwt{}}$, where $\Delta' = \Delta + \pair{\mu}{\fcwt{}} + \frac{1}{2} \killing{\fcwt{}}{\fcwt{}} \kk$.  This is unbounded below as $m_1$ and $m_2$ range over $\ZZ$ unless $\pair{\sroot{1}}{\fcwt{}} = \pair{\sroot{2}}{\fcwt{}} = 0$.

	The proof of \ref{it:sfsrel} starts out similarly, but because a weight of the top space of $\asemi{\lambda}{[\mu]}$ again has the form $\mu + m_1 \sroot{1} + m_2 \sroot{2}$ with $m_1 \in \ZZ$, but now with $m_2 \in \ZZ_{\le0}$, the conformal weights of the spectral flow of the top space are unbounded below unless $\pair{\sroot{1}}{\fcwt{}} = 0$ and $\pair{\sroot{2}}{\fcwt{}} \le 0$, that is unless $\fcwt{} = -p \fcwt{2}$ for some $p \in \NN$.

	Take $\fcwt{} = -\fcwt{2}$ and let $v_{m_1,m_2} \in \asemi{\lambda}{[\mu]}$ denote a state of minimal conformal weight among those whose weight is $\mu + m_1 \sroot{1} + m_2 \sroot{2}$, where $m_1,m_2 \in \ZZ$.  For $m_2 \le 0$, these are the ground states and we denote their common conformal weight by $\Delta$.  The conformal weight of $\sfmod{-\fcwt{2}}{v_{m_1,m_2}}$, $m_2\le0$, is then $\Delta' - m_2$, where $\Delta' = \Delta - \pair{\mu}{\fcwt{2}} + \frac{1}{2} \killing{\fcwt{2}}{\fcwt{2}} \kk$.  This is obviously minimised when $m_2=0$.  Of all the ground states of $\asemi{\lambda}{[\mu]}$, only the $v_{m_1,0}$ have spectral flow images that might be ground states in $\sfmod{-\fcwt{2}}{\asemi{\lambda}{[\mu]}}$.

	Consider next the state $v_{m_1,m_2}$, where $m_1 \in \ZZ$ and $m_2 \in \ZZ_{>0}$.  This is not a ground state of $\asemi{\lambda}{[\mu]}$, hence its conformal weight is $\Delta + n$, for some $n>0$.  In fact, we have $n \ge m_2$ because $v_{m_1,m_2}$ can only be obtained from the $v_{m_1,0}$ by acting with modes that include at least $m_2$ of the $e^2_{-r}$ or $e^3_{-s}$, $r,s>0$.  The conformal weight of $\sfmod{-\fcwt{2}}{v_{m_1,m_2}}$ is then $\Delta' + n - m_2$.  As this is bounded below by the conformal weight $\Delta'$ of the $\sfmod{-\fcwt{2}}{v_{m_1,0}}$, it follows that $\sfmod{-\fcwt{2}}{\asemi{\lambda}{[\mu]}}$ is positive-energy and the $\sfmod{-\fcwt{2}}{v_{m_1,0}}$ are indeed ground states.

	For $m_2>0$, $v_{m_1,m_2}$ also maps to a ground state in $\sfmod{-\fcwt{2}}{\asemi{\lambda}{[\mu]}}$ if and only if $n=m_2$.  However, the ground states must form an irreducible $\slthree$-module, so combining the results thus far with the classification in \cref{sec:LieIrr} forces the $v_{m_1,m_2}$ to map to ground states, hence each has $n=m_2$, for all $m_1 \in \ZZ$ and $m_2 \in \ZZ_{>0}$.  The top space of $\sfmod{-\fcwt{2}}{\asemi{\lambda}{[\mu]}}$ therefore has weight support $\mu + \ZZ \sroot{1} + \NN \sroot{2}$.  To get a standard semidense weight support, we must either conjugate or conjugate and twist by $\weylaut{1}$.  The latter is more convenient, as we shall see, so we conclude that
	\begin{equation} \label{eq:wctwist}
		\weylaut{1} \conjaut \sfmod{-\fcwt{2}}{\asemi{\lambda}{[\mu]}} \cong \asemi{\lambda'}{[\mu']},
	\end{equation}
	for some $\lambda' \in \dcsub$ and $[\mu'] \in (\lambda' + \CC \sroot{1}) / \ZZ \sroot{1}$.

	To compute $[\mu']$, we note that the above identification of ground states implies that it is obtained from $[\mu]$ and \eqref{eq:sfwts} as follows:
	\begin{equation}
		[\mu'] = [\weylaut{1} \conjmod{\mu - \kk \fcwt{2}^*}] = [\weylmod{1}{\kk \fwt{2} - \mu}] = [\kk \fwt{2} - \weylmod{1}{\mu}].
	\end{equation}
	To determine $\lambda'$, note that $\lambda$ is the highest weight of a \hw\ submodule of the irreducible semisimple coherent family of $\gltwo$-modules from which $\fsemi{\lambda}{[\mu]}$ was constructed.  It follows that there is a vector $v$ of weight $\lambda$ in $\asemi{\lambda}{[\lambda]}$ that is annihilated by $e^1_0$.  Spectral flow, conjugating and twisting by $\weylaut{1}$ give
	\begin{equation}
		e^1_0 \weylaut{1} \conjaut \sfmod{-\fcwt{2}}{v} \propto \weylaut{1} f^1_0 \conjaut \sfmod{-\fcwt{2}}{v} \propto \weylaut{1} \conjaut e^1_0 \sfmod{-\fcwt{2}}{v} = \weylaut{1} \conjaut \sfmod{-\fcwt{2}}{e^1_0 v} = 0
	\end{equation}
	(note that $\weylmod{1}{e^1_0}$ and $\conjmod{e^1_0}$ are proportional, but not necessarily equal, to $f^1_0$).  It follows that $\weylaut{1} \conjaut \sfmod{-\fcwt{2}}{v}$ generates an \infdim\ \hw\ $\gltwo$-submodule of $\asemi{\lambda'}{[\lambda']}$.  We may therefore identify $\lambda'$ with its highest weight:
	\begin{equation}
		\lambda' = \kk \fwt{2} - \weylmod{1}{\lambda}.
	\end{equation}
	It is easy to check that $\mu' \in \lambda' + \CC \sroot{1}$ because $\mu \in \lambda + \CC \sroot{1}$.  Note that if we had omitted the $\weylaut{1}$-twist in \eqref{eq:wctwist}, then applying spectral flow and conjugation would have resulted in an \infdim\ \lw\ submodule.

	Finally, take $\fcwt{} = -p \fcwt{2}$ with $p \in \ZZ_{>1}$.  As we have established above that $n=m_2$ whenever $m_2>0$, the $\sfmod{-p \fcwt{2}}{v_{m_1,m_2}}$ with $m_2>0$ have conformal weights $\Delta' - (p-1) m_2$.  This is unbounded below as $m_2 \to \infty$, hence $\sfmod{-p \fcwt{2}}{\asemi{\lambda}{[\mu]}}$ is not positive-energy.  This completes the proof of \ref{it:sfsrel}.

	The \hw\ cases \ref{it:sffdim}--\ref{it:sfhw} follow similarly, so we only outline the steps.  The setup for these cases is $\dynkaut$-invariant, so we only need discuss one identification each.  First, conformal weight considerations show that $\sfmod{-\fcwt{1}}{\airr{\mu}}$ is always positive-energy and explicit calculation shows that the image of the \hwv\ of $\airr{\mu}$ under $\weylaut{2} \weylaut{1} \sfaut{-\fcwt{1}}$ is a \hwv.  One therefore just needs to calculate its weight.  On the other hand, the same conformal weight considerations require $\mu_1 \in \NN$ when $\fcwt{} = \fcwt{1}-\fcwt{2}$.  This time, the \hw\ condition is preserved by $\weylaut{2} \sfaut{\fcwt{1}-\fcwt{2}} \weylaut{1}$.  Finally, $\sfmod{\fcwt{1}}{\airr{\mu}}$ is positive-energy if and only if $\mu \in \pwlat$ and $\sfaut{\fcwt{1}} \weylaut{1} \weylaut{2}$ allows us to identify the new highest weight.
\end{proof}
\noindent As we shall see, one may need to iterate the identifications given above for the spectral flows of the $\airr{\mu}$ in order to find all cases in which the spectral flow of an irreducible \hw\ $\aslthree$-module is again positive-energy.

\section{Admissible-level $\slthree$ minimal models $\minmoduv$} \label{sec:MinMod}

As noted in \cref{sec:AffAlg}, the level $\kk$ of a vertex algebra associated with $\slthree$ is said to be admissible if it has the form \eqref{eq:ktuv} with $\uu \ge 3$.  In this \lcnamecref{sec:MinMod}, we review the classification of \hw\ $\minmoduv$-modules, when $\kk$ is admissible, and extend it to all relaxed and semirelaxed \hw\ $\minmoduv$-modules.

\subsection{Irreducible \hw\ $\minmoduv$-modules} \label{sec:MinModHW}

For each $\ell \in \NN$, let $\apwlat{\ell}$ denote the set of dominant-integral level-$\ell$ weights of $\aslthree$, that is the set of weights $\affine{\lambda}$ whose Dynkin labels satisfy $\lambda_i \in \NN$ and $\lambda_0 + \lambda_1 + \lambda_2 = \ell$.\footnote{We shall generally drop the hat from affine Dynkin labels, trusting that this will not cause confusion.}  In what follows, we shall sometimes identify $\affine{\lambda}$ with its triple $(\lambda_0, \lambda_1, \lambda_2)$ of Dynkin labels when convenient.  Given an affine weight $\affine{\lambda}$, we write $\lambda$ for its projection onto $\dcsub$, so $\lambda$ is obtained from $\affine{\lambda}$ by setting $\lambda_0$ to $0$.  Of course, once a level has been fixed, $\affine{\lambda}$ is completely determined by $\lambda$.

In \cite{KacMod88}, Kac and Wakimoto introduced a finite set $\admuv$ of \emph{admissible weights} for each admissible level $\kk$.  For $\aslthree$, an equivalent characterisation \cite{KacCla89} of these weights involves writing them in the form
\begin{equation} \label{eq:admwt}
	\affine{\lambda} = \weylaut{} \cdot \brac*{\affine{\lambda}^I - \fracuv \, \affine{\lambda}^{F,\weylaut{}}},
\end{equation}
where $\weylaut{} \in \set{\wun, \weylaut{1}} \subset \weyl$,
$\affine{\lambda}^I \in \apwlat{\uu-3}$ is the \emph{integral part} of
$\affine{\lambda}$, $\affine{\lambda}^{F,\weylaut{}} \in \apwlat{\vv-1}$ is
the \emph{fractional part} of $\affine{\lambda}$, and
$\affine{\lambda}^{F,\weylaut{1}}$ has Dynkin label $\lambda^{F,\weylaut{1}}_1
\ne 0$.  We will describe an admissible weight $\affine{\lambda}$ as being of
either $\weylaut{} = \wun$ or $\weylaut{} = \weylaut{1}$ type, according as to
which $\weylaut{}$ is used in \eqref{eq:admwt}.  The sets, $\admuv^{\wun}$ and
$\admuv^{\weylaut{1}}$, of $\weylaut{} = \wun$ and $\weylaut{} = \weylaut{1}$ type admissible weights, respectively, are disjoint \cite{KacCla89}: $\admuv = \admuv^{\wun} \sqcup \admuv^{\weylaut{1}}$.

The classification of admissible-level \hw\ $\minmoduv$-modules is a special case of a general \hw\ classification result of Arakawa.
\begin{theorem}[{\cite[Main~Thm.]{AraRat16}}] \label{thm:hwclass}
	When $\kk$ is admissible, the irreducible level-$\kk$ \hw\ $\aslthree$-module $\airr{\lambda}$ is an $\minmoduv$-module if and only if $\affine{\lambda} \in \admuv$.  Moreover, every \hw\ $\minmoduv$-module is irreducible.
\end{theorem}
\noindent It follows that there are $\abs*{\admuv} = \frac{1}{2} (\uu-1) (\uu-2) \vv^2$ irreducible \hw\ $\minmoduv$-modules, up to isomorphism.

When $\uu \ge 3$ and $\vv = 1$, so $\kk \in \NN$, the admissibility conditions above reduce to $\affine{\lambda}^I \in \apwlat{\uu-3}$ and $\affine{\lambda}^{F, \weylaut{}} = 0$, hence $\admwts{\uu}{1}^{\wun} = \apwlat{\kk}$ and $\admwts{\uu}{1}^{\weylaut{1}} = \varnothing$.  In this case, Arakawa's classification reproduces the well known result \cite{GepStr86,FreVer92} that the irreducible modules over the rational \wzw\ \voa\ $\minmod{\uu}{1}$ are the integrable \hw\ ones.  We will therefore be chiefly interested in the case $\vv \ge 2$ in what follows.

Note that the Dynkin labels of a weight $\affine{\lambda} \in \admuv^{\wun}$ have the form
\begin{subequations} \label{eq:admdynk}
	\begin{equation} \label{eq:adm1dynk}
		\lambda_i = \lambda^I_i - \fracuv \, \lambda^{F,\wun}_i, \quad i=0,1,2.
	\end{equation}
	Since $\gcd \set{\uu,\vv} = 1$ and $0 \le \lambda^{F,\wun}_i \le
        \vv-1$, it follows for $\weylaut{} = \wun$ type admissible weights that $\lambda_i \in \NN$ if and only if $\lambda^{F,\wun}_i = 0$.  Similarly, the Dynkin labels of $\affine{\lambda} \in \admuv^{\weylaut{1}}$ have the form
	\begin{equation} \label{eq:admwdynk}
		\lambda_0 = 
		\uu - 2 - \lambda^I_2 - \fracuv \brac*{\vv - 1 - \lambda^{F,\weylaut{1}}_2}, \quad
		\lambda_1 = 
		\uu - 2 - \lambda^I_1 - \fracuv \brac*{\vv - \lambda^{F,\weylaut{1}}_1}, \quad
		\lambda_2 = 
		\uu - 2 - \lambda^I_0 - \fracuv \brac*{\vv - 1 - \lambda^{F,\weylaut{1}}_0}.
	\end{equation}
\end{subequations}
Since $1 \le \lambda^{F,\weylaut{1}}_1 \le \vv-1$ and $0 \le \lambda^{F,\weylaut{1}}_0, \lambda^{F,\weylaut{1}}_2 \le \vv-2$, a $\weylaut{} = \weylaut{1}$ admissible weight never has an integral Dynkin label.
\begin{proposition} \label{prop:fdimtop}
	When $\kk$ is admissible, the $\minmoduv$-module $\airr{\lambda}$ has a \fdim\ top space if and only if $\affine{\lambda} \in \admuv^{\wun}$ and $\lambda^{F,\wun}_1 = \lambda^{F,\wun}_2 = 0$.
\end{proposition}
\noindent There are therefore $\frac{1}{2} (\uu-1)(\uu-2)$ irreducible \hw\ $\minmoduv$-modules with \fdim\ top spaces, up to isomorphism.

\subsection{Irreducible \srhw\ $\minmoduv$-modules} \label{sec:MinModSHW}

For each $\uu \ge 3$ and $\vv \ge 2$, we shall show that $\minmoduv$ admits \srhwms.  This will follow from the algorithm, recently detailed and illustrated in \cite{KawRel19}, for classifying the (semi)\rhwms\ of a general affine \voa, given the \hw\ classification.  Here, we determine the result of this algorithm under the assumption that the parabolic is $\para$.  The analogous classification for the parabolic $\dynkmod{\para}$ follows immediately.

Since \srhwms\ have semidense top spaces, we project each of the weights $\lambda \in \dcsub$ corresponding to $\affine{\lambda} \in \admuv$ onto the weight space of the simple ideal $\alg{s} \cong \sltwo$ of the Levi factor $\levi \cong \gltwo$.  If the result $\pi(\lambda)$ is the highest weight of an \infdim\ irreducible \hw\ $\alg{s}$-module, then it determines an irreducible semisimple coherent family of $\alg{s}$-modules as in \cref{sec:LieIrr}.  Tensoring with the irreducible $\glone$-module of weight $\lambda - \pi(\lambda)$, we obtain an irreducible semisimple coherent family $\coh{\lambda}$ of $\levi$-modules.  When the direct summand $\coh{\lambda}_{[\mu]}$, $\mu \in \lambda + \CC \sroot{1}$, is irreducible, inducing and taking irreducible quotients results in the irreducible semidense $\slthree$-module $\fsemi{\lambda}{[\mu]}$ (\cref{sec:LieIrr}).  The irreducible quotients $\asemi{\lambda}{[\mu]}$ of the affinisations then exhaust the irreducible \srhw\ $\minmoduv$-modules for the parabolic $\para$ (up to isomorphism) \cite{KawRel19}.  However, one still has to determine when $\coh{\lambda}_{[\mu]}$ is irreducible.

As the subalgebra $\levi \subset \slthree$ has positive root vector $e^1$, the orthogonal projection is given by
\begin{equation}
	\pi(\lambda) = \frac{\bilin{\lambda}{\sroot{1}}}{\norm{\sroot{1}}^2} \sroot{1} = \lambda_1 \, \frac{1}{2} \sroot{1}.
\end{equation}
This then defines an \infdim\ irreducible \hw\ $\alg{s}$-module, hence a family of \srhw\ $\minmoduv$-modules, if and only if $\lambda_1 \notin \NN$.  Comparing with \eqref{eq:admdynk}, we learn that the $\affine{\lambda} \in \admuv^{\wun}$ define semirelaxed modules if and only if $\lambda^{F,\wun}_1 \ne 0$ and the $\affine{\lambda} \in \admuv^{\weylaut{1}}$ always define semirelaxed modules.  We shall set
\begin{equation}
	\sadmuv^1 = \set*{\affine{\lambda} \in \admuv^{\wun} \st \lambda^{F,\wun}_1 \ne 0} \quad \text{and} \quad \sadmuv = \sadmuv^1 \sqcup \admuv^{\weylaut{1}},
\end{equation}
noting also that $\sadmwts{\uu}{1} = \varnothing$.

We may therefore label the \srhw\ $\minmoduv$-modules by weights $\lambda \in \dcsub$, corresponding to $\affine{\lambda} \in \sadmuv$, and cosets $[\mu] \in \dcsub / \ZZ \sroot{1}$, $\mu \in \lambda + \CC \sroot{1}$.  However, this labelling is redundant because different $\lambda$ may correspond to the same irreducible semisimple coherent family of $\levi$-modules.  As the relevant $\alg{s}$-weights $\pi(\lambda)$ are never integral, there are always two such $\lambda$ and they are related by the shifted action of the Weyl reflection of $\alg{s}$.  In particular, if $\affine{\lambda} \in \sadmuv^1$, then it corresponds to the same semirelaxed modules as $\weylaut{1} \cdot \affine{\lambda} \in \admuv^{\weylaut{1}}$.  Since $\weylaut{1} \cdot$ gives a bijection between $\sadmuv^1$ and $\admuv^{\weylaut{1}}$, we may thus restrict attention to the former set.

Finally, recall that $\asemi{\lambda}{[\mu]}$ is irreducible if and only if the direct summand $\coh{\lambda}_{[\mu]}$ of the coherent family $\coh{\lambda}$ of $\levi$-modules is irreducible.  As this family is semisimple, $\coh{\lambda}_{[\mu]}$ being reducible is equivalent to it having an \infdim\ \hw\ submodule.  By the above, we must then have $[\mu] = [\lambda]$ or $[\mu] = [\weylaut{1} \cdot \lambda]$.  We emphasise that $[\lambda]$ and $[\weylaut{1} \cdot \lambda]$ are distinct cosets because $\affine{\lambda} \in \sadmuv$ implies that $\lambda_1 \notin \ZZ$.

We have therefore arrived at the classification of irreducible \srhw\ $\minmoduv$-modules.
\begin{proposition} \label{prop:srhwclass}
	When $\kk$ is admissible, every irreducible \srhw\ $\minmoduv$-module is isomorphic to one, and only one, of the form $\asemi{\lambda}{[\mu]}$, where $\affine{\lambda} \in \sadmuv^1$ and $\mu \in \lambda + \CC \sroot{1}$ satisfies $[\mu] \ne [\lambda], [\weylaut{1} \cdot \lambda]$.
\end{proposition}
\noindent It follows that $\minmoduv$ admits $\abs{\sadmuv^1} = \frac{1}{4} (\uu-1)(\uu-2) \vv (\vv-1)$ families of \srhwms.  These families correspond to the parabolic $\para$ that distinguishes $\sroot{1}$.  To obtain the analogous classification for $\dynkmod{\para}$, which distinguishes $\sroot{2}$, we need only apply $\dynkaut$ to the modules of \cref{prop:srhwclass}.

\subsection{Irreducible \rhw\ $\minmoduv$-modules} \label{sec:MinModRHW}

It is straightforward to use the algorithm of \cite{KawRel19} to generalise the preceding analysis to \rhwms.  Again, we shall show the existence of families of such $\minmoduv$-modules for all $\uu \ge 3$ and $\vv \ge 2$.

Relaxed \hwms\ have dense top spaces, so the appropriate Levi factor in this case is $\slthree$ itself and no projection is required.  We thus check, for each $\affine{\lambda} \in \admuv$, if the irreducible \hw\ $\slthree$-module $\firr{\lambda}$ is bounded (\cref{sec:LieIrr}).  If so, then it is a submodule of an irreducible semisimple coherent family $\coh{\lambda}$ of $\slthree$-modules.  When irreducible, the direct summand $\coh{\lambda}_{[\mu]}$, $[\mu] \in \dcsub / \rlat$, may be identified with the irreducible dense $\slthree$-module $\frel{\lambda}{[\mu]}$.  Moreover, the simple quotients of the affinisations $\arel{\lambda}{[\mu]}$ of these irreducibles exhaust the irreducible \rhw\ $\minmoduv$-modules \cite{KawRel19}.  Again, one still needs to determine which $\coh{\lambda}_{[\mu]}$ are irreducible.

Mathieu has classified bounded \hwms\ for all simple Lie algebras \cite{MatCla00}.  For $\slthree$, it turns out that $\firr{\lambda}$ is bounded if and only if at least one of the following three conditions are met:
\begin{subequations} \label{eq:hwbounded}
	\begin{alignat}{2}
		\lambda_1 &\in \NN \quad \text{and} \quad &\lambda_2 &\notin \NN, \label{eq:hwb1} \\
		\lambda_1 &\notin \NN \quad \text{and} \quad &\lambda_2 &\in \NN, \label{eq:hwb2} \\
		\lambda_1, \lambda_2 &\notin \NN \quad \text{and} \quad &\lambda_1 + \lambda_2 &\in \ZZ_{\ge-1}. \label{eq:hwb3}
	\end{alignat}
\end{subequations}
For $\affine{\lambda} \in \admuv^{\wun}$, \eqref{eq:adm1dynk} implies that $\firr{\lambda}$ is bounded if and only if $\lambda^{F,\wun}_1 = 0$ and $\lambda^{F,\wun}_2 \ne 0$, by \eqref{eq:hwb1}, or $\lambda^{F,\wun}_1 \ne 0$ and $\lambda^{F,\wun}_2 = 0$, by \eqref{eq:hwb2}.  The fact that \eqref{eq:hwb3} is never satisfied is easy to check: $\lambda^{F,\wun}_1, \lambda^{F,\wun}_2 \ne 0$ means $\lambda^{F,\wun}_0 \ne \vv-1$, hence
\begin{equation}
	\lambda_1 + \lambda_2 = \uu - 3 - \lambda^I_0 - \fracuv \brac*{\vv - 1 - \lambda^{F,\wun}_0} \notin \ZZ.
\end{equation}
For $\affine{\lambda} \in \admuv^{\weylaut{1}}$, a similar analysis using \eqref{eq:admwdynk} concludes that \eqref{eq:hwb1} and \eqref{eq:hwb2} are never satisfied, while \eqref{eq:hwb3} is satisfied if and only if $\lambda^{F,\weylaut{1}}_2 = 0$.

There are therefore three disjoint subsets of $\admuv$ which correspond to irreducible semisimple coherent families of $\slthree$-modules.  We let $\radmuv = \radmuv^1 \sqcup \radmuv^2 \sqcup \radmuv^3$, with
\begin{equation} \label{eq:defradmuv}
	\begin{gathered}
		\radmuv^1 = \set*{\affine{\lambda} \in \admuv^{\wun} \st \lambda^{F,\wun}_1 = 0\ \text{and}\ \lambda^{F,\wun}_2 \ne 0}, \quad
		\radmuv^2 = \set*{\affine{\lambda} \in \admuv^{\wun} \st \lambda^{F,\wun}_1 \ne 0\ \text{and}\ \lambda^{F,\wun}_2 = 0} \\
		\text{and} \quad
		\radmuv^3 = \set*{\affine{\lambda} \in \admuv^{\weylaut{1}} \st \lambda^{F,\weylaut{1}}_2 = 0},
	\end{gathered}
\end{equation}
noting again that $\radmwts{\uu}{1} = \varnothing$.  As in the previous \lcnamecref{sec:MinModSHW}, this description is redundant.  None of the $\slthree$-weights associated with these subsets are integral, hence the map from weights to coherent families is $3$ to $1$ \cite{MatCla00}.  Moreover, weights giving isomorphic coherent families are once again related by the shifted action of the Weyl group.  In particular, applying $\weylaut{1} \cdot$ to a weight $\affine{\lambda} \in \radmuv^2$ results in a weight in $\radmuv^3$ and applying $\weylaut{2} \cdot$ to the result moreover gives a weight in $\radmuv^1$.

We may therefore restrict $\affine{\lambda}$ to the subset $\radmuv^2$.\footnote{As we shall see, this arbitrary choice is convenient because the subsets of admissible weights classifying the relaxed, semirelaxed and \hw\ $\minmoduv$-modules then satisfy $\radmuv^2 \subseteq \sadmuv^1 \subset \admuv$.}  This gives us a classification of families of \rhw\ $\minmoduv$-modules.  It only remains to determine when these modules are irreducible; equivalently, to determine when $\coh{\lambda}_{[\mu]}$ is irreducible.  This latter problem was also solved quite generally by Mathieu \cite{MatCla00}.  His result states that $\coh{\lambda}_{[\mu]}$ is irreducible for all $[\mu] \in \brac*{\dcsub / \rlat} \setminus \sing{\lambda}$, where $\sing{\lambda}$ is the image of a union of certain codimension-$1$ affine subspaces in $\dcsub$.  For $\slthree$, $\sing{\lambda}$ is thus a union of curves and they may be identified as follows:
\begin{itemize}
	\item Determine the image in $\dcsub / \rlat$ of the shifted Weyl orbit of $\lambda$.  We take $\affine{\lambda} \in \radmuv^2$, so this image is easily checked to be $\set[\big]{[\lambda], [\weylaut{1} \cdot \lambda], [\weylaut{2} \weylaut{1} \cdot \lambda]}$.
	\item For each element $[\mu]$ in this image, determine the positive roots $\alpha$ such that $\pair{\mu}{\co{\alpha}} \notin \ZZ$.  For $[\lambda]$, this is $\sroot{1}$ and $\sroot{3}$; for $[\weylaut{1} \cdot \lambda]$, this is $\sroot{1}$ and $\sroot{2}$; while for $[\weylaut{2} \weylaut{1} \cdot \lambda]$, this is $\sroot{2}$ and $\sroot{3}$.
	\item Given all such $[\mu]$ and $\alpha$, $\sing{\lambda}$ is the union of all the $[\mu + \CC \alpha]$.  This appears to be the union of six curves, but in fact each curve appears twice in this description.  We end up with
	\begin{equation} \label{eq:defsing}
		\sing{\lambda} = [\lambda + \CC \sroot{1}] \cup [\weylaut{1} \cdot \lambda + \CC \sroot{2}] \cup [\lambda + \CC \sroot{3}].
	\end{equation}
\end{itemize}

This then completes the classification of irreducible \rhw\ $\minmoduv$-modules.
\begin{proposition} \label{prop:rhwclass}
	When $\kk$ is admissible, every irreducible \rhw\ $\minmoduv$-module is isomorphic to one, and only one, of the form $\arel{\lambda}{[\mu]}$, where $\affine{\lambda} \in \radmuv^2$ and $[\mu] \in \brac*{\dcsub / \rlat} \setminus \sing{\lambda}$.
\end{proposition}
\noindent This particular characterisation of these relaxed modules also appears in \cite[Thm.~1.1(a)]{KawRel20}.  It follows that $\minmoduv$ admits $\abs{\radmuv^2} = \frac{1}{4} (\uu-1)(\uu-2) (\vv-1)$ families of \rhwms.

\subsection{Classification summary} \label{sec:MinModClass}

We summarise the classification results thus far for convenience and note some relevant information about the twists of the irreducible modules being classified (\cref{sec:LieIrr}).  Recall the definitions of the sets $\admuv$, $\sadmuv$, $\radmuv$ and $\sing{\lambda}$ given above.  We illustrate the first three schematically in \cref{fig:admissibles}.%
\begin{figure}
	\begin{tikzpicture}[scale=1,>=latex]
		\draw[->] (0,0) -- (90:1) node[left] {$\lambda^{F,\weylaut{}}_2$};
		\draw[->] (0,0) -- (30:1) node[below right] {$\lambda^{F,\weylaut{}}_1$};
		\begin{scope}[shift={(-7,-0.5)}]
			\node at (1.75,2.5) {$\admuv$};
			\draw[fillcolour] (90:2) -- (0:0) -- (30:2) -- cycle;
			\node[blue] at (1,1.75) {$\admuv^{\wun}$};
			\node at (0.87,-0.3) {$\scriptstyle \weylaut{} = \wun$};
			\begin{scope}[shift={(2.0,0)}]
				\draw (90:2) -- (0:0) --
                                coordinate[pos=0.1](A1) (30:2) --
                                coordinate[pos=0.9](A2) (90:2); 
				\fill[fillcolour] (A2) -- (A1) -- (30:2) -- cycle;
				\node[blue] at (1,1.75) {$\admuv^{\weylaut{1}}$};
				\node at (0.87,-0.3) {$\scriptstyle \weylaut{} = \weylaut{1}$};
			\end{scope}
		\end{scope}
		\begin{scope}[shift={(-1.5,2)}]
			\node at (1.75,2.5) {$\sadmuv$};
			\draw (90:2) -- (0:0) -- coordinate[pos=0.1](B1)
                        (30:2) -- coordinate[pos=0.9](B2) (90:2); 
			\fill[fillcolour] (B2) -- (B1) -- (30:2) -- cycle;
			\node[blue] at (1,1.75) {$\sadmuv^1$};
			\node at (0.87,-0.3) {$\scriptstyle \weylaut{} = \wun$};
			\begin{scope}[shift={(2.0,0)}]
				\draw (90:2) -- (0:0) --
                                coordinate[pos=0.1](B3) (30:2) --
                                coordinate[pos=0.9](B4) (90:2); 
				\fill[fillcolour] (B4) -- (B3) -- (30:2) -- cycle;
				\node[blue] at (1,1.75) {$\admuv^{\weylaut{1}}$};
				\node at (0.87,-0.3) {$\scriptstyle \weylaut{} = \weylaut{1}$};
			\end{scope}
		\end{scope}
		\begin{scope}[shift={(-1.5,-3)}]
			\node at (1.75,2.5) {$\dynkmod{\sadmuv}$};
			\draw (90:2) -- coordinate[pos=0.9](B5) (0:0) --
                        (30:2) -- coordinate[pos=0.1](B6) (90:2); 
			\fill[fillcolour] (90:2) -- (B5) -- (B6) -- cycle;
			\node[blue] at (1,1.85) {$\dynkmod{\sadmuv^1}$};
			\node at (0.87,-0.3) {$\scriptstyle \weylaut{} = \wun$};
			\begin{scope}[shift={(2.0,0)}]
				\draw (90:2) -- (0:0) --
                                coordinate[pos=0.1](B3) (30:2) --
                                coordinate[pos=0.9](B4) (90:2); 
				\fill[fillcolour] (B4) -- (B3) -- (30:2) -- cycle;
				\node[blue] at (1,1.75) {$\admuv^{\weylaut{1}}$};
				\node at (0.87,-0.3) {$\scriptstyle \weylaut{} = \weylaut{1}$};
			\end{scope}
		\end{scope}
		\begin{scope}[shift={(4,-0.5)}]
			\node at (1.75,2.5) {$\radmuv$};
			\draw (90:2) -- coordinate[pos=0.9](C1) (0:0) -- coordinate[pos=0.1](C2) (30:2) -- cycle;
			\draw (90:2) -- node[pos=0.5,left,blue] {$\radmuv^1$} (C1) (0:0);               
			\draw[drawcolour] (90:2) -- (C1) (0:0);                                         
			\draw (C2) -- node[pos=0.55,below,blue] {$\quad\radmuv^2$} (30:2) (0:0);
			\draw[drawcolour] (C2) -- (30:2) (0:0);
			\node at (0.87,-0.3) {$\scriptstyle \weylaut{} = \wun$};
			\begin{scope}[shift={(2.0,0)}]
				\draw (90:2) -- (0:0) -- coordinate[pos=0.1](C3) (30:2) -- cycle;
				\draw (C3) -- node[pos=0.55,below,blue] {$\quad\radmuv^3$} (30:2) (0:0);
				\draw[drawcolour] (C3) -- (30:2) (0:0);
				\node at (0.87,-0.3) {$\scriptstyle \weylaut{} = \weylaut{1}$};
			\end{scope}
		\end{scope}
	\end{tikzpicture}
  \caption{The fractional parts $\affine{\lambda}^{F,\weylaut{}} \in \apwlat{\vv-1}$ of the admissible weights, projected onto the real slice of $\dcsub$, that define \hw\ $\minmoduv$-modules (left), \srhw\ $\minmoduv$-modules for the parabolic $\para$ (centre top) and $\dynkmod{\para}$ (centre bottom), and \rhw\ $\minmoduv$-modules (right).} \label{fig:admissibles}
\end{figure}
\begin{theorem} \label{thm:class}
	For $\kk$ admissible, every irreducible positive-energy weight $\minmoduv$-module, with \fdim\ weight spaces, is isomorphic to a $\weyl$-twist of one from the following list of mutually inequivalent modules:
	\begin{itemize}
		\item The \hwms\ $\airr{\lambda}$, with $\affine{\lambda} \in \admuv$.
		\item The \srhwms\ $\asemi{\lambda}{[\mu]}$, with $\affine{\lambda} \in \sadmuv^1$ and $[\mu] \in \brac*{\lambda + \CC \sroot{1}} / \ZZ \sroot{1}$ satisfying $[\mu] \ne [\lambda]$ and $[\mu] \ne [\weylaut{1} \cdot \lambda]$.
		\item The $\dynkaut$-twists of the \srhwms\ $\asemi{\lambda}{[\mu]}$ classified above.
		\item The \rhwms\ $\arel{\lambda}{[\mu]}$, with $\affine{\lambda} \in \radmuv^2$ and $[\mu] \in \brac*{\dcsub / \rlat} \setminus \sing{\lambda}$.
	\end{itemize}
\end{theorem}
\noindent The classification of admissible-level positive-energy weight
$\minmoduv$-modules was originally reported in \cite{AraWei16}, where the
results are described in the language of generalised Gelfand--Tsetlin modules.

Note that the top space of $\asemi{\lambda}{[\mu]}$, $\affine{\lambda} \in \sadmuv^1$, is a semidense $\slthree$-module whose weight support is $\mu + \ZZ \sroot{1} - \NN \sroot{2}$.  The weights in $\mu + \ZZ \sroot{1}$ all have multiplicity $1$, as this is the multiplicity of the weights of an irreducible coherent family of $\sltwo$-modules.  The top space multiplicities are moreover uniformly bounded when $\affine{\lambda} \in \radmuv^2$ and are unbounded when $\affine{\lambda} \in \sadmuv^1 \setminus \radmuv^2$.

On the other hand, the top space of $\arel{\lambda}{[\mu]}$, $\affine{\lambda} \in \radmuv^2$, is a dense $\slthree$-module whose weight support is $\mu + \rlat$.  The (constant) multiplicity of the weights coincides with the maximal multiplicity of the weights of the irreducible \hw\ $\slthree$-module $\firr{\lambda}$.  But, the latter is easily determined as $\lambda$ is dominant regular (this is true for all admissible weights).  Indeed, Kazhdan--Lusztig theory realises $\firr{\lambda}$ as the quotient of the corresponding Verma module by $\firr{\weylaut{2} \cdot \lambda}$.  A straightforward basis computation now shows that the desired multiplicity is $\lambda^I_2 + 1$.

One of the deep insights behind Arakawa's proof of \cref{thm:hwclass} is that the set of highest weights of the admissible-level \hw\ $\minmoduv$-modules is naturally partitioned by the three nilpotent orbits of $\slthree$:
\begin{equation}
	\nilorb{0} = \SLG{SU}{3} \cdot
	\begin{pmatrix}
		0&0&0\\
		0&0&0\\
		0&0&0
	\end{pmatrix}
	, \qquad \nilorb{min.} = \SLG{SU}{3} \cdot
	\begin{pmatrix}
		0&1&0\\
		0&0&0\\
		0&0&0
	\end{pmatrix}
	, \qquad \nilorb{pr.} = \SLG{SU}{3} \cdot
	\begin{pmatrix}
		0&1&0\\
		0&0&1\\
		0&0&0
	\end{pmatrix}
	.
\end{equation}
More precisely, for each $\affine{\lambda} \in \admuv$, the associated variety of the annihilator (in $\envalg{\slthree}$) of $\firr{\lambda}$ is the closure of some nilpotent orbit \cite{JosAss85}.  We compare this partition with our classification results:
\begin{itemize}
	\item If $\lambda$ parametrises families of \rhw\ $\minmoduv$-modules, hence $\affine{\lambda} \in \radmuv$, then $\lambda$ belongs to the minimal nilpotent orbit partition: $\lambda \in \nilorb{min.}$.  (This is a restatement of a general result of Mathieu \cite{MatCla00}.)
	\item If $\lambda$ parametrises families of \srhw\ $\minmoduv$-modules (but not relaxed families), hence $\affine{\lambda} \in \brac[\big]{\sadmuv \cup \dynkmod{\sadmuv}} \setminus \radmuv$, then $\lambda$ belongs to the principal nilpotent orbit partition: $\lambda \in \nilorb{pr.}$.
	\item If $\lambda$ only labels a single \hw\ $\minmoduv$-module (and does not parametrise families of (semi)relaxed modules), hence $\affine{\lambda} \in \admuv \setminus \brac[\big]{\sadmuv \cup \dynkmod{\sadmuv}}$, then $\lambda$ belongs to the zero nilpotent orbit partition: $\lambda \in \nilorb{0}$.
\end{itemize}
The question of whether $\lambda$ naturally defines a single \hwm, a family of semirelaxed modules or a family of relaxed modules is thus answered by nilpotent orbits.  It would be very interesting to see how this generalises beyond $\slthree$ and $\minmoduv$.

Returning to \cref{thm:class}, this classification list consists of mutually inequivalent irreducible modules.  However, their $\weyl$-twists are occasionally isomorphic to another module in the list.
\begin{proposition} \label{prop:Wtwists}
	The isomorphisms among the $\weyl$-twists of the list of irreducible $\minmoduv$-modules in \cref{thm:class} are as follows:
	\begin{itemize}
		\item If $\lambda$ is integral, so $\affine{\lambda} \in \admuv^{\wun}$ with $\affine{\lambda}^{F,\wun} = 0$, then we have $\weylmod{}{\airr{\lambda}} \cong \airr{\lambda}$ for all $\weylaut{} \in \weyl$.
		\item If $\lambda_1 \in \NN$ and $\lambda_2 \notin \NN$ ($\lambda_1 \notin \NN$ and $\lambda_2 \in \NN$), so $\affine{\lambda} \in \admuv^{\wun}$ with $\lambda^{F,\wun}_1 = 0$ and $\lambda^{F,\wun}_2 \ne 0$ ($\lambda^{F,\wun}_1 \ne 0$ and $\lambda^{F,\wun}_2 = 0$), then $\weylmod{1}{\airr{\lambda}} \cong \airr{\lambda}$ ($\weylmod{2}{\airr{\lambda}} \cong \airr{\lambda}$) while twisting by representatives of the other cosets of $\weyl / \ideal{\weylaut{1}}$ ($\weyl / \ideal{\weylaut{2}}$) results in modules whose top spaces are \hw\ with respect to a different Borel than $\borel$.
		\item If $\lambda_1, \lambda_2 \notin \NN$, so $\affine{\lambda} \in \admuv^{\weylaut{1}}$ or $\affine{\lambda} \in \admuv^{\wun}$ with $\lambda^{F,\wun}_1, \lambda^{F,\wun}_2 \ne 0$, then $\weylmod{}{\airr{\lambda}}$ is \hw\ with respect to a different Borel than $\borel$, for each $\weylaut{} \ne \wun$.
		\item We have $\weylmod{1}{\asemi{\lambda}{[\mu]}} \cong \asemi{\lambda}{[\weylmod{1}{\mu}]}$, while twisting by representatives of other cosets of $\weyl / \ideal{\weylaut{1}}$ results in modules whose top spaces are \srhw\ with respect to a different parabolic than $\para$.  These are also distinct from the $\weyl$-twists of $\dynkmod{\asemi{\lambda}{[\mu]}}$ for which we note that $\weylaut{2} \dynkmod{\asemi{\lambda}{[\mu]}} \cong \dynkaut \weylmod{1}{\asemi{\lambda}{[\mu]}} \cong \dynkmod{\asemi{\lambda}{[\weylmod{1}{\mu}]}}$.
		\item We have $\weylmod{}{\arel{\lambda}{[\mu]}} \cong \arel{\lambda}{[\weylmod{}{\mu}]}$ for all $\weylaut{} \in \weyl$.
	\end{itemize}
\end{proposition}
\noindent We remark that the isomorphisms of \cref{prop:Wtwists} are either well known or follow easily from the general fact \cite{MatCla00} that a coherent family is invariant under the action of the Weyl group.
\begin{proposition} \label{prop:dtwists}
	\leavevmode
	\begin{itemize}
		\item The $\dynkaut$-twist of an irreducible \hwm \ is again irreducible and \hw: $\dynkmod{\airr{\lambda}} \cong \airr{\dynkmod{\lambda}}$.
		\item The $\dynkaut$-twist of a \rhw\ $\minmoduv$-module is likewise given by
		\begin{equation} \label{eq:dynkrel}
			\dynkmod{\arel{\lambda}{[\mu]}}
			\cong \arel{\weylaut{1} \weylaut{2} \dynkaut \cdot \lambda}{[\dynkmod{\mu}]}
			\cong \arel{\conjaut \weylaut{2} \cdot \lambda}{[\dynkmod{\mu}]} \qquad \text{($\affine{\lambda} \in \radmuv^2$)}.
		\end{equation}
	\end{itemize}
\end{proposition}
\begin{proof}
	The identification of \hw\ $\dynkaut$-twists is well known.  The relaxed identification follows by noting that while $\dynkaut$ acts on the integral and fractional parts of $\affine{\lambda} \in \admuv^{\wun}$ by swapping the first and second Dynkin labels, the action on $\admuv^{\weylaut{1}}$ is somewhat different:
	\begin{equation}
		\begin{aligned}
			\affine{\lambda} \in \admuv^{\wun} &: & \dynkmod{\affine{\lambda}}^I &= (\lambda^I_0, \lambda^I_2, \lambda^I_1), & \dynkmod{\affine{\lambda}}^{F,\wun} &= (\lambda^{F,\wun}_0, \lambda^{F,\wun}_2, \lambda^{F,\wun}_1), \\
			\affine{\lambda} \in \admuv^{\weylaut{1}} &: & \dynkmod{\affine{\lambda}}^I &= (\lambda^I_1, \lambda^I_0, \lambda^I_2), & \dynkmod{\affine{\lambda}}^{F,\weylaut{1}} &= (\lambda^{F,\weylaut{1}}_1 - 1, \lambda^{F,\weylaut{1}}_0 + 1, \lambda^{F,\weylaut{1}}_2).
		\end{aligned}
	\end{equation}
	These formulae make it clear that $\admuv^{\wun}$, $\admuv^{\weylaut{1}}$ and $\radmuv^3$ are all preserved by $\dynkaut$ while $\radmuv^1$ and $\radmuv^2$ are exchanged.  The latter is relevant to the identification of the $\dynkaut$-twist of $\arel{\lambda}{[\mu]}$ as $\lambda \in \radmuv^2$ implies that $\dynkmod{\lambda} \in \radmuv^1$.  The weight in $\radmuv^2$ specifying the same coherent family of $\slthree$-modules as $\dynkmod{\lambda}$ is then $\weylaut{1} \weylaut{2} \cdot \dynkmod{\lambda}$, as per the discussion after \eqref{eq:defradmuv}.
\end{proof}

Finally, we study the action of spectral flow on the irreducible positive-energy weight $\minmoduv$-modules.  For this, we first note the following identifications, easily checked using \eqref{eq:admdynk}:
\begin{itemize}
	\item The level-$\kk$ affine weight corresponding to $\kk \fwt{1} - \weylmod{2}{\lambda}$ has Dynkin labels $(\lambda_1, \lambda_0, \lambda_2)$.  This permutation of Dynkin labels preserves $\admuv^{\wun}$ and $\admuv^{\weylaut{1}}$.
	\item The level-$\kk$ affine weight corresponding to $\kk \fwt{2} - \weylmod{1}{\lambda}$ has Dynkin labels $(\lambda_2, \lambda_1, \lambda_0)$.  This permutation of Dynkin labels preserves $\admuv^{\wun}$, $\admuv^{\weylaut{1}}$ and $\sadmuv^1$.
\end{itemize}
The spectral flow action now follows from \cref{prop:+ESF}.  Here and below, we shall somewhat abuse notation by equating a label $\mu \in \csub$ with the Dynkin indices $(\mu_0, \mu_1, \mu_2)$ of $\affine{\mu}$ (to emphasise the symmetry).
\begin{proposition} \label{prop:sftwists}
	\leavevmode
	\begin{subequations}
	\begin{itemize}
		\item For every $\affine{\mu} \in \admuv$ with $\mu \in \pwlat$, hence $\affine{\mu} \in \admuv^{\wun}$ and $\mu^{F,\wun}_1 = \mu^{F,\wun}_2 = 0$, we have
		\begin{equation}
			\sfmod{\fcwt{1}}{\airr{\mu}} \cong \airr{(\mu_2,\mu_0,\mu_1)} \quad \text{and} \quad
			\sfmod{\fcwt{2}}{\airr{\mu}} \cong \airr{(\mu_1,\mu_2,\mu_0)}.
		\end{equation}
		For all other $\affine{\mu} \in \admuv$, these spectral flows are not positive-energy.
		\item For every $\affine{\mu} \in \admuv$ with $\mu_1 \in \NN$ or $\mu_2 \in \NN$, hence $\affine{\mu} \in \admuv^{\wun}$ and $\mu^{F,\wun}_1 = 0$ or $\mu^{F,\wun}_2 = 0$, we have
		\begin{equation}
			\sfmod{\fcwt{1}-\fcwt{2}}{\airr{\mu}} \cong \weylmod{2}{\airr{(\mu_1,\mu_2,\mu_0})} \quad \text{or} \quad
			\sfmod{\fcwt{2}-\fcwt{1}}{\airr{\mu}} \cong \weylmod{1}{\airr{(\mu_2,\mu_0,\mu_1)}},
		\end{equation}
		respectively.  For all other $\affine{\mu} \in \admuv$, these spectral flows are not positive-energy.
		\item For every $\affine{\mu} \in \admuv$, we have
		\begin{equation}
			\sfmod{-\fcwt{1}}{\airr{\mu}} \cong \weylaut{1} \weylmod{2}{\airr{(\mu_1,\mu_2,\mu_0)}} \quad \text{and} \quad
			\sfmod{-\fcwt{2}}{\airr{\mu}} \cong \weylaut{2} \weylmod{1}{\airr{(\mu_2,\mu_0,\mu_1)}}.
		\end{equation}
		\item For every $\affine{\lambda} \in \sadmuv^1$ and $[\mu] \in (\lambda + \CC \sroot{1}) / \ZZ \sroot{1}$, we have
		\begin{equation}
			\sfmod{-\fcwt{2}}{\asemi{\lambda}{[\mu]}} \cong \conjaut \weylmod{1}{\asemi{(\lambda_2,\lambda_1,\lambda_0)}{[(\mu_2,\mu_1,\mu_0)]}} \cong \conjmod{\asemi{(\lambda_2,\lambda_1,\lambda_0)}{[\kk \fwt{2} - \mu]}}.
		\end{equation}
		All other (nontrivial) spectral flows are not positive-energy.
	\end{itemize}
	\end{subequations}
\end{proposition}

By iterating these identifications, we can visualise the spectral flow orbits on the set of isomorphism classes of irreducible \hw\ $\minmoduv$-modules.  In \cref{fig:sforbits}, we illustrate the parts of these orbits that correspond to positive-energy modules.  This depends only on the number of nonnegative integer Dynkin labels.
\begin{figure}
	\begin{tikzpicture}[scale=1,<->,>=latex]
		\draw[->] (0:0) -- (30:1) node[below right] {$\fcwt{1}$};
		\draw[->] (0:0) -- (90:1) node[left] {$\fcwt{2}$};
		\node[nom] at (-4.25,5.5) {$N=0$};
		\begin{scope}[shift={(-5,3)},scale=2.5]
			\node (12) at (0:0) {$\weylaut{1} \weylmod{2}{\airr{(\mu_1,\mu_2,\mu_0)}}$};
			\node (0) at (30:1) {$\airr{(\mu_0,\mu_1,\mu_2)}$};
			\node (21) at (-30:1) {$\weylaut{2} \weylmod{1}{\airr{(\mu_2,\mu_0,\mu_1)}}$};
			\draw (12) -- (0);
			\draw (21) -- (0);
			\draw (12) -- (21);
		\end{scope}
		\node[nom] at (4.25,5.5) {$N=1$ ($\mu_0 \in \NN$)};
		\begin{scope}[shift={(3.5,3.5)},scale=2.5]
			\node (12) at (0:0) {$\weylaut{1} \weylmod{2}{\airr{(\mu_1,\mu_2,\mu_0)}}$};
			\node (0) at (30:1) {$\airr{(\mu_0,\mu_1,\mu_2)}$};
			\node (21) at (-30:1) {$\weylaut{2} \weylmod{1}{\airr{(\mu_2,\mu_0,\mu_1)}}$};
			\node (3) at (-90:1) {$\weylmod{3}{\airr{(\mu_0,\mu_1,\mu_2)}}$};
			\draw (12) -- (0);
			\draw (21) -- (0);
			\draw (12) -- (21);
			\draw (12) -- (3);
			\draw (21) -- (3);
		\end{scope}
		\node[nom] at (-4.25,0) {$N=2$ ($\mu_0 \notin \NN$)};
		\begin{scope}[shift={(-4.25,-3.5)},scale=2.5]
			\node (0) at (0:0) {$\airr{(\mu_0,\mu_1,\mu_2)}$};
			\node (-1) at (30:1) {$\airr{(\mu_2,\mu_0,\mu_1)}$};
			\node (-2) at (90:1) {$\airr{(\mu_1,\mu_2,\mu_0)}$};
			\node (1) at (150:1) {$\weylmod{1}{\airr{(\mu_2,\mu_0,\mu_1)}}$};
			\node (2) at (-30:1) {$\weylmod{2}{\airr{(\mu_1,\mu_2,\mu_0)}}$};
			\node (12) at (-150:1) {$\weylaut{1} \weylmod{2}{\airr{(\mu_1,\mu_2,\mu_0)}}$};
			\node (21) at (-90:1) {$\weylaut{2} \weylmod{1}{\airr{(\mu_2,\mu_0,\mu_1)}}$};
			\draw (0) -- (-1);
			\draw (0) -- (-2);
			\draw (0) -- (1);
			\draw (0) -- (2);
			\draw (0) -- (12);
			\draw (0) -- (21);
			\draw (12) -- (21);
			\draw (21) -- (2);
			\draw (2) -- (-1);
			\draw (-1) -- (-2);
			\draw (-2) -- (1);
			\draw (1) -- (12);
		\end{scope}
		\node[nom] at (4.25,0) {$N=3$ ($\affine{\mu} \in \apwlat{\kk}$)};
		\begin{scope}[shift={(4.25,-3.5)},scale=2]
			\node (0) at (0:0) {$\airr{(\mu_0,\mu_1,\mu_2)}$};
			\node (-1) at (30:1) {$\airr{(\mu_2,\mu_0,\mu_1)}$};
			\node (-2) at (90:1) {$\airr{(\mu_1,\mu_2,\mu_0)}$};
			\node (1) at (150:1) {$\airr{(\mu_2,\mu_0,\mu_1)}$};
			\node (2) at (-30:1) {$\airr{(\mu_1,\mu_2,\mu_0)}$};
			\node (12) at (-150:1) {$\airr{(\mu_1,\mu_2,\mu_0)}$};
			\node (21) at (-90:1) {$\airr{(\mu_2,\mu_0,\mu_1)}$};
			\node (d1) at (30:1.6) {\rotatebox{30}{$\cdots$}};
			\node (d2) at (90:1.4) {\rotatebox{90}{$\cdots$}};
			\node (d3) at (150:1.6) {\rotatebox{150}{$\cdots$}};
			\node (d4) at (-150:1.6) {\rotatebox{-150}{$\cdots$}};
			\node (d5) at (-90:1.4) {\rotatebox{-90}{$\cdots$}};
			\node (d6) at (-30:1.6) {\rotatebox{-30}{$\cdots$}};
			\draw (0) -- (-1);
			\draw (0) -- (-2);
			\draw (0) -- (1);
			\draw (0) -- (2);
			\draw (0) -- (12);
			\draw (0) -- (21);
			\draw (12) -- (21);
			\draw (21) -- (2);
			\draw (2) -- (-1);
			\draw (-1) -- (-2);
			\draw (-2) -- (1);
			\draw (1) -- (12);
		\end{scope}
	\end{tikzpicture}
\caption{Illustrations of the part of the spectral flow orbit through $\airr{\mu}$, $\affine{\mu} \in \admuv$, corresponding to positive-energy modules.  Here, $N \in \set{0,1,2,3}$ denotes the number of Dynkin labels of $\affine{\mu}$ that lie in $\NN$.  Arrows indicate the spectral flow $\sfaut{\fcwt{}}$, where $\fcwt{} \in \set*{\pm\fcwt{1}, \pm\fcwt{2}, \pm(\fcwt{1}-\fcwt{2})}$.  (Note that $N=3$ requires $\kk \in \NN$, hence $\vv=1$.  Spectral flow is $\crlat$-periodic in this case.)} \label{fig:sforbits}
\end{figure}

\subsection{Degenerations} \label{sec:MinModDeg}

We have classified the irreducible (semi)\rhw\ $\minmoduv$-modules for general
admissible levels.  They form families that mirror the coherent families of
$\slthree$-modules from which they were constructed.  However, the members of
the latter families are only generically irreducible, hence the same is true
for their (semi)relaxed counterparts.  It is therefore interesting, and useful
for what follows, to determine explicit decomposition formulae when the
$\minmoduv$-modules constructed in this way are reducible.  We refer to these decompositions as the \emph{degenerations} of a given family of (semi)\rhwms.

We first consider the degenerations of the family of \srhwms\ $\asemi{\lambda}{[\mu]}$ corresponding to a given $\affine{\lambda} \in \sadmuv^1$.  According to \cref{prop:sdegen}, these occur when the coset $[\mu]$ in $(\lambda + \CC \sroot{1}) / \ZZ \sroot{1}$ takes the value $[\lambda]$ or $[\weylaut{1} \cdot \lambda]$.  Since the $\asemi{\lambda}{[\mu]}$ were constructed from an irreducible semisimple coherent family of $\levi$-modules, it follows that $\asemi{\lambda}{[\lambda]}$ has an irreducible submodule isomorphic to $\airr{\lambda}$.  Its complement is likewise an irreducible \hwm, with highest weight $\lambda + \sroot{1}$, but with respect to the Borel $\weylmod{1}{\borel}$ spanned by $\csub$, $e^2$, $e^3$ and $f^1$.  It is therefore isomorphic to the twist $\weylmod{1}{\airr{\weylaut{1} \cdot \lambda}}$.  The analysis for $[\mu] = [\weylaut{1} \cdot \lambda]$ is similar.

We thereby arrive at our first degeneration result.
\begin{proposition} \label{prop:sdegen}
	For $\kk$ admissible, the reducible \srhw\ $\minmoduv$-modules decompose as follows:
	\begin{equation}
		\asemi{\lambda}{[\lambda]} \cong \airr{\lambda} \oplus \weylmod{1}{\airr{\weylaut{1} \cdot \lambda}} \quad \text{and} \quad
		\asemi{\lambda}{[\weylaut{1} \cdot \lambda]} \cong \airr{\weylaut{1} \cdot \lambda} \oplus \weylmod{1}{\airr{\lambda}} \qquad
		\text{($\affine{\lambda} \in \sadmuv^1$)}.
	\end{equation}
\end{proposition}
\noindent This implies similar degenerations for the reducible \srhwms\ obtained by $\grp{D}_6$-twists.

The degenerations of the \rhwms\ $\arel{\lambda}{[\mu]}$, $\affine{\lambda} \in \radmuv^2$, require more work.  By \cref{prop:rdegen}, they occur when $[\mu] \in \dcsub / \rlat$ takes values in $\sing{\lambda}$, defined in \eqref{eq:defsing}.  We shall assume first that
\begin{equation}
	[\mu] \in [\lambda + \CC \sroot{1}] \subset \sing{\lambda}
\end{equation}
so that we may choose a representative $\mu \in \lambda + \CC \sroot{1}$ of $[\mu]$.  As before, $\arel{\lambda}{[\mu]}$ decomposes as a direct sum of irreducible submodules.  Because $\radmuv^2 \subseteq \sadmuv^1$, it is clear that one direct summand must be (isomorphic to) $\asemi{\lambda}{[\mu]}$, where now $[\mu] \in (\lambda + \CC \sroot{1}) / \ZZ \sroot{1}$, and its complement must be (isomorphic to) $\conjmod{\asemi{\lambda'}{[\mu']}}$, for some $\lambda' \in \sadmuv^1$ and $[\mu'] \in (\lambda' + \CC \sroot{1}) / \ZZ \sroot{1}$:
\begin{equation} \label{eq:rdeg1}
	\arel{\lambda}{[\mu]} \cong \asemi{\lambda}{[\mu]} \oplus \conjmod{\asemi{\lambda'}{[\mu']}} \qquad \text{($\mu \in \lambda + \CC \sroot{1}$)}.
\end{equation}
As $\rlat \cap \CC \sroot{1} = \ZZ \sroot{1}$, this decomposition is independent of the choice of $\mu$.

To identify $\lambda'$ and $[\mu']$, note that \eqref{eq:rdeg1} implies that
\begin{equation}
	\conjmod{\asemi{\lambda}{[\mu]}} \oplus \asemi{\lambda'}{[\mu']} \cong \conjmod{\arel{\lambda}{[\mu]}}
	= \dynkaut \weylmod{3}{\arel{\lambda}{[\mu]}} \cong \dynkmod{\arel{\lambda}{[\weylmod{3}{\mu}]}}
	\cong \arel{\conjaut \weylaut{2} \cdot \lambda}{[\conjmod{\mu}]},
\end{equation}
by \cref{prop:Wtwists,eq:dynkrel}.  Comparing with \eqref{eq:rdeg1}, this identifies $\lambda'$ as $\conjaut \weylaut{2} \cdot \lambda$.  To determine $[\mu']$, we need to find a representative $\mu' \in \lambda' + \CC \sroot{1}$ of $[\conjmod{\mu}] \in \dcsub / \rlat$.  Since $\affine{\lambda} \in \radmuv^2$, we compute that
\begin{equation}
	\conjmod{\mu} - \lambda' = \lambda - \mu + 2 \sroot{1} - (\lambda_2^I - 1) \sroot{2},
\end{equation}
hence we may take $\mu' = \conjmod{\mu} + (\lambda_2^I - 1) \sroot{2}$ (since $\lambda - \mu + 2 \sroot{1} \in \CC \sroot{1}$).  Indeed, this choice gives
\begin{equation}
	\arel{\lambda'}{[\mu']} = \arel{\conjaut \weylaut{2} \cdot \lambda}{[\conjmod{\mu}]} \cong \asemi{\lambda'}{[\mu']} \oplus \conjmod{\asemi{\lambda}{[\mu]}} \qquad \text{($\mu' \in \lambda' + \CC \sroot{1}$)},
\end{equation}
which is \eqref{eq:rdeg1} with $\lambda \leftrightarrow \lambda'$ and $\mu \leftrightarrow \mu'$.  We conclude that the explicit degeneration in this case is
\begin{equation} \label{eq:rdeg}
	\arel{\lambda}{[\mu]} \cong \asemi{\lambda}{[\mu]} \oplus \conjmod{\asemi{\conjaut \weylaut{2} \cdot \lambda}{[\conjmod{\mu} + (\lambda_2^I - 1) \sroot{2}]}} \qquad \text{($\mu \in \lambda + \CC \sroot{1}$)}.
\end{equation}

We still have to repeat this analysis for the other curves in $\sing{\lambda}$: $[\weylaut{1} \cdot \lambda + \CC \sroot{2}]$ and $[\lambda + \CC \sroot{3}]$.  However, the degenerations for these curves may be deduced from \eqref{eq:rdeg} by $\weyl$-twisting.  Indeed, for $\affine{\lambda} \in \radmuv^2$ and $\mu \in \lambda + \CC \sroot{3}$, we may twist $\arel{\lambda}{[\mu]}$ by $\weylaut{2}$ to obtain (using \cref{prop:Wtwists})
\begin{equation}
	\weylmod{2}{\arel{\lambda}{[\mu]}}
	\cong \arel{\lambda}{[\weylmod{2}{\mu}]}
	= \arel{\lambda}{[\weylmod{2}{\mu} + \lambda^I_2 \sroot{2}]} \quad \text{and} \quad
	\weylmod{2}{\mu} + \lambda^I_2 \sroot{2} \in \lambda + \CC \sroot{1}.
\end{equation}
Thus, $\weylaut{2}$ maps degenerate \rhwms\ with $[\mu] \in [\lambda + \CC \sroot{3}]$ to those with $[\lambda + \CC \sroot{1}]$ (and vice versa).  Similarly, when $\mu \in \weylaut{1} \cdot \lambda + \CC \sroot{2}$, twisting by $\weylaut{1}$ gives
\begin{equation}
	\weylmod{1}{\arel{\lambda}{[\mu]}}
	\cong \arel{\lambda}{[\weylmod{1}{\mu}]}
	= \arel{\lambda}{[\weylaut{1} \cdot \mu]} \quad \text{and} \quad \weylaut{1} \cdot \mu \in \lambda + \CC \sroot{3},
\end{equation}
so $\weylaut{1}$ exchanges the degenerations for $[\mu] \in [\lambda + \CC \sroot{3}]$ with those for $[\mu] \in [\weylaut{1} \cdot \lambda + \CC \sroot{2}]$.  Combining these identifications with the known degeneration \eqref{eq:rdeg} then completes the analysis of the relaxed degenerations.
\begin{proposition} \label{prop:rdegen}
	For $\kk$ admissible, the reducible \rhw\ $\minmoduv$-modules decompose as follows:
	\begin{equation}
		\arel{\lambda}{[\mu]} \cong
		\begin{cases*}
			\asemi{\lambda}{[\mu]} \oplus \conjmod{\asemi{\conjaut \weylaut{2} \cdot \lambda}{[\conjmod{\mu} + (\lambda_2^I - 1) \sroot{2}]}} & if $\mu \in \lambda + \CC \sroot{1}$, \\
			\weylaut{1} \weylmod{2}{\asemi{\lambda}{[\weylaut{2} \weylmod{1}{\mu} + (\lambda^I_2-1) \sroot{2}]}} \oplus \conjaut \weylaut{1} \weylmod{2}{\asemi{\conjaut \weylaut{2} \cdot \lambda}{[\conjaut \weylaut{2} \weylmod{1}{\mu}]}} & if $\mu \in \weylaut{1} \cdot \lambda + \CC \sroot{2}$, \\
			\weylmod{2}{\asemi{\lambda}{[\weylmod{2}{\mu} + \lambda^I_2 \sroot{2}]}} \oplus \conjaut \weylmod{2}{\asemi{\conjaut \weylaut{2} \cdot \lambda}{[\conjaut \weylaut{2} \cdot \mu]}} & if $\mu \in \lambda + \CC \sroot{3}$.
		\end{cases*}
	\end{equation}
\end{proposition}
\noindent We emphasise here that the conditions on the \rhs\ have to be interpreted as requiring that the coset $[\mu] \in \dcsub / \rlat$ on the \lhs\ has a representative $\mu \in \dcsub$ in the set given.

Of course, one can combine the information in \cref{prop:sdegen,prop:rdegen} to determine what happens to the degenerate $\arel{\lambda}{[\mu]}$ when their semirelaxed direct summands also degenerate.  This happens if and only if $[\mu]$ belongs to the intersection of two of the three curves comprising $\sing{\lambda}$,\footnote{It is easy to check that $\affine{\lambda} \in \radmuv^2$ implies that there are no triple intersections.} these intersections being $[\lambda]$, $[\weylaut{1} \cdot \lambda]$ and $[\weylaut{2} \weylaut{1} \cdot \lambda]$.  For example, $[\lambda]$ belongs to both $[\lambda + \CC \sroot{1}]$ and $[\lambda + \CC \sroot{3}]$, so \cref{prop:sdegen,prop:rdegen} give both
\begin{subequations}
	\begin{align}
		\arel{\lambda}{[\lambda]}
		&\cong \asemi{\lambda}{[\lambda]} \oplus \conjmod{\asemi{\conjaut \weylaut{2} \cdot \lambda}{[\conjaut \weylaut{2} \cdot \lambda]}}
		\cong \airr{\lambda} \oplus \weylmod{1}{\airr{\weylaut{1} \cdot \lambda}} \oplus \conjmod{\airr{\conjaut \weylaut{2} \cdot \lambda}} \oplus \conjaut \weylmod{1}{\airr{\weylaut{1} \conjaut \weylaut{2} \cdot \lambda}} \notag \\
		&= \airr{\lambda} \oplus \weylmod{1}{\airr{\weylaut{1} \cdot \lambda}} \oplus \conjmod{\airr{\conjaut \weylaut{2} \cdot \lambda}} \oplus \weylaut{2} \weylmod{1}{\airr{\weylaut{1} \cdot \lambda}}
		\intertext{(where we note that $\conjaut \weylaut{1} = \weylaut{1} \weylaut{3} \dynkaut = \weylaut{2} \weylaut{1} \dynkaut$ and $\weylaut{1} \conjaut \weylaut{2} = \dynkaut \weylaut{3} \weylaut{1} \weylaut{2} = \dynkaut \weylaut{1}$, then apply \cref{prop:dtwists}) and}
		\arel{\lambda}{[\lambda]}
		&\cong \weylmod{2}{\asemi{\lambda}{[\lambda]}} \oplus \conjaut \weylmod{2}{\asemi{\conjaut \weylaut{2} \cdot \lambda}{[\conjaut \weylaut{2} \cdot \lambda]}}
		\cong \weylmod{2}{\airr{\lambda}} \oplus \weylaut{2} \weylmod{1}{\airr{\weylaut{1} \cdot \lambda}} \oplus \conjaut \weylmod{2}{\airr{\conjaut \weylaut{2} \cdot \lambda}} \oplus \conjaut \weylaut{2} \weylmod{1}{\airr{\weylaut{1} \conjaut \weylaut{2} \cdot \lambda}} \notag \\
		&= \weylmod{2}{\airr{\lambda}} \oplus \weylaut{2} \weylmod{1}{\airr{\weylaut{1} \cdot \lambda}} \oplus \conjaut \weylmod{2}{\airr{\conjaut \weylaut{2} \cdot \lambda}} \oplus \weylmod{1}{\airr{\weylaut{1} \cdot \lambda}}
	\end{align}
\end{subequations}
(which follows similarly).  The fact that these two decompositions agree relies on the fact that $\affine{\lambda} \in \radmuv^2$ implies that $\lambda_2 = \lambda^I_2 \in \NN$, hence $\weylmod{2}{\airr{\lambda}} \cong \airr{\lambda}$ by \cref{prop:Wtwists}, and $(\conjaut \weylaut{2} \cdot \lambda)_2 = \lambda^I_2 \in \NN$, hence $\conjaut \weylmod{2}{\airr{\conjaut \weylaut{2} \cdot \lambda}} \cong \conjmod{\airr{\conjaut \weylaut{2} \cdot \lambda}}$.  The checks for $[\mu] = [\weylaut{1} \cdot \lambda] \in [\lambda + \CC \sroot{1}] \cap [\weylaut{1} \cdot \lambda + \CC \sroot{2}]$ and $[\mu] = [\weylaut{2} \weylaut{1} \cdot \lambda] \in [\weylaut{1} \cdot \lambda + \CC \sroot{2}] \cap [\lambda + \CC \sroot{3}]$ are very similar and are left to the reader.

Assuming, as always, that $\kk$ is admissible, the degeneration picture for positive-energy $\minmoduv$-modules may then be summarised as follows (up to $\grp{D}_6$-twists):
\begin{itemize}
	\item For $\affine{\lambda} \in \admuv \setminus \sadmuv$, there is only an irreducible \hwm\ $\airr{\lambda}$.
	\item For $\affine{\lambda} \in \sadmuv^1 \setminus \radmuv^2$, the \srhwms\ $\asemi{\lambda}{[\mu]}$, $[\mu] \in (\lambda + \CC \sroot{1}) / \ZZ \sroot{1}$, are irreducible for all cosets $[\mu]$ but two.  At these two cosets, the $\asemi{\lambda}{[\mu]}$ decompose into the direct sum of two irreducible \hwms\ as per \cref{prop:sdegen}.
	\item For $\affine{\lambda} \in \radmuv^2$, the \rhwms\ $\arel{\lambda}{[\mu]}$, $[\mu] \in \dcsub / \rlat$, are irreducible for all cosets $[\mu]$ except those belonging to the union of three curves.  On these three curves, the $\arel{\lambda}{[\mu]}$ decompose into the direct sum of two \srhwms\ as per \cref{prop:rdegen}.  If $[\mu]$ belongs to exactly one of these curves, then these semirelaxed modules are irreducible.  Otherwise, they are reducible and $\arel{\lambda}{[\mu]}$ decomposes into the direct sum of four irreducible \hwms.
\end{itemize}

\section{The minimal model \texorpdfstring{$\themm$}{A2(3,2)}} \label{sec:minmod32}

The previous \lcnamecref{sec:MinMod} gives a complete account of the irreducible positive-energy modules of the $\slthree$ minimal models.  Together with their spectral flow orbits, we expect that these modules will form the building blocks of an associated $\slthree$ minimal model \cft.  To test this, we shall explore the modularity of the characters and check if the Grothendieck fusion coefficients, as computed by the standard Verlinde formula of \cite{CreLog13,RidVer14}, are nonnegative integers.

As we shall explain, technical reasons will restrict the present exploration to $\uu = 3$ and $\vv = 2$.  This exploration is by no means trivial, nor is it without independent interest.  In fact, the $\slthree$ minimal model $\themm$ is closely related to several other interesting logarithmic \voas\ including the $N=4$ superconformal minimal model with $c=-9$ \cite{AdaRea14}, the Feigin--Tipunin algebra $W_{A_2}^0(2)$ \cite{AdaPar20} and Semikhatov's octuplet algebra $W_{A_2}(2)$ \cite{SemNot13}.  Further details concerning these relations may be found in the companion paper \cite{CreKaz21}.

\subsection{Characters and linear independence} \label{sec:32linindep}

The character of a weight $\aslthree$-module $\affine{\Mod{M}}$ is defined to be the following formal power series in $\yy$, $\zz$ and $\qq$:
\begin{equation} \label{eq:deffchar}
	\fch{\affine{\Mod{M}}}{\yy}{\zz}{\qq} = \traceover{\affine{\Mod{M}}} \yy^K \zz^{h_0} \qq^{L_0 - \cc/24}.
\end{equation}
Here, we do not fix a choice of $h \in \csub$, but instead leave it unspecified --- the eigenvalue of $\zz^{h_0}$ on a weight vector of $\slthree$-weight $\lambda$ is then, formally, $\zz^{\lambda}$.  We shall also introduce new variables $\theta \in \CC$, $\zeta \in \csub$ and $\tau \in \HH$, where $\HH \subset \CC$ is the upper-half plane, satisfying
\begin{equation}
	\yy = \ee^{2 \pi \ii \theta}, \quad \zz^{h_0} = \ee^{2 \pi \ii \zeta} \quad \text{and} \quad \qq = \ee^{2 \pi \ii \tau}.
\end{equation}
In these variables, the generic expression for the character takes the form
\begin{equation} \label{eq:defmchar}
	\mch{\affine{\Mod{M}}}{\theta}{\zeta}{\tau} = \traceover{\affine{\Mod{M}}} \exp \sqbrac[\big]{2 \pi \ii (K \theta + \zeta_0 + (L_0 - \cc/24) \tau)}.
\end{equation}

If $\omega$ is an automorphism of $\aslthree$, then the character of $\func{\omega}{\affine{\Mod{M}}}$ is obtained from \eqref{eq:defmchar} by inserting $\omega^{-1}$ into the argument of the exponential.  In particular, the characters of the $\grp{D}_6$-twists and spectral flows of $\affine{\Mod{M}}$ are given by
\begin{equation} \label{eq:autchar}
	\begin{aligned}
		\mch{\func{\omega}{\affine{\Mod{M}}}}{\theta}{\zeta}{\tau} &= \mch{\affine{\Mod{M}}}{\theta}{\func{\omega^{-1}}{\zeta}}{\tau}, & \omega &\in \grp{D}_6, \\ \text{and} \quad
		\mch{\sfmod{\xi}{\affine{\Mod{M}}}}{\theta}{\zeta}{\tau} &= \mch{\affine{\Mod{M}}}{\theta + \killing{\zeta}{\xi} + \tfrac{1}{2} \killing{\xi}{\xi} \tau}{\zeta + \tau \xi}{\tau}, & \xi &\in \cwlat.
	\end{aligned}
\end{equation}

When $\kk$ is admissible, the character of a given irreducible \hw\ $\minmoduv$-module is explicitly known \cite{KacMod88} as a holomorphic function converging on a certain subdomain of $(\theta, \zeta, \tau) \in \CC \times \csub \times \HH$.  This character may also be analytically continued to a meromorphic function on $\CC \times \csub \times \HH$ and the result has nice modular properties.  Unfortunately, the modular transforms do not respect the actual convergence regions of the \hw\ characters and using them to compute fusion coefficients via Verlinde's formula yields unacceptable results \cite{KohFus88,KacRem17}, namely negative multiplicities.

As was pointed out in \cite{RidSL208}, the root cause of this failure of modularity is the fact that the meromorphic extensions identify characters of inequivalent $\minmoduv$-modules, in particular those of the irreducible \hwms\ with certain of their spectral flows.  This lack of linear independence motivated the development of the \emph{standard module formalism} of \cite{CreLog13,RidVer14} in which characters are treated as formal power series (or distributions) in $\zz$ and $\zeta$ whose coefficients are meromorphic functions on $(\theta, \tau) \in \CC \times \HH$.  This minor change of viewpoint often allows one to identify a collection of \emph{standard modules} whose characters form a topological basis for the space spanned by all relevant characters.

For the $\slthree$ minimal models $\minmoduv$ with $\kk$ admissible and $\vv>1$, we expect that a set of standard modules will be provided by the \rhwms, some (possibly none) of the \srhwms\ and their spectral flows.  The standard characters can then be deduced from those of the \hwms\ by generalising the technology developed in \cite{KawRel18} or by lifting Mathieu's twisted localisation functors to the affine setting, see \cite{KawRel20}.  The following result is, in particular, germane.
\begin{theorem}[{\cite[Thm.~1.1(b)]{KawRel20}}] \label{thm:rchar}
	For $\kk$ admissible, the relaxed characters of $\minmoduv$ take the form
	\begin{equation} \label{eq:rchar}
		\fch{\arel{\lambda}{[\mu]}}{\yy}{\zz}{\qq} = \yy^{\kk} \frac{\ch{H^0_{\textup{min.}}(\airr{\lambda})}(\qq)}{\eta(\qq)^4} \sum_{\nu \in [\mu]} \zz^{\nu}, \quad \affine{\lambda} \in \radmuv^2,\ [\mu] \in \dcsub / \rlat.
	\end{equation}
\end{theorem}
\noindent Here, $H^0_{\textup{min.}}(\airr{\lambda})$ denotes a ``$-$-type'' minimal \qhr\ of $\airr{\lambda}$.\footnote{We recall that the ``$-$-type'' \qhr\ was introduced by Frenkel, Kac and Wakimoto in \cite{FreCha92} for regular (principal) nilpotent elements.  It differs from the usual ``$+$-type'' regular reduction in that it gauges the negative root vectors instead of the positive ones.  Although both reductions give isomorphic W-algebras, the corresponding functors on modules are different.  The reduction functor used in \cref{thm:rchar} is a generalisation of this $-$-type functor to all nilpotents due to Kac--Wakimoto \cite{KacRat07} and Arakawa \cite{AraRep11}.}  We refer to \cite{AraRep11} for details, only noting here that $H^0_{\textup{min.}}(\airr{\lambda})$ is nonzero because $\affine{\lambda} \in \radmuv^2$ implies that $\lambda_1$ and $\lambda_1 + \lambda_2 + 1$ are not nonnegative integers \cite[Thm.~5.7.1]{AraRep11}.  (This distinguishes the reduction from the usual ``$+$-type'' minimal reduction of \cite{KacQua03,KacQua04} because in this case the reduction would be nonzero if and only if $\lambda_0$ is not a nonnegative integer \cite[Thm.~6.7.4]{AraRep05}.)

An easy application of \cref{prop:Wtwists,prop:dtwists} shows that the module conjugate to $\arel{\lambda}{[\mu]}$, $\lambda \in \radmuv^2$ and $[\mu] \in \dcsub / \rlat$, is $\arel{\lambda'}{[-\mu]}$, where $\affine{\lambda}' = (\lambda_1+\fracuv, \lambda_0-\fracuv, \lambda_2) \in \radmuv^2$.  Combining this with \eqref{eq:autchar} and \eqref{eq:rchar} thus gives
\begin{equation}
	\yy^{\kk} \frac{\ch{H^0_{\textup{min.}}(\airr{\lambda})}(\qq)}{\eta(\qq)^4} \sum_{\nu \in [\mu]} \zz^{-\nu}
	= \ch{\conjmod{\arel{\lambda}{[\mu]}}} = \ch{\arel{\lambda'}{[-\mu]}}
	= \yy^{\kk} \frac{\ch{H^0_{\textup{min.}}(\airr{\lambda'})}(\qq)}{\eta(\qq)^4} \sum_{\nu \in [-\mu]} \zz^{\nu},
\end{equation}
from which we conclude that $H^0_{\textup{min.}}(\airr{\lambda})$ and $H^0_{\textup{min.}}(\airr{\lambda'})$ have the same $\qq$-character.  However, it now follows that
\begin{equation}
	\fch{\arel{\lambda}{[\mu]}}{\yy}{\zz}{\qq} = \fch{\arel{\lambda'}{[\mu]}}{\yy}{\zz}{\qq},
\end{equation}
which gives an undesired linear dependence of characters unless $\lambda = \lambda'$.  We can therefore only save linear independence if ${\lambda'_0}^I = \lambda^I_1$ and ${\lambda'_0}^{F,\wun} = \lambda^{F,\wun}_1 - 1$ for all $\affine{\lambda} \in \radmuv^2$.  Since $\affine{\lambda}^I \in \apwlat{\uu-3}$ and $\affine{\lambda}^{F,\wun} \in \apwlat{\vv-1}$ satisfies $\lambda^{F,\wun}_1 \ne 0$, linear independence of the relaxed characters thus requires $\uu=3$ and $\vv=2$.
\begin{corollary} \label{cor:only32}
	Let $\kk$ be admissible with $\vv>1$ (so that \rhw\ $\minmoduv$-modules exist).  Then, the characters of the $\arel{\lambda}{[\mu]}$, with $\affine{\lambda} \in \radmuv^2$ and $[\mu] \in \dcsub / \rlat$, are linearly dependent unless $\uu=3$ and $\vv=2$.
\end{corollary}
\noindent Of course, when $\kk$ is admissible with $\vv=1$ (so $\kk \in \NN$), the irreducible $\minmoduv$-modules are all \hw\ and their characters are well known to be linearly independent.

As the standard module approach to modularity and fusion assumes that the standard modules have linearly independent characters, it is therefore restricted by \cref{cor:only32} to $\themm$ (and the rational \wzw\ models $\minmod{\uu}{1}$, $\uu\ge3$).  The problem here is that the definition \eqref{eq:deffchar} of characters cannot distinguish $\arel{\lambda}{[\mu]}$ from $\arel{\lambda'}{[\mu]}$ because the eigenvalues of $h_0$ and $L_0$ do not even distinguish their top spaces $\frel{\lambda}{[\mu]}$ and $\frel{\lambda'}{[\mu]}$.  Indeed, they share the same $\slthree$-weights and eigenvalue of the quadratic Casimir.  However, the eigenvalue of the cubic Casimir of $\slthree$ may be different, a possibility to which \eqref{eq:deffchar} is blind.

In the absence of a better definition of character, the obvious fix is to lift characters to $1$-point functions on the torus \cite{ZhuMod96}.\footnote{We remark that this name, while quite standard, may be a little misleading.  It does not refer to a $1$-point correlation function of genus $1$ in an appropriate \cft, but rather to a chiral version where the trace is taken over a fixed module.  In other words, this concept generalises the definition \eqref{eq:deffchar} of a character by inserting some fixed zero mode (usually unexponentiated!) inside the trace.}  We shall not pursue this here, leaving it for a sequel, and here just note that this lift was recently worked out for the logarithmic \bp\ minimal models in \cite{FehMod21} by relating the relevant $1$-point functions to those of the rational $W_3$ minimal models, for which linear independence is generically guaranteed by a theorem of Zhu \cite{ZhuMod96}.  We expect that combining this method with the recent results of \cite{AraMod19,AraRat19,AdaRel21} will allow us to tackle general $\slthree$ minimal models and hope to return to this in the near future.

In what follows, we shall continue with the analysis of the $\slthree$ minimal model $\themm$ whose \rhwms\ have linearly independent characters.  Along with their spectral flows, these will furnish us with a collection of standard modules on which to test modularity and nonnegativity of fusion coefficients.

\subsection{$\themm$: a summary} \label{sec:32summary}

For convenience, we collect here some of the many general results obtained above, specialised to the case $\uu=3$ and $\vv=2$.  The level is thus $\kk = -\frac{3}{2}$ and the central charge is $\cc = -8$.  The minimal model $\themm$ has already been considered in several prior works including \cite{PerVer08,BeeInf15,AdaRea16,AraQua16}.  However, we believe that the modularity study reported below is new.

An admissible weight $\affine{\lambda} \in \theawts$ has $\affine{\lambda}^I = 0$ and $\affine{\lambda}^{F,\wun} \in \apwlat{1}$ or $\affine{\lambda}^{F,\weylaut{1}} = \fwt{1}$.  Similarly, $\affine{\lambda} \in \theswts^1 = \therwts^2$ implies that $\affine{\lambda}^I = 0$ and $\affine{\lambda}^{F,\wun} = \fwt{1}$.  Up to $\grp{D}_6$-twists, the isomorphism classes of the irreducible positive-energy weight $\themm$-modules are therefore represented by:
\begin{itemize}
	\item Four \hwms\ $\airr{0}$, $\airr{-3\fwt{1}/2}$, $\airr{-3\fwt{2}/2}$ and $\airr{-\wvec/2}$.  These were originally classified in \cite{PerVer08}.  Their ground states have conformal weight $0$, $-\frac{1}{2}$, $-\frac{1}{2}$ and $-\frac{1}{2}$, respectively.
	\item A single family of \srhwms\ $\asemi{-3\fwt{1}/2}{[\mu]}$, for all $[\mu] \in (-\frac{3}{2} \fwt{1} + \CC \sroot{1}) / \ZZ \sroot{1}$ except for $[\mu] = [-\frac{3}{2} \fwt{1}]$ and $[\mu] = [-\frac{1}{2} \wvec]$.  These were classified in \cite{AraWei16,KawRel19}.  Their ground states always have conformal weight $-\frac{1}{2}$.
	\item A single family of \rhwms\ $\arel{-3\fwt{1}/2}{[\mu]}$, for all $[\mu] \in (\dcsub / \rlat) \setminus \sing{-\frac{3}{2} \fwt{1}}$.  Here,
	\begin{equation}
		\sing{-\tfrac{3}{2} \fwt{1}} = [-\tfrac{3}{2} \fwt{1} + \CC \sroot{1}] \cup [-\tfrac{1}{2} \wvec + \CC \sroot{2}] \cup [-\tfrac{3}{2} \fwt{1} + \CC \sroot{3}].
	\end{equation}
	These were first constructed in \cite{AdaRea16} and shown to be a complete set in \cite{KawRel19}.  Their ground states also always have conformal weight $-\frac{1}{2}$.
\end{itemize}
By virtue of the fact that $\abs{\theswts} = \abs{\therwts} = 1$, we shall from here on drop the superscript label $\lambda$ on the $\asemi{\lambda}{[\mu]}$ and $\arel{\lambda}{[\mu]}$, for brevity.

The common multiplicity of the weights of the top space of the $\arel{}{[\mu]}$ is $\lambda^I_2+1 = 1$.  In fact, the same is true for the $\asemi{}{[\mu]}$ and $\airr{\mu}$.  This follows easily from the degenerations \eqref{eq:degR32} below, but also follows from the fact that the image of $\ccent{\slthree}$, the centraliser of the Cartan subalgebra, in the Zhu algebra of $\themm$ is commutative, itself a straightforward consequence of the results of \cite{PerVer08}.

Twisting an irreducible in the above list by a (nontrivial) $\grp{D}_6$-automorphism generally results in a new irreducible, meaning one not already listed above.  However, there are exceptions:
\begin{subequations} \label{eq:ident32}
	\begin{gather}
		\begin{gathered}
			\weylmod{}{\airr{0}} \cong \airr{0}, \quad
			\weylmod{1}{\airr{-3\fwt{2}/2}} \cong \airr{-3\fwt{2}/2}, \quad
			\weylmod{2}{\airr{-3\fwt{1}/2}} \cong \airr{-3\fwt{1}/2}, \\
			\weylmod{1}{\asemi{}{[\mu]}} \cong \asemi{}{[\weylmod{1}{\mu}]}, \quad
			\weylmod{}{\arel{}{[\mu]}} \cong \arel{}{[\weylmod{}{\mu}]}
		\end{gathered}
		\qquad \text{($\weylaut{} \in \weyl$)}, \label{eq:Wident32} \\
		\dynkmod{\airr{0}} \cong \airr{0}, \quad
		\dynkmod{\airr{-3\fwt{1}/2}} \cong \airr{-3\fwt{2}/2}, \quad
		\dynkmod{\airr{-\wvec/2}} \cong \airr{-\wvec/2}, \quad
		\dynkmod{\arel{}{[\mu]}} \cong \arel{}{[\dynkmod{\mu}]}. \label{eq:dident32}
	\end{gather}
	There are similar exceptions when the automorphism is a (nontrivial) spectral flow.  With the layout as defined in \cref{fig:sforbits}, these exceptions may be represented as follows:
	\begin{equation} \label{eq:sfident32}
		\begin{tikzpicture}[scale=2,<->,>=latex,baseline=(0)]
			\begin{scope}[shift={(0,0)}]
				\node (0) at (0:0) {$\airr{0}$};
				\node (-1) at (30:1) {$\airr{-3\fwt{1}/2}$};
				\node (-2) at (90:1) {$\airr{-3\fwt{2}/2}$};
				\node (1) at (150:1) {$\weylmod{1}{\airr{-3\fwt{1}/2}}$};
				\node (2) at (-30:1) {$\weylmod{2}{\airr{-3\fwt{2}/2}}$};
				\node (12) at (-150:1) {$\weylaut{1} \weylmod{2}{\airr{-3\fwt{2}/2}}$};
				\node (21) at (-90:1) {$\weylaut{2} \weylmod{1}{\airr{-3\fwt{1}/2}}$};
				\draw (0) -- (-1);
				\draw (0) -- (-2);
				\draw (0) -- (1);
				\draw (0) -- (2);
				\draw (0) -- (12);
				\draw (0) -- (21);
				\draw (12) -- (21);
				\draw (21) -- (2);
				\draw (2) -- (-1);
				\draw (-1) -- (-2);
				\draw (-2) -- (1);
				\draw (1) -- (12);
			\end{scope}
			\begin{scope}[shift={(2.25,0.75)}]
				\node (12) at (0:0) {$\weylaut{1} \weylmod{2}{\airr{-\wvec/2}}$};
				\node (0') at (30:1) {$\airr{-\wvec/2}$};
				\node (21) at (-30:1) {$\weylaut{2} \weylmod{1}{\airr{-\wvec/2}}$};
				\draw (12) -- (0');
				\draw (21) -- (0');
				\draw (12) -- (21);
			\end{scope}
			\begin{scope}[shift={(2.25,-1.25)}]
				\node (12) at (0:0) {$\weylmod{3}{\airr{-\wvec/2}}$};
				\node (0') at (30:1) {$\weylmod{2}{\airr{-\wvec/2}}$};
				\node (21) at (90:1) {$\weylmod{1}{\airr{-\wvec/2}}$};
				\draw (12) -- (0');
				\draw (21) -- (0');
				\draw (12) -- (21);
			\end{scope}
			\begin{scope}[shift={(5,0)}]
				\node (12) at (0,0.5) {$\asemi{}{[\mu]}$};
				\node (21) at (0,-0.5) {$\conjmod{\asemi{}{[-\mu-3\fwt{2}/2]}}$};
				\draw (12) -- (21);
			\end{scope}
		\end{tikzpicture}
	\end{equation}
\end{subequations}

For what follows, it will be important to know how the reducible (semi)relaxed \hw\ $\themm$-modules decompose.  The degenerations of the semirelaxed family take the form
\begin{equation} \label{eq:degS32}
	\asemi{}{[-3\fwt{1}/2]} \cong \airr{-3\fwt{1}/2} \oplus \weylmod{1}{\airr{-\wvec/2}} \quad \text{and} \quad
	\asemi{}{[-\wvec/2]} \cong \airr{-\wvec/2} \oplus \weylmod{1}{\airr{-3\fwt{1}/2}}.
\end{equation}
Similarly, the degenerations of the relaxed family for $[\mu]$ in $\sing{-\tfrac{3}{2} \fwt{1}}$ take the forms
\begin{equation} \label{eq:degR32}
	\arel{}{[\mu]} \cong
	\begin{cases*}
		\asemi{}{[\mu]} \oplus \conjmod{\asemi{}{[-\mu-\sroot{2}]}} & if $\mu \in -\frac{3}{2} \fwt{1} + \CC \sroot{1}$, \\
		\weylaut{1} \weylmod{2}{\asemi{}{[\weylaut{2} \cdot \weylmod{1}{\mu}]}} \oplus \weylaut{2} \dynkmod{\asemi{}{[\dynkaut \weylaut{2} \cdot \mu]}} & if $\mu \in -\frac{1}{2} \wvec + \CC \sroot{2}$, \\
		\weylmod{2}{\asemi{}{[\weylmod{2}{\mu}]}} \oplus \conjaut \weylmod{2}{\asemi{}{[\conjaut \weylaut{2} \cdot \mu]}} & if $\mu \in -\tfrac{3}{2} \fwt{1} + \CC \sroot{3}$.
	\end{cases*}
\end{equation}
We emphasise that one must always pick a representative $\mu$ from $[\mu] \in \dcsub / \rlat$ that satisfies the indicated condition.  For example, when $[\mu] = [-\frac{3}{2} \fwt{2}]$, we may take $\mu = -\frac{3}{2} \fwt{2}$ in the second degeneration of \eqref{eq:degR32} as $-\frac{3}{2} \fwt{2} = -\frac{1}{2} \wvec - \frac{1}{2} \sroot{2}$.  However, we should instead take $\mu = -\frac{3}{2} \fwt{2} + \sroot{2}$ in the third degeneration as $-\frac{3}{2} \fwt{2} + \sroot{2} = -\frac{3}{2} \fwt{1} + \frac{1}{2} \sroot{3}$.

Finally, for certain $[\mu] \in \dcsub / \rlat$, the relaxed modules further degenerate into \hwms:
\begin{equation} \label{eq:degRS32}
	\begin{aligned}
		\arel{}{[-3\fwt{1}/2]} &\cong \airr{-3\fwt{1}/2} \oplus \conjmod{\airr{-3\fwt{1}/2}} \oplus \weylmod{1}{\airr{-\wvec/2}} \oplus \conjaut \weylmod{1}{\airr{-\wvec/2}}, \\
		\arel{}{[-\wvec/2]} &\cong \airr{-\wvec/2} \oplus \conjmod{\airr{-\wvec/2}} \oplus \weylmod{1}{\airr{-3\fwt{1}/2}} \oplus \weylmod{2}{\airr{-3\fwt{2}/2}}, \\
		\arel{}{[-3\fwt{2}/2]} &\cong \airr{-3\fwt{2}/2} \oplus \conjmod{\airr{-3\fwt{2}/2}} \oplus \weylmod{2}{\airr{-\wvec/2}} \oplus \conjaut \weylmod{2}{\airr{-\wvec/2}}.
	\end{aligned}
\end{equation}
All of these degenerations are illustrated, for convenience, in \cref{fig:degenerations}.  In the language of the standard module formalism of \cite{CreLog13,RidVer14}, the \rhwms\ of \eqref{eq:degRS32} are \emph{atypical of degree $2$}, as are their direct summands.  The direct summands of the \rhwms\ of \eqref{eq:degR32}, excluding those of \eqref{eq:degRS32}, are then \emph{atypical of degree $1$} and the irreducible \rhwms\ are \emph{typical} (or atypical of degree $0$).

\begin{figure}
	\makebox[\textwidth][c]{
		\begin{tikzpicture}[scale=2,<->,>=latex]
			\begin{scope}[shift={(-1.5,3.2)}]
				\makeaxes{}{}
				\node[nom] at (0,1.3) {$\asemi{}{[-3\fwt{1}/2]}$};
				\thintriangle{0.45}{-60}{above:{$\scriptstyle -\frac{3}{2} \fwt{1}$}}{1.3}{$\airr{-3\fwt{1}/2}$}
				\thicktriangle{0.26}{30}{above right:{$\scriptstyle -\frac{3}{2} \fwt{1} + \sroot{1}$}}{0.8}{$\weylmod{1}{\airr{-\wvec/2}}$}
			\end{scope}
			\begin{scope}[shift={(1.5,3.2)}]
				\makeaxes{}{}
				\node[nom] at (0,1.3) {$\asemi{}{[-\wvec/2]}$};
				\thintriangle{0.45}{60}{above:{$\scriptstyle -\frac{1}{2} \wvec + \sroot{1}$}}{1.4}{$\weylmod{1}{\airr{-3\fwt{1}/2}}$}
				\thicktriangle{0.26}{-30}{above left:{$\scriptstyle -\frac{1}{2} \wvec$}}{0.8}{$\airr{-\wvec/2}$}
			\end{scope}
			\begin{scope}[shift={(0,0)}]
				\makeaxes{}{}
				\node[nom] at (0,1.6) {$\arel{}{[\mu]}$ ($\mu \in -\frac{3}{2} \fwt{1} + \CC \sroot{1}$)};
				\node[wt, label=above:$\scriptstyle -\frac{3}{2} \fwt{1}$] at (30:-0.45) {};
				\brick{0.225}{0}{below:{$\scriptstyle \mu$}}{0.6}
				\node at (-0.2,-1.2) {$\asemi{}{[\mu]}$};
				\brick{0.225}{180}{above left:$\scriptstyle \mu+\sroot{2}$}{0.53}
				\node at (0.4,1.17) {$\conjmod{\asemi{}{[-\mu-\sroot{2}]}}$};
			\end{scope}
			\begin{scope}[shift={(-3,0)}]
				\makeaxes{}{}
				\node[nom] at (0,1.6) {$\arel{}{[\mu]}$ ($\mu \in -\frac{1}{2} \wvec + \CC \sroot{2}$)};
				\node[wt, label=below left:$\scriptstyle -\frac{1}{2} \wvec$] at (60:-0.26) {};
				\brick{0.225}{-60}{right:{$\scriptstyle \mu$}}{0.75}
				\node at (-1.1,-1.1) {$\weylaut{2} \dynkmod{\asemi{}{[\dynkaut \weylaut{2} \cdot \mu]}}$};
				\brick{0.225}{120}{above right:$\scriptstyle \mu+\sroot{3}$}{0.38}
				\node at (1.25,1.15) {$\weylaut{1} \weylmod{2}{\asemi{}{[\weylaut{2} \cdot \weylmod{1}{\mu}]}}$};
			\end{scope}
			\begin{scope}[shift={(3,0)}]
				\makeaxes{}{}
				\node[nom] at (0,1.6) {$\arel{}{[\mu]}$ ($\mu \in -\frac{3}{2} \fwt{1} + \CC \sroot{3}$)};
				\node[wt, label=left:$\scriptstyle -\frac{3}{2} \fwt{1}$] at (30:-0.45) {};
				\brick{0.225}{-120}{left:$\scriptstyle \mu$}{0.55}
				\node at (-1.05,1.15) {$\weylmod{2}{\asemi{}{[\weylmod{2}{\mu}]}}$};
				\brick{0.225}{60}{right:$\scriptstyle \mu+\sroot{1}$}{0.58}
				\node at (1.1,-1.1) {$\conjaut \weylmod{2}{\asemi{}{[\conjaut \weylaut{2} \cdot \mu]}}$};
			\end{scope}
			\begin{scope}[shift={(2.5,-3.2)}]
				\makeaxes{}{}
				\node[nom] at (0,1.4) {$\arel{}{[-\wvec/2]}$};
				\thintriangle{0.45}{60}{above right:$\scriptstyle -\frac{1}{2} \wvec + \sroot{1}$}{1.4}{$\weylmod{1}{\airr{-3\fwt{1}/2}}$}
				\thicktriangle{0.26}{-30}{below left:{$\scriptstyle -\frac{1}{2} \wvec$}}{0.7}{$\airr{-\wvec/2}$}
				\thintriangle{0.45}{-120}{below left:$\scriptstyle -\frac{3}{2} \fwt{1}$}{1.4}{$\weylmod{2}{\airr{-3\fwt{2}/2}}$}
				\thicktriangle{0.26}{150}{above right:{$\scriptstyle -\frac{1}{2} \wvec$}}{0.8}{$\conjmod{\airr{-\wvec/2}}$}
			\end{scope}
			\begin{scope}[shift={(-2.5,-3.2)}]
				\makeaxes{}{}
				\node[nom] at (0,1.4) {$\arel{}{[-3\fwt{1}/2]}$};
				\thintriangle{0.45}{-60}{{[shift={(-75:0.8)}]:$\scriptstyle -\frac{3}{2} \fwt{1}$}}{1.2}{$\airr{-3\fwt{1}/2}$}
				\thicktriangle{0.26}{30}{above right:{$\scriptstyle -\frac{3}{2} \fwt{1} + \sroot{1}$}}{0.8}{$\weylmod{1}{\airr{-\wvec/2}}$}
				\thintriangle{0.45}{120}{{[shift={(175:0.15)}]:$\scriptstyle \frac{3}{2} \fwt{1}$}}{1.35}{$\conjmod{\airr{-3\fwt{1}/2}}$}
				\thicktriangle{0.26}{-150}{below left:{$\scriptstyle -\frac{3}{2} \fwt{1} + \sroot{1}$}}{0.85}{$\conjaut \weylmod{1}{\airr{-\wvec/2}}$}
			\end{scope}
			\begin{scope}[shift={(0,-4.8)}]
				\makeaxes{}{}
				\node[nom] at (0,2) {$\arel{}{[-3\fwt{2}/2]}$};
				\thintriangle{0.45}{0}{left:$\scriptstyle -\frac{3}{2} \fwt{2}$}{1.1}{$\airr{-3\fwt{2}/2}$}
				\thicktriangle{0.26}{-90}{left:{$\scriptstyle -\frac{3}{2} \fwt{2} + \sroot{2}$}}{1.1}{$\weylmod{2}{\airr{-\wvec/2}}$}
				\thintriangle{0.45}{180}{right:$\scriptstyle \frac{3}{2} \fwt{2}$}{1.1}{$\conjmod{\airr{-3\fwt{2}/2}}$}
				\thicktriangle{0.26}{90}{right:{$\scriptstyle -\frac{3}{2} \fwt{2} + \sroot{3}$}}{1.1}{$\conjaut \weylmod{2}{\airr{-\wvec/2}}$}
			\end{scope}
		\end{tikzpicture}
	}
\caption{The degenerations \eqref{eq:degS32} of the semirelaxed $\themm$-modules (top), the generic degenerations \eqref{eq:degR32} of the relaxed $\themm$-modules $\arel{}{[\mu]}$, $\mu \in \sing{-\frac{3}{2} \fwt{1}}$, (middle) and the special degenerations \eqref{eq:degRS32} of the relaxed $\themm$-modules (bottom).  In each case, the degeneration is illustrated by the (convex hulls of the) weights of the direct summands and the black circles indicate representative weights of the summands.} \label{fig:degenerations}
\end{figure}

\subsection{Standard characters and modularity} \label{sec:32modular}

The minimal \qhr\ of $\themm$ is the trivial \bp\ minimal model $\bpminmod{3}{2}$ of central charge $0$ \cite{AraRep05}.  It then follows from \cite{AraRep11} that
\begin{equation}
	H^0_{\textup{min.}}(\airr{-3\fwt{1}/2}) \cong \CC, \quad \text{hence} \quad \ch{H^0_{\textup{min.}}(\airr{-3\fwt{1}/2})}(\qq) = 1.
\end{equation}
By \cref{thm:rchar}, the characters of the \rhw\ $\themm$-modules are thus
\begin{equation} \label{eq:rchar32}
	\fch{\arel{}{[\mu]}}{\yy}{\zz}{\qq}
	= \frac{\yy^{-3/2}}{\eta(\qq)^4} \sum_{\nu \in [\mu]} \zz^{\nu}
	= \frac{\ee^{2 \pi \ii \brac[\big]{-3 \theta / 2 + \pair{\mu}{\zeta}}}}{\eta(\tau)^4} \sum_{\sroot{} \in \rlat} \ee^{2 \pi \ii \pair{\sroot{}}{\zeta}}, \quad [\mu] \in \dcsub / \rlat.
\end{equation}
We can manipulate the sum in this character formula by decomposing $\sroot{}$ as a linear combination of simple roots:
\begin{equation}
	\sum_{\sroot{} \in \rlat} \ee^{2 \pi \ii \pair{\sroot{}}{\zeta}}
	= \sum_{m_1, m_2 \in \ZZ} \ee^{2 \pi \ii \brac[\big]{\pair{\sroot{1}}{\zeta} m_1 + \pair{\sroot{2}}{\zeta} m_2}}
	= \prod_{i=1}^2 \sum_{m_i \in \ZZ} \ee^{2 \pi \ii \pair{\sroot{i}}{\zeta} m_i}
	= \prod_{i=1}^2 \sum_{n_1 \in \ZZ} \delta \brac[\big]{\pair{\sroot{i}}{\zeta} - n_i}.
\end{equation}
Defining the Dirac delta function on $\csub$ by
\begin{equation}
	\delta(\zeta) = \prod_{i=1}^2 \delta \brac[\big]{\pair{\sroot{i}}{\zeta}}, \quad \zeta \in \csub,
\end{equation}
the sum takes the compact form
\begin{equation} \label{eq:expdelta}
	\sum_{\sroot{} \in \rlat} \ee^{2 \pi \ii \pair{\sroot{}}{\zeta}}
	= \sum_{n_1, n_2 \in \ZZ} \prod_{i=1}^2 \delta \brac[\big]{\pair{\sroot{i}}{\zeta - n_i \fcwt{i}}}
	= \sum_{\fcwt{} \in \cwlat} \delta(\zeta - \fcwt{}),
\end{equation}
which we recognise as a Dirac comb supported on the coweight lattice.

We remark that this derivation immediately generalises to give
\begin{equation} \label{eq:expdelta'}
	\sum_{\sroot{} \in \grp{L}} \ee^{2 \pi \ii \pair{\sroot{}}{\zeta}}
	= \sum_{\gamma \in \grp{L}^\vee} \delta(\zeta - \gamma),
\end{equation}
where $\grp{L}$ is any lattice, $\zeta \in \grp{L} \otimes_{\ZZ} \CC$ and $\grp{L}^\vee$ its integer dual.  In particular, the characterisation of lattice Dirac combs derived in \eqref{eq:expdelta} continues to hold when we exchange the roles of $\rlat$ and $\cwlat$.

The spectral flows of the \rhw\ $\themm$-modules $\arel{}{[\mu]}$, with $[\mu] \in \rdcsub / \rlat$, are our candidate standard modules, where $\rdcsub$ denotes the subset of $\dcsub$ consisting of \emph{real} weights.  Substituting this back into \eqref{eq:rchar32} and applying \eqref{eq:autchar}, their characters are as follows.
\begin{proposition} \label{prop:stchar}
  The character of the spectrally flowed \rhw\ $\themm$-module $\sfmod{\fcwt{}}{\arel{}{[\mu]}}$, for any $\fcwt{} \in \cwlat$ and $[\mu] \in \dcsub / \rlat$, is
  \begin{equation} \label{eq:rchar32sf}
    \mch{\sfmod{\fcwt{}}{\arel{}{[\mu]}}}{\theta}{\zeta}{\tau}
    = \frac{\ee^{-3 \pi \ii \theta} \ee^{-3 \pi \ii \killing{\fcwt{}}{\zeta + \tau \fcwt{} / 2}} \ee^{2 \pi \ii \pair{\mu}{\zeta + \tau \fcwt{}}}}{\eta(\tau)^4} \sum_{\fcwt{}' \in \cwlat} \delta(\zeta + \tau \fcwt{} - \fcwt{}').
  \end{equation}
\end{proposition}
\noindent To confirm our choice of standard modules, we must show that the span of their characters carries a representation of the modular group $\mgrp = \ang{\mods, \modt \st \mods^4=\wun\ \text{and}\ (\mods \modt)^3 = \mods^2}$ and that their characters form a topological basis for the space of all characters of modules of $\awcat{3,2}$.

The modularity is the subject of our next result.  First, we propose the following action of $\mods$ and $\modt$ on the coordinates $\modcoords{\theta}{\zeta}{\tau} \in \CC \times \csub \times \HH$:
\begin{equation}
	\begin{aligned}
		\mods \colon \modcoords{\theta}{\zeta}{\tau}
		&\mapsto \modcoords{\theta - \frac{\killing{\zeta}{\zeta}}{2 \tau} - \frac{2 \arg(\tau) - \arg(-1)}{3 \pi}}{\frac{\zeta}{\tau}}{-\frac{1}{\tau}}, \\
	  \modt \colon \modcoords{\theta}{\zeta}{\tau}
	  &\mapsto \modcoords{\theta - \frac{\arg(-1)}{9 \pi}}{\zeta}{\tau+1}.
	\end{aligned}
\end{equation}
Here, $\arg$ denotes any choice of complex argument, for example the principal one.  It is easy to check, using $\arg(-\frac{1}{\tau}) = \arg(-1) - \arg(\tau)$, that this does indeed satisfy the defining relations of $\mgrp$.\footnote{The reader will no doubt recognise the action of $\mods$ and $\modt$ on $\theta$ as a somewhat strange-looking generalisation of the usual formulae familiar from rational models.  The terms involving complex arguments seem to be necessary to deal with the unusual automorphy factor that results from transforming Dirac combs.}
\begin{theorem} \label{thm:stcharmod}
	The modular S-transform of the standard $\themm$-module character \eqref{eq:rchar32sf} is given by
	\begin{subequations} \label{eq:stcharS}
		\begin{equation} \label{eq:stcharStrans}
	    \mods \set*{\ch{\sfmod{\fcwt{}}{\arel{}{[\mu]}}}}
	    = \sum_{\fcwt{}' \in \cwlat} \int_{\rdcsub / \rlat} \smat{\fcwt{}, \fcwt{}'}{[\mu], [\mu']} \ch{\sfmod{\fcwt{}'}{\arel{}{[\mu']}}} \, \dd [\mu'],
	  \end{equation}
	  where $\fcwt{} \in \cwlat$, $[\mu] \in \rdcsub / \rlat$ and
	  \begin{equation} \label{eq:stcharSmat}
	    \smat{\fcwt{}, \fcwt{}'}{[\mu], [\mu']} = \ee^{2 \pi \ii \brac[\big]{3 \killing{\fcwt{}}{\fcwt{}'} / 2 - \pair{\mu}{\fcwt{}'} - \pair{\mu'}{\fcwt{}}}}.
	  \end{equation}
	\end{subequations}
\end{theorem}
\begin{proof}
  We begin by simplifying the \lhs{} of \eqref{eq:stcharStrans}, using $\mods \set*{\eta(\tau)} = \sqrt{-\ii \tau} \eta(\tau)$ and $\delta(\zeta / \tau) = \abs{\tau}^2 \delta(\zeta)$:\footnote{This second identity is well known for delta functions with real arguments.  Here, as in many other applications of the standard module formalism, we assume that it may be extended to complex arguments.  We expect that this formula can be established rigorously by finding the correct space of test functions to pair with and hope to pursue this in future work.}
	\begin{align}
		\mods \set*{\ch{\sfmod{\fcwt{}}{\arel{}{[\mu]}}}}
		&= \ee^{-3 \pi \ii \brac[\big]{\theta - \killing{\zeta}{\zeta} / 2 \tau}} \frac{\ee^{\ii \brac[\big]{2 \arg(\tau) - \arg(-1)}}}{-\tau^2 \eta(\tau)^4} \ee^{-3 \pi \ii \killing{\fcwt{}}{\zeta - \fcwt{}/2} / \tau} \ee^{2 \pi \ii \pair{\mu}{\zeta - \fcwt{}} / \tau} \sum_{\fcwt{}' \in \cwlat} \delta \brac*{\frac{\zeta - \fcwt{} - \fcwt{}' \tau}{\tau}} \notag \\
		&= \frac{\ee^{-3 \pi \ii \theta} \abs{\tau}^2 \ee^{\ii \brac[\big]{2 \arg(\tau) - \arg(-1)}}}{-\tau^2 \eta(\tau)^4} \sum_{\fcwt{}' \in \cwlat} \ee^{3 \pi \ii \brac[\big]{\killing{\zeta}{\zeta/2} - \killing{\fcwt{}}{\zeta - \fcwt{}/2}} / \tau} \ee^{2 \pi \ii \pair{\mu}{\zeta - \fcwt{}} / \tau} \delta(\zeta + \fcwt{}' \tau - \fcwt{}).
	\end{align}
	Using the delta function to replace $\zeta$ everywhere in the exponentials by $\fcwt{} - \fcwt{}' \tau$, we arrive at
	\begin{equation} \label{eq:theanswer}
		\mods \set*{\ch{\sfmod{\fcwt{}}{\arel{}{[\mu]}}}} = \frac{\ee^{-3 \pi \ii \theta}}{\eta(\tau)^4} \sum_{\fcwt{}' \in \cwlat} \ee^{3 \pi \ii \killing{\fcwt{}'}{\fcwt{}'} \tau/2} \ee^{-2 \pi \ii \pair{\mu}{\fcwt{}'}} \delta(\zeta + \fcwt{}' \tau - \fcwt{}).
	\end{equation}
	Given the easily derived identity
	\begin{equation}
		\int_{\rdcsub / \rlat} \ee^{2 \pi \ii \pair{\mu}{\fcwt{}}} \, \dd [\mu] = \delta_{\fcwt{},0}, \quad \fcwt{} \in \cwlat,
	\end{equation}
	the \rhs{} of \eqref{eq:stcharStrans} likewise simplifies:
	\begin{align}
		\sum_{\fcwt{}' \in \cwlat} &\int_{\rdcsub / \rlat} \smat{\fcwt{}, \fcwt{}'}{[\mu], [\mu']} \ch{\sfmod{\fcwt{}'}{\arel{}{[\mu']}}} \, \dd [\mu'] \notag \\
		&= \sum_{\fcwt{}' \in \cwlat} \ee^{2 \pi \ii \brac[\big]{3 \killing{\fcwt{}}{\fcwt{}'} / 2 - \pair{\mu}{\fcwt{}'}}} \frac{\ee^{-3 \pi \ii \theta} \ee^{-3 \pi \ii \killing{\fcwt{}'}{\zeta + \fcwt{}' \tau/2}}}{\eta(\tau)^4}
		\sum_{\fcwt{}'' \in \cwlat} \int_{\rdcsub / \rlat} \ee^{2 \pi \ii \pair{\mu'}{\zeta + \fcwt{}' \tau - \fcwt{}}} \, \dd [\mu'] \, \delta(\zeta + \fcwt{}' \tau - \fcwt{}'') \notag \\
		&= \frac{\ee^{-3 \pi \ii \theta}}{\eta(\tau)^4} \sum_{\fcwt{}' \in \cwlat} \ee^{3 \pi \ii \killing{\fcwt{}}{\fcwt{}'}} \ee^{-2 \pi \ii \pair{\mu}{\fcwt{}'}} \ee^{-3 \pi \ii \killing{\fcwt{}'}{\zeta + \fcwt{}' \tau/2}} \delta(\zeta + \fcwt{}' \tau - \fcwt{}).
	\end{align}
	Replacing $\zeta$ in the exponential by $\fcwt{} - \fcwt{}' \tau$, this reproduces \eqref{eq:theanswer}.
\end{proof}
\noindent The corresponding explicit formula for the T-transform of the standard characters is easy to derive.  As we shall not need it, we omit this result.

Note that the ``S-matrix'' \eqref{eq:stcharSmat} is manifestly symmetric under exchanging primed and unprimed labels.  It is also unitary:
\begin{equation}
	\sum_{\fcwt{}'' \in \cwlat} \int_{\rdcsub / \rlat} \smat{\fcwt{}, \fcwt{}''}{[\mu], [\mu'']} \overline{\smat{\fcwt{}'', \fcwt{}'}{[\mu''], [\mu']}} \, \dd [\mu''] = \delta_{\fcwt{}, \fcwt{}'} \delta([\mu] - [\mu']).
\end{equation}
Finally, its square picks out the conjugation automorphism $\conjaut$ at the level of the standard module labels:
\begin{equation}
	\sum_{\fcwt{}'' \in \cwlat} \int_{\rdcsub / \rlat} \smat{\fcwt{}, \fcwt{}''}{[\mu], [\mu'']} \smat{\fcwt{}'', \fcwt{}'}{[\mu''], [\mu']} \, \dd [\mu''] = \delta_{\fcwt{}, -\fcwt{}'} \delta([\mu] + [\mu']).
\end{equation}
The properties generalise those from the familiar case of rational (and $C_2$-cofinite) \voas.  This is, in a sense, the hallmark of the standard module formalism.  In particular, it suggests that Verlinde computations will give meaningful answers for the (Grothendieck) fusion coefficients.

\subsection{Modularity of atypical characters} \label{sec:32atypical}

Our next task is to demonstrate that the characters of the remaining irreducible $\themm$-modules in $\awcat{3,2}$ may be expressed as infinite-linear combinations of standard characters.  This is achieved by constructing resolutions for the irreducibles in terms of the standard modules and applying the \ep\ principle to deduce the required character formulae and their modular transformations.

As we have chosen the standard modules to be completely reducible, we may actually skip the resolutions entirely and instead work directly with character formulae deduced from the degeneration formulae \eqref{eq:degS32}--\eqref{eq:degRS32}.  For example, combining the spectral flow identifications \eqref{eq:sfident32} with the degenerations \eqref{eq:degR32} for $\mu \in -\frac{3}{2} \fwt{1} + \CC \sroot{1}$ gives
\begin{equation} \label{eq:sfdegR32}
	\arel{}{[\mu]} \cong \asemi{}{[\mu]} \oplus \conjmod{\asemi{}{[-\mu-\sroot{2}]}}
	\cong \asemi{}{[\mu]} \oplus \sfmod{-\fcwt{2}}{\asemi{}{[\mu+\sroot{2}-3\fwt{2}/2]}}
	= \asemi{}{[\mu]} \oplus \sfmod{-\fcwt{2}}{\asemi{}{[\mu-\sroot{1}/2]}}.
\end{equation}
The corresponding character identity is thus
\begin{equation}
	\ch{\asemi{}{[\mu]}} = \ch{\arel{}{[\mu]}} - \ch{\sfmod{-\fcwt{2}}{\asemi{}{[\mu-\sroot{1}/2]}}}.
\end{equation}
Of course, we may replace $\mu$ by $\mu - \frac{1}{2} \sroot{1}$ in \eqref{eq:sfdegR32} and then apply the exact functor $\sfaut{-\fcwt{2}}$ to obtain the character identity
\begin{equation}
	\ch{\sfmod{-\fcwt{2}}{\asemi{}{[\mu-\sroot{1}/2]}}} = \ch{\sfmod{-\fcwt{2}}{\arel{}{[\mu-\sroot{1}/2]}}} - \ch{\sfmod{-2\fcwt{2}}{\asemi{}{[\mu]}}}.
\end{equation}
By substituting back and iterating this process, we arrive at the following formula expressing the semirelaxed characters as an infinite-linear combination of the standard ones.
\begin{lemma} \label{lem:atypchS}
	The character of the \srhw\ $\themm$-module $\asemi{}{[\mu]}$, $[\mu] \in -\frac{3}{2} \fwt{1} + \CC \sroot{1}$, satisfies
	\begin{equation} \label{eq:atypchS}
		\ch{\asemi{}{[\mu]}} = \sum_{n=0}^{\infty} \brac*{\ch{\sfmod{-2n\fcwt{2}}{\arel{}{[\mu]}}} - \ch{\sfmod{-(2n+1)\fcwt{2}}{\arel{}{[\mu-\sroot{1}/2]}}}}.
	\end{equation}
\end{lemma}

Of course, we could have instead replaced $\mu$ by $\mu + \frac{1}{2} \sroot{1}$ in \eqref{eq:sfdegR32} and applied $\sfaut{\fcwt{2}}$ to get
\begin{equation}
	\ch{\asemi{}{[\mu]}} = \ch{\sfmod{\fcwt{2}}{\arel{}{[\mu+\sroot{1}/2]}}} - \ch{\sfmod{\fcwt{2}}{\asemi{}{[\mu+\sroot{1}/2]}}}.
\end{equation}
Substituting and iterating, as above, then results in an expression for the semirelaxed characters as a different infinite-linear combination of standard characters, this time involving spectral flows with positive multiples of $\fcwt{2}$:
\begin{equation} \label{eq:atypchS'}
	\ch{\asemi{}{[\mu]}} = \sum_{n=1}^{\infty} \brac*{\ch{\sfmod{(2n-1)\fcwt{2}}{\arel{}{[\mu+\sroot{1}/2]}}} - \ch{\sfmod{2n\fcwt{2}}{\arel{}{[\mu]}}}}.
\end{equation}
This difference should be interpreted as corresponding to a different topological completion of the vector space spanned by the standard characters.

More precisely, the completion implicit in \eqref{eq:atypchS} allows for infinitely many terms with spectral flows by negative multiples of the coweight $\fcwt{2}$, but only finitely many terms with positive multiples. On the other hand, the completion implicit in \eqref{eq:atypchS'} allows for infinitely many terms with positive multiples of $\fcwt{2}$ but only finitely negative multiples. This is akin to power series completions of Laurent polynomials where one allows the exponent of the variable to be unbounded either above or below, but not both. A geometric interpretation is that we have completed along a ray in either the $-\fcwt{2}$ or the $\fcwt{2}$ direction, respectively.

\begin{proposition} \label{prop:atypcharS}
	The modular S-transforms of the spectral flows of the semirelaxed $\themm$-module characters are given by
	\begin{subequations} \label{eq:atypcharS}
		\begin{equation} \label{eq:atypcharStrans}
	    \mods \set*{\ch{\sfmod{\fcwt{}}{\asemi{}{[\mu]}}}}
	    = \sum_{\fcwt{}' \in \cwlat} \int_{\rdcsub / \rlat} \asmat{\fcwt{}, \fcwt{}'}{[\mu], [\mu']} \ch{\sfmod{\fcwt{}'}{\arel{}{[\mu']}}} \, \dd [\mu'],
	  \end{equation}
	  where $\fcwt{} \in \cwlat$, $[\mu] \in (-\frac{3}{2} \fwt{1} + \RR \sroot{1}) / \ZZ \sroot{1}$ and
	  \begin{equation} \label{eq:atypcharSmat}
	    \asmat{\fcwt{}, \fcwt{}'}{[\mu], [\mu']} = \frac{\ee^{-\pi \ii \pair{\mu'}{\fcwt{2}}}}{2 \cos \brac[\big]{\pi \pair{\mu'}{\fcwt{2}}}} \smat{\fcwt{}, \fcwt{}'}{[\mu], [\mu']}.
	  \end{equation}
	\end{subequations}
	Here, the semirelaxed S-matrix entry \eqref{eq:atypcharSmat} is to be expanded as a geometric series in either $\ee^{2 \pi \ii \pair{\mu'}{\fcwt{2}}}$ or $\ee^{-2 \pi \ii \pair{\mu'}{\fcwt{2}}}$, depending on whether \eqref{eq:atypchS} or \eqref{eq:atypchS'}, respectively, is used.
\end{proposition}
\begin{proof}
	Substituting \eqref{eq:stcharStrans} into the $\sfaut{\fcwt{}}$-spectral flow of \eqref{eq:atypchS} gives \eqref{eq:atypcharStrans}, with
	\begin{equation}
		\asmat{\fcwt{}, \fcwt{}'}{[\mu], [\mu']} = \sum_{n=0}^{\infty} \brac*{\smat{\fcwt{}-2n\fcwt{2}, \fcwt{}'}{[\mu], [\mu']} - \smat{\fcwt{}-(2n+1)\fcwt{2}, \fcwt{}'}{[\mu-\sroot{1}/2], [\mu']}}.
	\end{equation}
	Substituting \eqref{eq:stcharSmat} now gives
	\begin{align}
		\asmat{\fcwt{}, \fcwt{}'}{[\mu], [\mu']}
		&= \smat{\fcwt{}, \fcwt{}'}{[\mu], [\mu']} \sum_{n=0}^{\infty} \ee^{2 \pi \ii n \brac[\big]{-3 \killing{\fcwt{2}}{\fcwt{}'} + 2 \pair{\mu'}{\fcwt{2}}}} \cdot \brac*{1 - \ee^{2 \pi \ii \brac[\big]{-3 \killing{\fcwt{2}}{\fcwt{}'}/2 + \pair{\sroot{1}}{\fcwt{}'} / 2 + \pair{\mu'}{\fcwt{2}}}}} \notag \\
		&= \smat{\fcwt{}, \fcwt{}'}{[\mu], [\mu']} \sum_{n=0}^{\infty} \ee^{2 \pi \ii n \pair{\mu'}{2 \fcwt{2}}} \cdot \brac*{1 - \ee^{2 \pi \ii \pair{\mu'}{\fcwt{2}}}},
	\end{align}
	which simplifies to \eqref{eq:atypcharSmat} upon summing the geometric series in $\ee^{2 \pi \ii \pair{\mu'}{\fcwt{2}}}$.  Here, it is useful to note that $\killing{3 \fcwt{2}}{\fcwt{}'} = \pair{3 \fwt{2}}{\fcwt{}'} \in \ZZ$, since $3 \fwt{2} = \sroot{1} + 2 \sroot{2} \in \rlat$, and
	\begin{equation}
		-\tfrac{3}{2} \killing{\fcwt{2}}{\fcwt{}'} + \tfrac{1}{2} \pair{\sroot{1}}{\fcwt{}'} = \pair{-\tfrac{3}{2} \fwt{2} + \sroot{1}}{\fcwt{}'} = \pair{-\sroot{2}}{\fcwt{}'} \in \ZZ.
	\end{equation}
	It is easy to check that using the spectral flow of \eqref{eq:atypchS'} instead of \eqref{eq:atypchS} gives the same result after summing a geometric series in $\ee^{-2 \pi \ii \pair{\mu'}{\fcwt{2}}}$.
\end{proof}

Unlike the situation for the relaxed modules, there are other \srhw\ $\themm$-modules that may be obtained from the $\asemi{}{[\mu]}$ through twisting by an automorphism of $\grp{D}_6$.  This is however easily accommodated.
\begin{theorem} \label{thm:atypcharS}
	The modular S-transforms of the twists of the semirelaxed $\themm$-module characters are given by
	\begin{subequations} \label{eq:atypcharS'}
		\begin{equation} \label{eq:atypcharStrans'}
	    \mods \set*{\ch{\sfaut{\fcwt{}} \func{\omega}{\asemi{}{[\mu]}}}}
	    = \sum_{\fcwt{}' \in \cwlat} \int_{\rdcsub / \rlat} \asmat{\fcwt{}, \fcwt{}'; \omega}{[\mu], [\mu']} \ch{\sfmod{\fcwt{}'}{\arel{}{[\mu']}}} \, \dd [\mu'],
	  \end{equation}
	  where $\fcwt{} \in \cwlat$, $\omega \in \grp{D}_6$, $[\mu] \in (-\frac{3}{2} \fwt{1} + \RR \sroot{1}) / \ZZ \sroot{1}$ and
	  \begin{equation} \label{eq:atypcharSmat'}
	    \asmat{\fcwt{}, \fcwt{}'; \omega}{[\mu], [\mu']} = \frac{\ee^{-\pi \ii \pair{\mu'}{\omega(\fcwt{2})}}}{2 \cos \brac[\big]{\pi \pair{\mu'}{\omega(\fcwt{2})}}} \smat{\fcwt{}, \fcwt{}'}{[\omega(\mu)], [\mu']}.
	  \end{equation}
	\end{subequations}
\end{theorem}
\begin{proof}
	The character of $\func{\omega}{\asemi{}{[\mu]}}$ is simply
	\begin{align} \label{eq:atypchSD6}
		\ch{\func{\omega}{\asemi{}{[\mu]}}}
		&= \sum_{n=0}^{\infty} \brac*{\ch{\omega \sfmod{-2n\fcwt{2}}{\arel{}{[\mu]}}} - \ch{\omega \sfmod{-(2n+1)\fcwt{2}}{\arel{}{[\mu-\sroot{1}/2]}}}} \notag \\
		&= \sum_{n=0}^{\infty} \brac*{\ch{\sfaut{-2n \omega(\fcwt{2})} \func{\omega}{\arel{}{[\mu]}}} - \ch{\sfaut{-(2n+1) \omega(\fcwt{2})} \func{\omega}{\arel{}{[\mu-\sroot{1}/2]}}}} \notag \\
		&= \sum_{n=0}^{\infty} \brac*{\ch{\sfmod{-2n \omega(\fcwt{2})}{\arel{}{[\omega(\mu)]}}} - \ch{\sfmod{-(2n+1) \omega(\fcwt{2})}{\arel{}{[\omega(\mu)-\omega(\sroot{1})/2]}}}},
	\end{align}
	by \cref{prop:Wtwists,prop:dtwists,lem:atypchS}.  Applying $\sfaut{\fcwt{}}$, we can use \cref{prop:atypcharS} with $[\mu] \mapsto [\omega(\mu)]$, $\fcwt{2} \mapsto \omega(\fcwt{2})$ and $\sroot{1} \mapsto \omega(\sroot{1})$ (the simplifications in the proof continue to hold because $\omega$ preserves $\rlat$).
\end{proof}
\noindent Obviously, the result would remain unchanged if we had started with the character formula \eqref{eq:atypchS'} instead.

It therefore only remains to determine character formulae and S-transforms for the \hw\ $\themm$-modules.  Unlike the semirelaxed case, in which there were two choices for the spectral flow ``direction'' in the character formulae ($-\fcwt{2}$ in \eqref{eq:atypchS} and $\fcwt{2}$ in \eqref{eq:atypchS'}), there are now many choices.  We shall only consider a single choice for simplicity, omitting the check that other choices reproduce the same S-transformation formulae (up to expanding geometric series in different domains).

\begin{lemma} \label{lem:atypchL}
	We have the following \hw\ $\themm$-character formula:
		\begin{equation} \label{eq:atypchLthin}
			\ch{\airr{-3\fwt{1}/2}} = \sum_{n=0}^{\infty} \brac*{\ch{\sfaut{-2n \fcwt{1}} \weylmod{2}{\asemi{}{[-3\fwt{1}/2]}}} - \ch{\sfmod{-2n \fcwt{1} - \fcwt{2}}{\asemi{}{[-\wvec/2]}}}}.
		\end{equation}
\end{lemma}
\begin{proof}
	We start with the degenerations of \eqref{eq:degS32}, applying $\weylaut{2}$ to the first:
	\begin{subequations} \label{eq:Sdeg}
		\begin{align}
			\weylmod{2}{\asemi{}{[-3\fwt{1}/2]}}
			&\cong \weylmod{2}{\airr{-3\fwt{1}/2}} \oplus \weylaut{2} \weylmod{1}{\airr{-\wvec/2}}
			\cong \airr{-3\fwt{1}/2} \oplus \sfmod{-\fcwt{2}}{\airr{-\wvec/2}}, \label{eq:Sdeg1} \\
			\asemi{}{-\wvec/2}
			&\cong \airr{-\wvec/2} \oplus \weylmod{1}{\airr{-3\fwt{1}/2}}
			\cong \airr{-\wvec/2} \oplus \sfmod{-2\fcwt{1}+\fcwt{2}}{\airr{-3\fwt{1}/2}}. \label{eq:Sdeg2}
		\end{align}
	\end{subequations}
	Apply $\sfaut{-\fcwt{2}}$ to \eqref{eq:Sdeg2}, take its character and substitute into the character of \eqref{eq:Sdeg1}.  The result is
	\begin{equation}
		\ch{\airr{-3\fwt{1}/2}} = \ch{\weylmod{2}{\asemi{}{[-3\fwt{1}/2]}}} - \ch{\sfmod{-\fcwt{2}}{\asemi{}{-\wvec/2}}} + \ch{\sfmod{-2\fcwt{1}}{\airr{-3\fwt{1}/2}}}.
	\end{equation}
	Iterating this gives \eqref{eq:atypchLthin}.
\end{proof}

As both $\airr{0}$ and $\airr{-3\fwt{2}/2}$ may be obtained from $\airr{-3\fwt{1}/2}$ by twisting, their characters may be expressed as an infinite-linear combination of standard characters.  The same is true for $\airr{-\wvec/2}$, as \eqref{eq:Sdeg} makes clear, hence it is true for the character of every module in $\awcat{3,2}$.  Along with the modularity result \cref{thm:stcharmod}, this confirms our choice of the $\sfmod{\fcwt{}}{\arel{}{[\mu]}}$, with $\fcwt{} \in \cwlat$ and $[\mu] \in \rdcsub / \rlat$, as the standard modules of $\themm$.

Notice that, as with the character formula of \cref{lem:atypchS}, we have chosen a completion of the vector space spanned by the standard characters to obtain \eqref{eq:atypchLthin}, specifically that all but finitely many spectral flow coweights must lie in the cone spanned by \(-\fcwt{2}\) and \(\fcwt{2}-\fcwt{1}\). (Observe that $\weylaut{2}$ reflects the coweight \(-\fcwt{2}\) coming from the semidense module in the first summand of \eqref{eq:atypchLthin} to \(\fcwt{2}-\fcwt{1}\)). Therefore, every simple module admits a character formula with all but finitely many spectral flow indices in this cone. A different choice of completion would have resulted in a different cone. Specifically the choices of completions are in bijection with the cones spanned by any of the six pairs of coweights in the set \(\set{\pm\fcwt{1},\pm\fcwt{2},\pm(\fcwt{1}-\fcwt{2})}\) whose relative angle is \(2\pi/3\).

\begin{theorem} \label{thm:atypcharL}
	The modular S-transforms of the spectral flows of the \hw\ $\themm$-module characters are given by
	\begin{subequations} \label{eq:atypcharL}
		\begin{equation} \label{eq:atypcharLtrans}
	    \mods \set*{\ch{\sfmod{\fcwt{}}{\airr{\mu}}}}
	    = \sum_{\fcwt{}' \in \cwlat} \int_{\rdcsub / \rlat} \aasmat{\fcwt{}, \fcwt{}'}{\mu, [\mu']} \ch{\sfmod{\fcwt{}'}{\arel{}{[\mu']}}} \, \dd [\mu'],
	  \end{equation}
	  where $\fcwt{} \in \cwlat$ and $\mu \in \set{0, -\frac{3}{2}\fwt{1}, -\frac{3}{2}\fwt{2}, -\frac{1}{2}\wvec}$.  In particular, we have
    \begin{equation} \label{eq:atypcharLmat}
			\aasmat{\fcwt{}, \fcwt{}'}{-3\fwt{1}/2, [\mu']}
			= \frac{\ee^{2 \pi \ii \brac[\big]{3 \killing{\fcwt{}+\fcwt{1}}{\fcwt{}'}/2 - \pair{\mu'}{\fcwt{}+\fcwt{1}}}}}{2 \brac*{1 + \cos \brac[\big]{2 \pi \pair{\mu'}{\fcwt{1}}} + \cos \brac[\big]{2 \pi \pair{\mu'}{\fcwt{2}}} + \cos \brac[\big]{2 \pi \pair{\mu'}{\fcwt{1}-\fcwt{2}}}}}.
		\end{equation}
	\end{subequations}
\end{theorem}
\begin{proof}
	These calculations are similar to the proof of \cref{prop:atypcharS}.  First, note that \eqref{eq:atypcharSmat'} and \eqref{eq:atypchLthin} give
	\begin{align}
		\aasmat{\fcwt{}, \fcwt{}'}{-3\fwt{1}/2, [\mu']}
		&= \sum_{n=0}^{\infty} \brac*{\asmat{\fcwt{} - 2n \fcwt{1}, \fcwt{}'; \weylaut{2}}{[-3\fwt{1}/2], [\mu']} - \asmat{\fcwt{} - 2n \fcwt{1} - \fcwt{2}, \fcwt{}'}{[-\wvec/2], [\mu']}} \notag \\
		&= \frac{\ee^{-\pi \ii \pair{\mu'}{\fcwt{1}-\fcwt{2}}}}{2 \cos \brac[\big]{\pi \pair{\mu'}{\fcwt{1}-\fcwt{2}}}} \sum_{n=0}^{\infty} \smat{\fcwt{} - 2n \fcwt{1}, \fcwt{}'}{[-3\fwt{1}/2], [\mu']}
		- \frac{\ee^{-\pi \ii \pair{\mu'}{\fcwt{2}}}}{2 \cos \brac[\big]{\pi \pair{\mu'}{\fcwt{2}}}} \sum_{n=0}^{\infty} \smat{\fcwt{} - 2n \fcwt{1} - \fcwt{2}, \fcwt{}'}{[-\wvec/2], [\mu']}.
	\end{align}
	Substituting \eqref{eq:stcharSmat} and summing the geometric series now gives
	\begin{equation}
		\aasmat{\fcwt{}, \fcwt{}'}{-3\fwt{1}/2, [\mu']}
		= \frac{\ee^{2 \pi \ii \brac[\big]{3\pair{\fcwt{}+\fcwt{1}}{\fcwt{}'}/2 - \pair{\mu'}{\fcwt{}}}}}{1 - \ee^{2 \pi \ii \pair{\mu'}{2 \fcwt{1}}}}
		\sqbrac*{\frac{\ee^{-\pi \ii \pair{\mu'}{\fcwt{1}-\fcwt{2}}}}{2 \cos \brac[\big]{\pi \pair{\mu'}{\fcwt{1}-\fcwt{2}}}} - \frac{\ee^{+\pi \ii \pair{\mu'}{\fcwt{2}}}}{2 \cos \brac[\big]{\pi \pair{\mu'}{\fcwt{2}}}}},
	\end{equation}
	after simplifying appropriately.  The rest is trigonometric identities.
\end{proof}
\noindent This can of course be generalised to include $\grp{D}_6$-twists.  We leave this detail as an exercise.  Setting $\fcwt{} = -\fcwt{1}$ in \cref{thm:atypcharL} gives us the S-transform of the vacuum character.
\begin{corollary} \label{cor:Svac}
	The modular S-transform of the character of the vacuum module $\airr{0}$ of $\themm$ is given by \eqref{eq:atypcharLtrans}, where
	\begin{equation} \label{eq:Svac}
		\aasmat{0, \fcwt{}'}{0, [\mu']} = \frac{1}{2 \brac[\Big]{1 + \cos \brac[\big]{2 \pi \pair{\mu'}{\fcwt{1}}} + \cos \brac[\big]{2 \pi \pair{\mu'}{\fcwt{2}}} + \cos \brac[\big]{2 \pi \pair{\mu'}{\fcwt{1}-\fcwt{2}}}}}.
	\end{equation}
\end{corollary}

\subsection{Grothendieck fusion rules} \label{sec:32fusion}

We have identified a set of standard modules for the $\slthree$ minimal model $\themm$, computed their characters and modular S-transforms and determined explicit formulae that express the character of any module in $\awcat{3,2}$ in the basis of standard characters.  To test the consistency of our results, we shall apply the standard Verlinde formula of \cite{CreLog13,RidVer14} to determine the multiplicities of the Grothendieck fusion rules of $\themm$.  A highly nontrivial check of this consistency is whether these multiplicities are indeed nonnegative integers.  We remark that the standard module formalism involves ideas and manipulations which may be unfamiliar to those versed in the modularity of rational \voas.  In our opinion, this is an unavoidable consequence of the existence of modules that are not completely reducible and we refer to the original articles for further discussion and motivations.

In point of fact, we can only conjecture that the numbers computed by the standard Verlinde formula are actually the multiplicities appearing in the Grothendieck fusion rules.  The minimal model $\themm$ is neither rational nor $C_2$-cofinite and its standard modules are not even $C_1$-cofinite.  It is therefore not clear that the fusion multiplicities need be finite.  Nor is it even clear that the fusion product $\fuse$ induces a well defined product $\Grfuse$ on the Grothendieck group of $\awcat{3,2}$ (we do not know if fusing with a given module defines an exact functor on $\awcat{3,2}$).

We nevertheless expect that all these technical objections can somehow be overcome.  Let $\awcat{3,2}'$ denote the full subcategory of $\awcat{3,2}$ whose objects are $\themm$-modules with \emph{real} weights.  We conjecture that $\awcat{3,2}'$ is a rigid braided tensor category, hence that the Grothendieck fusion ring is well defined.\footnote{This rigidity conjecture is very natural as all rational \voas\ are known to produce rigid module categories (modular tensor categories even) \cite{HuaVer04} and a growing number of nonrational \voas\ are also known to admit rigid module categories \cite{TsuTen12,AllBos20,CreTen20,CreTen20c,CreRib20,CreTen20b}.  There is, however, a known counterexample \cite{GabFus09}.  A modified Grothendieck fusion ring for this counterexample was studied in detail in \cite{RidMod13}.}  We shall additionally conjecture that the structure constants of this ring, the Grothendieck fusion multiplicities, are indeed computed by the standard Verlinde formula.  These conjectures are understood to be in force for the remainder of this \lcnamecref{sec:32fusion}.  (Because of these conjectures, we will not state the results in this \lcnamecref{sec:32fusion} as theorems.)

With this understanding, the standard Verlinde formula amounts to the following statement.  The Grothendieck fusion rules of the modules of $\awcat{3,2}'$ take the form
\begin{subequations} \label{eq:svf}
	\begin{equation} \label{eq:grfr}
		\Gr{\Mod{M}} \Grfuse \Gr{\Mod{N}} = \sum_{\fcwt{}'' \in \cwlat} \int_{\rdcsub / \rlat} \fuscoeff{\Mod{M}}{\Mod{N}}{\sfmod{\fcwt{}''}{\arel{}{[\mu'']}}} \Gr{\sfmod{\fcwt{}''}{\arel{}{[\mu'']}}} \, \dd [\mu''],
	\end{equation}
	where $\Gr{\Mod{M}}$ denotes the image of the $\themm$-module $\Mod{M}$ in the Grothendieck ring and the Grothendieck fusion coefficients are given by
	\begin{equation} \label{eq:grfc}
		\fuscoeff{\Mod{M}}{\Mod{N}}{\sfmod{\fcwt{}''}{\arel{}{[\mu'']}}} = \sum_{\fcwt{}''' \in \cwlat} \int_{\rdcsub / \rlat} \frac{\mods \sqbrac[\big]{\Mod{M}, \sfmod{\fcwt{}''}{\arel{}{[\mu'']}}} \mods \sqbrac[\big]{\Mod{N}, \sfmod{\fcwt{}''}{\arel{}{[\mu'']}}} \brac*{\smat{\fcwt{}'', \fcwt{}'''}{[\mu''], [\mu''']}}^*}{\aasmat{0, \fcwt{}'''}{0, [\mu''']}} \, \dd [\mu'''].
	\end{equation}
\end{subequations}
Here, $(\blank)^*$ denotes complex conjugation and the S-matrix entries involving $\Mod{M}$ and $\Mod{N}$ have to be interpreted as those appropriate to their atypicality degrees.

Before we start applying the standard Verlinde formula, there are two simplifications to note.  Consider first the effect of a $\grp{D}_6$-twist $\omega$ on the standard Grothendieck fusion coefficients.  Using the explicit formulae \eqref{eq:stcharSmat} and \eqref{eq:Svac} for the standard and vacuum S-matrix entries, as well as the fact that $\cwlat$ and $\rlat$ are $\grp{D}_6$-invariant, we see that
\begin{align}
	\fuscoeff{\omega \sfmod{\fcwt{}}{\arel{}{[\mu]}}}{\omega \sfmod{\fcwt{}'}{\arel{}{[\mu']}}}{\omega \sfmod{\fcwt{}''}{\arel{}{[\mu'']}}}
	&= \fuscoeff{\sfmod{\omega(\fcwt{})}{\arel{}{[\omega(\mu)]}}}{\sfmod{\omega(\fcwt{}')}{\arel{}{[\omega(\mu')]}}}{\sfmod{\omega(\fcwt{}'')}{\arel{}{[\omega(\mu'')]}}} \notag \\
	&= \sum_{\fcwt{}''' \in \cwlat} \int_{\rdcsub / \rlat} \frac{\smat{\omega(\fcwt{}), \fcwt{}'''}{[\omega(\mu)], [\mu''']} \smat{\omega(\fcwt{}'), \fcwt{}'''}{[\omega(\mu')], [\mu''']} \brac*{\smat{\omega(\fcwt{}''), \fcwt{}'''}{[\omega(\mu'')], [\mu''']}}^*}{\aasmat{0, \fcwt{}'''}{0, [\mu''']}} \, \dd [\mu'''] \notag \\
	&=\sum_{\fcwt{}''' \in \cwlat} \int_{\rdcsub / \rlat} \frac{\smat{\fcwt{}, \omega^{-1}(\fcwt{}''')}{[\mu], [\omega^{-1}(\mu''')]} \smat{\fcwt{}', \omega^{-1}(\fcwt{}''')}{[\mu'], [\omega^{-1}(\mu''')]} \brac*{\smat{\fcwt{}'', \omega^{-1}(\fcwt{}''')}{[\mu''], [\omega^{-1}(\mu''')]}}^*}{\aasmat{0, \omega^{-1}(\fcwt{}''')}{0, [\omega^{-1}(\mu''')]}} \, \dd [\mu'''] \notag \\
	&= \sum_{\fcwt{}''' \in \cwlat} \int_{\rdcsub / \rlat} \frac{\smat{\fcwt{}, \fcwt{}'''}{[\mu], [\mu''']} \smat{\fcwt{}', \fcwt{}'''}{[\mu'], [\mu''']} \brac*{\smat{\fcwt{}'', \fcwt{}'''}{[\mu''], [\mu''']}}^*}{\aasmat{0, \fcwt{}'''}{0, [\mu''']}} \, \dd [\mu'''] \notag \\
	&= \fuscoeff{\sfmod{\fcwt{}}{\arel{}{[\mu]}}}{\sfmod{\fcwt{}'}{\arel{}{[\mu']}}}{\sfmod{\fcwt{}''}{\arel{}{[\mu'']}}}.
\end{align}
Because the standard characters form a basis for the space spanned by the characters of the modules of $\awcat{3,2}'$, this $\grp{D}_6$-invariance also holds for general Grothendieck fusion coefficients.  The Grothendieck fusion rules \eqref{eq:grfr} therefore remain true if one applies $\omega$ to each module (on both sides):
\begin{equation} \label{eq:grD6}
	\Gr{\omega(\Mod{M})} \Grfuse \Gr{\omega(\Mod{N})} = \sum_{\fcwt{}'' \in \cwlat} \int_{\rdcsub / \rlat} \fuscoeff{\Mod{M}}{\Mod{N}}{\sfmod{\fcwt{}''}{\arel{}{[\mu'']}}} \Gr{\omega(\sfmod{\fcwt{}''}{\arel{}{[\mu'']}})} \, \dd [\mu''].
\end{equation}
The upshot of this is that, if desired, one may restrict to Grothendieck fusion rules in which only one of the two $\themm$-modules being fused is $\grp{D}_6$-twisted.  We note that this $\grp{D}_6$-invariance of Grothendieck fusion rules is expected because genuine fusion rules are known to satisfy a similar invariance property \cite{AdaFus19}.

There is a slightly different invariance formula for Grothendieck fusion rules and spectral flow, again expected because of the known version for genuine fusion rules \cite{LiPhy97}.  To deduce it, note that \eqref{eq:stcharSmat} gives
\begin{equation}
	\smat{\fcwt{},\fcwt{}'''}{[\mu],[\mu''']} = \ee^{2 \pi \ii \brac[\big]{3 \killing{\fcwt{}}{\fcwt{}'''}/2 - \pair{\mu'''}{\fcwt{}}}} \smat{0,\fcwt{}'''}{[\mu],[\mu''']}.
\end{equation}
In the formula for the standard Grothendieck fusion coefficients, we can extract these phases from the unprimed and singly primed S-matrix elements and then absorb them into the doubly primed one.  The result is
\begin{equation}
	\fuscoeff{\sfmod{\fcwt{}}{\arel{}{[\mu]}}}{\sfmod{\fcwt{}'}{\arel{}{[\mu']}}}{\sfmod{\fcwt{}''}{\arel{}{[\mu'']}}}
	= \fuscoeff{\arel{}{[\mu]}}{\arel{}{[\mu']}}{\sfmod{\fcwt{}''-\fcwt{}'-\fcwt{}}{\arel{}{[\mu'']}}},
\end{equation}
where the minus signs arise because of complex conjugation.  As above, this implies that the Grothendieck fusion rules \eqref{eq:grfr} continue to hold upon applying spectral flow, as long as the product of the spectral flows on the \lhs\ matches that on the \rhs.  Thus,
\begin{equation} \label{eq:grsf}
	\Gr{\sfmod{\fcwt{}}{\Mod{M}}} \Grfuse \Gr{\sfmod{\fcwt{}'}{\Mod{N}}} = \sum_{\fcwt{}'' \in \cwlat} \int_{\rdcsub / \rlat} \fuscoeff{\Mod{M}}{\Mod{N}}{\sfmod{\fcwt{}''}{\arel{}{[\mu'']}}} \Gr{\sfmod{\fcwt{}+\fcwt{}'+\fcwt{}''}{\arel{}{[\mu'']}}} \, \dd [\mu''].
\end{equation}
If desired, one may therefore restrict to Grothendieck fusion rules in which there are no explicit spectral flow twists on the \lhs.

Let us now take $\Mod{M}$ and $\Mod{N}$ to be \rhw\ $\themm$-modules (with real weights).  The Grothendieck fusion coefficients are relatively straightforward to compute:
\begin{align}
	&\fuscoeff{\arel{}{[\mu]}}{\arel{}{[\mu']}}{\sfmod{\fcwt{}''}{\arel{}{[\mu'']}}}
	= \sum_{\fcwt{}''' \in \cwlat} \int_{\rdcsub / \rlat} \frac{\smat{0, \fcwt{}'''}{[\mu], [\mu''']} \smat{0, \fcwt{}'''}{[\mu'], [\mu''']} \brac*{\smat{\fcwt{}'', \fcwt{}'''}{[\mu''], [\mu''']}}^*}{\aasmat{0, \fcwt{}'''}{0, [\mu''']}} \, \dd [\mu'''] \notag \\
	&\quad =\sum_{\fcwt{}''' \in \cwlat} \ee^{2 \pi \ii \pair{-3(\fcwt{}'')^*/2 - \mu-\mu'+\mu''}{\fcwt{}'''}} \notag \\
	&\quad\quad \cdot \int_{\rdcsub / \rlat} \ee^{2 \pi \ii \pair{\mu'''}{\fcwt{}''}} \, 2 \brac[\Big]{1 + \cos \brac[\big]{2 \pi \pair{\mu'''}{\fcwt{1}}} + \cos \brac[\big]{2 \pi \pair{\mu'''}{\fcwt{2}}} + \cos \brac[\big]{2 \pi \pair{\mu'''}{\fcwt{1}-\fcwt{2}}}} \, \dd [\mu'''] \notag \\
	&\qquad = \delta \brac[\big]{[\mu'' - \mu' - \mu - \tfrac{3}{2} (\fcwt{}'')^*]}
	\brac*{2 \delta_{\fcwt{}'',0} + \delta_{\fcwt{}'',\fcwt{1}} + \delta_{\fcwt{}'',-\fcwt{1}} + \delta_{\fcwt{}'',\fcwt{2}} + \delta_{\fcwt{}'',-\fcwt{2}} + \delta_{\fcwt{}'',\fcwt{3}} + \delta_{\fcwt{}'',-\fcwt{3}}}.
\end{align}
Here, we have set $\fcwt{3} = \fcwt{2}-\fcwt{1}$ for brevity.  If we similarly set $\fwt{3} = \fwt{2}-\fwt{1}$, then the relaxed-by-relaxed Grothendieck fusion rule takes the form
\begin{equation} \label{eq:grfrRR}
	\Gr{\arel{}{[\mu]}} \Grfuse \Gr{\arel{}{[\mu']}} =
	\begin{tikzpicture}[xscale=2,yscale=1.75,baseline=(0)]
		\node (0) at (0:0) {$2 \Gr{\arel{}{[\mu+\mu']}}$};
		\node at (30:1) {$\Gr{\arel{(\fcwt{1})}{[\mu+\mu'+3\fwt{1}/2]}}$};
		\node at (30:-1) {$\Gr{\arel{(-\fcwt{1})}{[\mu+\mu'+3\fwt{1}/2]}}$};
		\node at (90:1) {$\Gr{\arel{(\fcwt{2})}{[\mu+\mu'+3\fwt{2}/2]}}$};
		\node at (90:-1) {$\Gr{\arel{(-\fcwt{2})}{[\mu+\mu'+3\fwt{2}/2]}}$};
		\node at (150:1) {$\Gr{\arel{(\fcwt{3})}{[\mu+\mu'+3\fwt{3}/2]}}$};
		\node at (150:-1) {$\Gr{\arel{(-\fcwt{3})}{[\mu+\mu'+3\fwt{3}/2]}}$};
	\end{tikzpicture}
	,
\end{equation}
where we have introduced a slightly more compact notation $\arel{(\fcwt{})}{[\mu]} = \sfmod{\fcwt{}}{\arel{}{[\mu]}}$ (the parentheses in the superscript are meant to distinguish this notation from $\arel{\lambda}{[\mu]}$), arranged the summands on the \rhs\ according to their spectral flow coweights and noted that $3\fwt{i} \in \rlat$, $i=1,2,3$.  Applying \eqref{eq:grsf} now gives the standard Grothendieck fusion rules for $\themm$.

This convention for arranging summands makes it natural to ask if the appearance of the weights of the adjoint module of $\slthree$ in \eqref{eq:grfrRR} is meaningful.  A similar phenomenon occurs for the relaxed-by-relaxed Grothendieck fusion rules for the $\sltwo$ minimal model $\slnminmod{1}{2}{3}$ of level $-\frac{4}{3}$, see \cite[Eq.~(5.21)]{CreMod12}.  For these minimal models, the general rules \cite{CreMod13} show that other levels do not reproduce the adjoint weights, so we expect that $\themm$ is likewise the only $\slthree$ minimal model with this property.

Nevertheless, the appearance of these weights is striking and we conjecture that this is in fact a ``low-rank coincidence'' resulting from the fact that $\slnminmod{1}{2}{3}$ and $\themm$ coincide with two of the \voas\ $B_Q(p)$ introduced in \cite{CreCos13,CreLog18}.  Here, $Q$ is an ADE-type root lattice and $p \in \ZZ_{\ge2}$.  In fact, one has $\slnminmod{1}{2}{3} \cong B_{A_1}(3)$ \cite{CreCos13} and $\themm \cong B_{A_2}(2)$ \cite{AdaRea14}.  These $B_Q(p)$-algebras are related to (higher-rank) singlet algebras by a parafermionic coset construction \cite{CreCos13,AdaPar20} and thence to triplet algebras.  (This is discussed in more detail in the companion paper \cite{CreKaz21}.)  We mention that the automorphism groups of the corresponding triplet algebras are $\SLG{PSL}{2}(\CC)$ and $\SLG{PSL}{3}(\CC)$, respectively, which may be ultimately responsible for the appearance of adjoint weights in \eqref{eq:grfrRR}.\footnote{We thank an anonymous referee for this suggestion.}

Next, we consider the semirelaxed-by-relaxed Grothendieck fusion rules.  As the $\grp{D}_6$-twist can be chosen to act on the relaxed module (and thereby absorbed), we may restrict to the untwisted result.  Instead of applying the standard Verlinde formula \eqref{eq:svf} directly, we will compute the Grothendieck fusion rule directly by combining the rules \eqref{eq:grfrRR} with
\begin{equation} \label{eq:grS}
	\Gr{\asemi{}{[\mu]}} = \sum_{n=0}^{\infty} \brac*{\Gr{\arel{(-2n\fcwt{2})}{[\mu]}} - \Gr{\arel{(-(2n+1)\fcwt{2})}{[\mu-\sroot{1}/2]}}},
\end{equation}
the latter being a consequence of \cref{lem:atypchS} and the fact that standard characters are linearly independent.  The result benefits from many cancellations:
\begin{align} \label{eq:grfrRS}
	\Gr{\asemi{}{[\mu]}} \Grfuse \Gr{\arel{}{[\mu']}} &= \sum_{n=0}^{\infty} \left(
	\begin{tikzpicture}[xscale=2,yscale=1.75,baseline=(0)]
		\node (0) at (0:0) {$2 \Gr{\arel{(-2n\fcwt{2})}{[\mu+\mu']}}$};
		\node at (30:1) {$\Gr{\arel{(\fcwt{1}-2n\fcwt{2})}{[\mu+\mu'+3\fwt{1}/2]}}$};
		\node at (30:-1) {$\Gr{\arel{(-\fcwt{1}-2n\fcwt{2})}{[\mu+\mu'+3\fwt{1}/2]}}$};
		\node at (90:1) {$\Gr{\arel{(-(2n-1)\fcwt{2})}{[\mu+\mu'+3\fwt{2}/2]}}$};
		\node at (90:-1) {$\Gr{\arel{(-(2n+1)\fcwt{2})}{[\mu+\mu'+3\fwt{2}/2]}}$};
		\node at (150:1) {$\Gr{\arel{(\fcwt{3}-2n\fcwt{2})}{[\mu+\mu'+3\fwt{3}/2]}}$};
		\node at (150:-1) {$\Gr{\arel{(-\fcwt{3}-2n\fcwt{2})}{[\mu+\mu'+3\fwt{3}/2]}}$};
	\end{tikzpicture}
	\ - \
	\begin{tikzpicture}[xscale=2,yscale=1.75,baseline=(0)]
		\node (0) at (0:0) {$2 \Gr{\arel{(-(2n+1)\fcwt{2})}{[\mu+\mu'+3\fwt{2}/2]}}$};
		\node at (30:1) {$\Gr{\arel{(-\fcwt{3}-2n\fcwt{2})}{[\mu+\mu'+3\fwt{3}/2]}}$};
		\node at (30:-1) {$\Gr{\arel{(\fcwt{3}-2(n+1)\fcwt{2})}{[\mu+\mu'+3\fwt{3}/2]}}$};
		\node at (90:1) {$\Gr{\arel{(-2n\fcwt{2})}{[\mu+\mu']}}$};
		\node at (90:-1) {$\Gr{\arel{(-2(n+1)\fcwt{2})}{[\mu+\mu']}}$};
		\node at (150:1) {$\Gr{\arel{(-\fcwt{1}-2n\fcwt{2})}{[\mu+\mu'+3\fwt{1}/2]}}$};
		\node at (150:-1) {$\Gr{\arel{(\fcwt{1}-2(n+1)\fcwt{2})}{[\mu+\mu'+3\fwt{1}/2]}}$};
	\end{tikzpicture}
	\right) \notag \\
	&=
	\begin{tikzpicture}[xscale=2,yscale=1.75,baseline=(0)]
		\node (0) at (90:0.5) {};
		\node at (0:0) {$\Gr{\arel{}{[\mu+\mu']}}$};
		\node at (30:1) {$\Gr{\arel{(\fcwt{1})}{[\mu+\mu'+3\fwt{1}/2]}}$};
		\node at (90:1) {$\Gr{\arel{(\fcwt{2})}{[\mu+\mu'+3\fwt{2}/2]}}$};
		\node at (150:1) {$\Gr{\arel{(\fcwt{3})}{[\mu+\mu'+3\fwt{3}/2]}}$};
	\end{tikzpicture}
	.
\end{align}

The highest-weight-by-relaxed rules follow similarly, though the only interesting rule is the one involving $\airr{-\wvec/2}$ (since $\airr{-3\fwt{1}/2}$ and $\airr{-3\fwt{2}/2}$ are spectral flows of the vacuum module $\airr{0}$).\footnote{That $\airr{0}$ is the Grothendieck fusion unit follows directly from the standard Verlinde formula and the fact that the standard S-matrix is unitary.}  The Grothendieck version of \eqref{eq:Sdeg2} gives
\begin{equation} \label{eq:grL}
	\Gr{\airr{-\wvec/2}} = \Gr{\asemi{}{[-\wvec/2]}} - \Gr{\airr{-3\fwt{1}/2}^{(-2\fcwt{1}+\fcwt{2})}} = \Gr{\asemi{}{[-\wvec/2]}} - \Gr{\airr{0}^{(\fcwt{3})}},
\end{equation}
from which we obtain
\begin{align} \label{eq:grfrLR}
	\Gr{\airr{-\wvec/2}} \Grfuse \Gr{\arel{}{[\mu']}} &= \Gr{\asemi{}{[-\wvec/2]}} \Grfuse \Gr{\arel{}{[\mu']}} - \Gr{\airr{0}^{(\fcwt{3})}} \Grfuse \Gr{\arel{}{[\mu']}} = \Gr{\asemi{}{[-\wvec/2]}} \Grfuse \Gr{\arel{}{[\mu']}} - \Gr{\arel{(\fcwt{3})}{[\mu']}} \notag \\
	&=
	\begin{tikzpicture}[xscale=2,yscale=1.75,baseline=(0)]
		\node (0) at (90:0.5) {};
		\node at (0:0) {$\Gr{\arel{}{[\mu'+3\fwt{3}/2]}}$};
		\node at (30:1) {$\Gr{\arel{(\fcwt{1})}{[\mu'+3\fwt{2}/2]}}$};
		\node at (90:1) {$\Gr{\arel{(\fcwt{2})}{[\mu'+3\fwt{1}/2]}}$};
	\end{tikzpicture}
	=
	\begin{tikzpicture}[xscale=2,yscale=1.75,baseline=(0)]
		\node (0) at (90:0.5) {};
		\node at (0:0) {$\Gr{\arel{}{[\mu'+\sroot{3}/2]}}$};
		\node at (30:1) {$\Gr{\arel{(\fcwt{1})}{[\mu'+\sroot{1}/2]}}$};
		\node at (90:1) {$\Gr{\arel{(\fcwt{2})}{[\mu'+\sroot{2}/2]}}$};
	\end{tikzpicture}
	.
\end{align}
Here, we have used \eqref{eq:grfrRS} and noted that $-\frac{1}{2} \wvec = \frac{3}{2} \fwt{3} \bmod{\rlat}$, while $\frac{3}{2} \fwt{i} = \frac{1}{2} \sroot{\dynkmod{i}}$ for $i=1,2,3$.

The semirelaxed-by-semirelaxed rules require one further insight.  Combining \eqref{eq:grS} and \eqref{eq:grfrRS} gives
\begin{align}
	&\Gr{\asemi{}{[\mu]}} \Grfuse \Gr{\asemi{}{[\mu']}} \notag \\
	&= \sum_{n=0}^{\infty} \left(
	\begin{tikzpicture}[xscale=2,yscale=1.75,baseline=(0)]
		\node (0) at (90:0.5) {};
		\node at (0:0) {$\Gr{\arel{(-2n\fcwt{})}{[\mu+\mu']}}$};
		\node at (30:1) {$\Gr{\arel{(\fcwt{1}-2n\fcwt{})}{[\mu+\mu'+3\fwt{1}/2]}}$};
		\node at (90:1) {$\Gr{\arel{(-(2n-1)\fcwt{2})}{[\mu+\mu'+3\fwt{2}/2]}}$};
		\node at (150:1) {$\Gr{\arel{(\fcwt{3}-2n\fcwt{})}{[\mu+\mu'+3\fwt{3}/2]}}$};
	\end{tikzpicture}
	\ - \
	\begin{tikzpicture}[xscale=2,yscale=1.75,baseline=(0)]
		\node (0) at (90:0.5) {};
		\node at (0:0) {$\Gr{\arel{(-(2n+1)\fcwt{2})}{[\mu+\mu'+3\fwt{2}/2]}}$};
		\node at (30:1) {$\Gr{\arel{(\fcwt{1}-(2n+1)\fcwt{2})}{[\mu+\mu'+3\fwt{1}/2-\sroot{1}/2]}}$};
		\node at (90:1) {$\Gr{\arel{(-2n\fcwt{2})}{[\mu+\mu']}}$};
		\node at (150:1) {$\Gr{\arel{(\fcwt{3}-(2n+1)\fcwt{2})}{[\mu+\mu'+3\fwt{3}/2-\sroot{1}/2]}}$};
	\end{tikzpicture}
	\right),
\end{align}
where we recall that $\frac{3}{2} \fwt{2} = \frac{1}{2} \sroot{1} \bmod{\rlat}$.  The bottom terms on the left cancel with the top terms on the right and a similar near-cancellation of the top-left and bottom-right terms leaves only $\Gr{\arel{(\fcwt{2})}{[\mu+\mu'+3\fwt{2}/2]}}$.  Since
\begin{equation}
	\mu, \mu' \in -\tfrac{3}{2} \fwt{1} + \RR \sroot{1} \quad \Rightarrow \quad
	\mu + \mu' + \tfrac{3}{2} \fwt{2} + \sroot{} \notin -\tfrac{3}{2} \fwt{1} + \RR \sroot{1},
\end{equation}
for any $\sroot{} \in \rlat$, this standard module does not degenerate.  The remaining terms however combine to give spectral flows of semirelaxed modules because $\mu + \mu' + \frac{3}{2} \fwt{1}$ and $\mu + \mu' + \frac{3}{2} \fwt{3}$ do belong to $-\frac{3}{2} \fwt{1} + \RR \sroot{1}$.  The Grothendieck fusion rule is thus
\begin{equation} \label{eq:grfrSS}
	\Gr{\asemi{}{[\mu]}} \Grfuse \Gr{\asemi{}{[\mu']}} =
	\begin{tikzpicture}[xscale=2,yscale=1.75,baseline=(0)]
		\node (0) at (90:0.67) {};
		\node at (30:1) {$\Gr{\asemi{(\fcwt{1})}{[\mu+\mu'+3\fwt{1}/2]}}$};
		\node at (90:1) {$\Gr{\arel{(\fcwt{2})}{[\mu+\mu'+3\fwt{2}/2]}}$};
		\node at (150:1) {$\Gr{\asemi{(\fcwt{3})}{[\mu+\mu'+3\fwt{3}/2]}}$};
	\end{tikzpicture}
	.
\end{equation}

The highest-weight-by-semirelaxed rule is now easily deduced from \eqref{eq:grL} and \eqref{eq:grfrSS}:
\begin{equation} \label{eq:grfrLS}
	\Gr{\airr{-\wvec/2}} \Grfuse \Gr{\asemi{}{[\mu']}} =
	\begin{tikzpicture}[xscale=2,yscale=1.75,baseline=(0)]
		\node (0) at (90:0.67) {};
		\node at (30:1) {$\Gr{\asemi{(\fcwt{1})}{[\mu'+\sroot{1}/2]}}$};
		\node at (90:1) {$\Gr{\arel{(\fcwt{2})}{[\mu'+\sroot{2}/2]}}$};
	\end{tikzpicture}
	.
\end{equation}
This, along with the spectral flow identifications \eqref{eq:sfident32} and the degeneration formulae \eqref{eq:degS32} and \eqref{eq:degRS32}, implies the highest-weight-by-highest-weight rule:
\begin{align}
	\Gr{\airr{-\wvec/2}} \Grfuse \Gr{\airr{-\wvec/2}}
	&= \Gr{\airr{-\wvec/2}} \Grfuse \Gr{\asemi{}{-\wvec/2}} - \Gr{\airr{\wvec/2}^{(\fcwt{3})}}
	= \Gr{\asemi{(\fcwt{1})}{-3\fwt{1}/2}} + \Gr{\arel{(\fcwt{2})}{-3\fwt{2}/2}} - \Gr{\airr{\wvec/2}^{(\fcwt{3})}} \notag \\
	&= \Gr{\sfmod{\fcwt{2}}{\airr{-3\fwt{2}/2}}} + \Gr{\sfaut{\fcwt{2}} \conjmod{\airr{-3\fwt{2}/2}}} + \Gr{\sfaut{\fcwt{2}} \weylmod{2}{\airr{-\wvec/2}}} + \Gr{\sfaut{\fcwt{2}} \conjaut \weylmod{2}{\airr{-\wvec/2}}} \notag \\
	&\qquad + \Gr{\sfmod{\fcwt{1}}{\airr{-3\fwt{1}/2}}} + \Gr{\sfaut{\fcwt{1}} \weylmod{2}{\airr{-\wvec/2}}} - \Gr{\airr{\wvec/2}^{(\fcwt{3})}} \notag \\
	&= \Gr{\airr{-3\fwt{2}/2}^{(\fcwt{2})}} + \Gr{\airr{0}} + \Gr{\airr{-3\fwt{1}/2}^{(\fcwt{1})}} + \Gr{\sfaut{\fcwt{1}} \weylmod{1}{\airr{-\wvec/2}}} + \Gr{\sfaut{\fcwt{2}} \weylmod{2}{\airr{-\wvec/2}}}.
\end{align}
Arranging the summands by spectral flow (with conjugation corresponding to negation) and simplifying, the Grothendieck fusion rule is
\begin{equation} \label{eq:grfrLL}
	\Gr{\airr{-\wvec/2}} \Grfuse \Gr{\airr{-\wvec/2}} =
	\begin{tikzpicture}[scale=1.5,baseline=(0)]
		\node at (0:0) {$\Gr{\airr{0}}$};
		\node (0) at (30:2) {$\Gr{\airr{0}^{(2\fcwt{1})}}$};
		\node at (90:2) {$\Gr{\airr{0}^{(2\fcwt{2})}}$};
		\node at (60:1.732) {$2\Gr{\conjmod{\airr{-\wvec/2}^{(-\fcwt{1}-\fcwt{2})}}}$};
	\end{tikzpicture}
	.
\end{equation}

The Grothendieck fusion rules presented above are not quite complete because there is no relaxed module on the \lhs\ of \eqref{eq:grfrSS}, \eqref{eq:grfrLS} or \eqref{eq:grfrLL} to absorb a $\grp{D}_6$-twist.  As per \eqref{eq:grD6}, one of the modules in the semirelaxed-by-semirelaxed rule \eqref{eq:grfrSS} should be twisted by $\omega \in \grp{D}_6$.  However, $\weylaut{1}$-twists are dealt with by \eqref{eq:Wident32} and $\conjaut$-twists by \eqref{eq:sfident32}, leaving only $\weylaut{2}$ and $\weylaut{1} \weylaut{2}$.  Using \eqref{eq:atypchSD6} and \eqref{eq:grfrRS}, we deduce that
\begin{equation} \label{eq:grfrSSD6}
	\begin{aligned}
		\Gr{\asemi{}{[\mu]}} \Grfuse \Gr{\weylmod{2}{\asemi{}{[\mu']}}}
		&= \Gr{\arel{}{[\mu+\weylmod{2}{\mu'}]}} + \Gr{\arel{(\fcwt{1})}{[\mu+\weylmod{2}{\mu'}+3\fwt{1}/2]}}, \\
		\Gr{\asemi{}{[\mu]}} \Grfuse \Gr{\weylaut{1} \weylmod{2}{\asemi{}{[\mu']}}}
		&= \Gr{\arel{}{[\mu+\weylaut{1}\weylmod{2}{\mu'}]}} + \Gr{\arel{(\fcwt{3})}{[\mu+\weylaut{1}\weylmod{2}{\mu'}+3\fwt{3}/2]}}.
	\end{aligned}
\end{equation}
The highest-weight-by-semirelaxed rule \eqref{eq:grfrLS} is actually sufficiently general because \eqref{eq:ident32} implies that
\begin{equation}
	\weylmod{2}{\airr{-\wvec/2}}
	\cong \sfaut{\fcwt{1}} \weylmod{3}{\airr{-\wvec/2}}
	\cong \sfaut{\fcwt{1}} \conjmod{\airr{-\wvec/2}}
	\cong \conjaut \sfmod{-\fcwt{1}}{\airr{-\wvec/2}},
\end{equation}
hence \eqref{eq:grD6} and \eqref{eq:grsf} give
\begin{align}
	\Gr{\airr{-\wvec/2}} \Grfuse \Gr{\weylmod{2}{\asemi{}{[\mu']}}}
	&= \weylmod{2}{\Gr{\weylmod{2}{\airr{-\wvec/2}}} \Grfuse \Gr{\asemi{}{[\mu']}}}
	= \weylaut{2} \conjmod{\Gr{\sfmod{-\fcwt{1}}{\airr{-\wvec/2}}} \Grfuse \Gr{\conjmod{\asemi{}{[\mu']}}}} \notag \\
	&= \weylaut{2} \conjaut \sfmod{-\fcwt{1}-\fcwt{2}}{\Gr{\airr{-\wvec/2}} \Grfuse \Gr{\asemi{}{[-\mu'-3\fwt{2}/2]}}}
\end{align}
(and a similar computation gives the result when $\weylaut{2}$ is replaced by $\weylaut{1} \weylaut{2}$).

It remains to generalise the highest-weight-by-highest-weight Grothendieck fusion rule \eqref{eq:grfrLL} by adding a twist.  In this case, it is easy to see from \eqref{eq:sfident32} that it suffices to give the result when the twist is any one of the reflections in $\grp{D}_6$ except $\dynkaut$.  We shall therefore compute
\begin{align} \label{eq:grfrLLD6}
	\Gr{\airr{-\wvec/2}} \Grfuse \Gr{\conjmod{\airr{-\wvec/2}}}
	&= \Gr{\airr{-\wvec/2}} \Grfuse \Gr{\conjmod{\asemi{}{[-\wvec/2]}}} - \Gr{\airr{-\wvec/2}} \Grfuse \Gr{\conjaut \sfmod{\fcwt{3}}{\airr{0}}} \notag \\
	&= \Gr{\airr{-\wvec/2}} \Grfuse \Gr{\asemi{(-\fcwt{2})}{[-3\fwt{1}/2]}} - \Gr{\airr{-\wvec/2}^{(-\fcwt{3})}}
	= \Gr{\arel{}{[0]}} + \Gr{\asemi{(-\fcwt{3})}{[-\wvec/2]}} - \Gr{\airr{-\wvec/2}^{(-\fcwt{3})}} \notag \\
	&= \Gr{\arel{}{[0]}} + \Gr{\sfaut{-\fcwt{3}} \weylmod{1}{\airr{-3\fwt{1}/2}}} = \Gr{\arel{}{[0]}} + \Gr{\airr{0}},
\end{align}
noting with some satisfaction that the vacuum module indeed appears in the Grothendieck fusion rule for this irreducible and its conjugate, as expected.

For completeness, we present the Grothendieck fusion rules for the other irreducibles by their conjugates.  The semirelaxed rule is easily seen to be
\begin{equation}
	\Gr{\asemi{}{[\mu]}} \Grfuse \Gr{\conjmod{\asemi{}{[\mu]}}} = \Gr{\arel{}{[0]}} + 2 \Gr{\airr{0}} + \Gr{\airr{-\wvec/2}^{(-\fcwt{3})}} + \Gr{\conjmod{\airr{-\wvec/2}^{(-\fcwt{3})}}}.
\end{equation}
Its relaxed counterpart is unsurprisingly more complicated:
\begin{align}
	\Gr{\arel{}{[\mu]}} \Grfuse &\Gr{\conjmod{\arel{}{[\mu]}}}
	= 2 \Gr{\arel{}{[0]}} + 6 \Gr{\airr{0}} + \Gr{\airr{0}^{(2\fcwt{1})}} + \Gr{\airr{0}^{(-2\fcwt{1})}} + \Gr{\airr{0}^{(2\fcwt{2})}} + \Gr{\airr{0}^{(-2\fcwt{2})}} + \Gr{\airr{0}^{(2\fcwt{3})}} + \Gr{\airr{0}^{(-2\fcwt{3})}} \notag \\
	&+ 2 \brac*{\Gr{\airr{-\wvec/2}^{(\fcwt{1}-\fcwt{2})}} + \Gr{\airr{-\wvec/2}^{(\fcwt{2}-\fcwt{1})}} + \Gr{\airr{-\wvec/2}^{(-\fcwt{1}-\fcwt{2})}}} + 2 \conjaut \brac*{\Gr{\airr{-\wvec/2}^{(\fcwt{1}-\fcwt{2})}} + \Gr{\airr{-\wvec/2}^{(\fcwt{2}-\fcwt{1})}} + \Gr{\airr{-\wvec/2}^{(-\fcwt{1}-\fcwt{2})}}}.
\end{align}
In both cases, the vacuum module appears with multiplicity greater than $1$.  This indicates that the corresponding genuine fusion products (meaning not the Grothendieck ones) are not completely reducible.

We shall not pursue the structures of these reducible but indecomposable $\themm$-module here, though their existence is clear.  Instead, we refer to a companion paper \cite{CreKaz21} in which conjectures for these structures are presented.  These indecomposables also include modules that we conjecture are the projective covers and injective hulls of the irreducible modules in $\awcat{3,2}$.  The composition factors of all these indecomposables are of course determined by the Grothendieck fusion rules calculated here.

A closely related remark is that the Grothendieck fusion ring presented here has a one-dimensional representation which is constant on the $\grp{D}_6$- and spectral flow orbits, but otherwise takes the values
\begin{equation}
	\Gr{\arel{}{[\mu]}} \mapsto 8, \quad \Gr{\asemi{}{[\mu]}} \mapsto 4, \quad \Gr{\airr{-\wvec/2}} \mapsto 3 \quad \text{and} \quad \Gr{\airr{0}} \mapsto 1.
\end{equation}
This is also addressed in the companion paper \cite{CreKaz21} where the existence of a \kl-type correspondence is discussed.  This correspondence takes the form of a (conjectural) tensor equivalence between the category $\awcat{3,2}$ of weight $\themm$-modules (with \fdim\ weight spaces) and an appropriate modification of the category of \fdim\ modules over a certain quantum group $\overline{U}{}^H_q(\slthree)$ at $q=\ii$.  The values of the above representation are then just the dimensions of the corresponding irreducible quantum group modules.



\flushleft
\providecommand{\opp}[2]{\textsf{arXiv:\mbox{#2}/#1}}
\providecommand{\pp}[2]{\textsf{arXiv:#1 [\mbox{#2}]}}

\end{document}